\documentclass[12pt]{amsart}
\usepackage[utf8]{inputenc}
\usepackage[margin=1in]{geometry}
\usepackage[english]{babel}
\usepackage[shortlabels]{enumitem}
\usepackage[colorinlistoftodos]{todonotes}
\usepackage[colorlinks=true, allcolors=blue]{hyperref}
\usepackage{graphicx, fancyhdr, amsmath, amssymb, amsthm, tikz, forest, pdfpages, pgfkeys, multicol, yfonts, pgf, pgfplots, mathrsfs, fixmath, xcolor, soul, cancel, wasysym, harpoon, scrextend}
\usepackage{arydshln}
\usepackage{tikz-cd}

\usetikzlibrary{arrows,calc,through,backgrounds,matrix}

\theoremstyle{plain} 
\newtheorem{theorem}{Theorem}[section] 
\newtheorem{lemma}[theorem]{Lemma}
\newtheorem{proposition}[theorem]{Proposition}
\newtheorem{corollary}[theorem]{Corollary}
\newtheorem{definition}[theorem]{Definition}
\newtheorem{example}[theorem]{Example}
\theoremstyle{remark}
\newtheorem{remark}[theorem]{Remark}

\newcommand{\cellX}[1]{\mathcal{C}_{#1}^X}

\newcommand{\match}[1]{\mathcal{#1}}
\newcommand{\anc}[2]{anc^{#2}(#1)}
\newcommand{\init}{\textup{init}}
\newcommand{\term}{\textup{term}}
\newcommand{\jbeg}{\mathcal{J}_{beg}}
\newcommand{\jend}{\mathcal{J}_{end}}
\newcommand{\jnot}{\mathcal{J}_{not}}

\title{Cell Closures for Two-Row Springer Fibers via Noncrossing Matchings}
\author{Talia Goldwasser \and Meera Nadeem \and Garcia Sun \and Julianna Tymoczko}

\begin{document}

\begin{abstract}
   Springer fibers are a family of subvarieties of the flag variety parametrized by nilpotent matrices that are important in geometric representation theory and whose geometry encodes deep combinatorics.  Two-row Springer fibers, which correspond to nilpotent matrices with two Jordan blocks, also arise in knot theory, in part because their components are indexed by noncrossing matchings.  Springer fibers are paved by affines by intersection with appropriately-chosen Schubert cells.  In the two-row case, we provide an elementary description of these \emph{Springer Schubert} cells in terms of \emph{standard} noncrossing matchings and describe closure relations  explicitly.  To do this, we define an operation called \emph{cutting arcs} in a matching, which successively unnests arcs while ``remembering" the arc originally on top.  We then prove that the boundary of the Springer Schubert cell corresponding to a matching consists of the affine subsets of cells corresponding to all ways of cutting arcs in that matching.   
\end{abstract}
\maketitle

\section{Introduction}


The Springer fiber of a linear operator $X$ is the fixed-point set of $X$ in the full flag variety (see below for a formal definition).  Discovered by Tonny Springer in 1976, Springer fibers are the seminal example of a geometric representation: geometric objects whose cohomology carries a symmetric group representation \cite{Spr76}.  Their geometry and topology have been extensively studied as a tool to analyze the underlying representations (e.g. proving geometric properties of Springer fibers in order to apply certain theorems to study the representation) \cite{Spa76, Spa77, Shi80, Shi85}), in the broader context of modern Schubert calculus \cite{BayHar12, DewHar12}, and because of intrinsically-interesting ways their geometry encodes relations between the combinatorics of partitions and permutations \cite{Fre09, Fre10, Fre11, Fre12, Fre16, FreMel10, FreMel11, FreMel13}.

It is well-known that Springer fibers can be paved by affines \cite{Spa76, DLP88}, and that in fact these affines can be chosen to be the intersection with Schubert cells \cite{Tym05, Fre16}. Compared to more familiar CW-decompositions, where the boundary of a cell is the union of smaller-dimensional cells, the subtlety of a \emph{paving} is that the boundary may consist of proper subsets of other cells, possibly even subsets of cells of the same dimension as the original.  This is part of why the geometry and topology of Springer fibers is not inherited from that of Schubert varieties, and in fact is more complicated than that of Schubert varieties.  Indeed, even determining whether the components of Springer fibers are singular is quite challenging \cite{Fre10, Fre11, FreMel10, FreMel11}, relying on both the combinatorics of permutations (from the Schubert cells) and of partitions (from the Jordan type of the linear operator $X$ defining the Springer fiber).

In this paper, we characterize closures of cells completely for two-row Springer fibers, namely the case when $X$ is nilpotent and has at most two Jordan blocks.  To do this, we use a different combinatorial representation of the cells than the more-commonly-used row-strict tableaux: a kind of noncrossing matching called \emph{standard}.  When we do this, the closures are naturally described as the union of certain \emph{labeled} standard noncrossing matchings.

More precisely, a \emph{noncrossing matching} is drawn as a set of noncrossing arcs above the $x$-axis that start and end at distinct integers within $\{1, 2, \ldots, N\}$.  A noncrossing matching is \emph{standard} if in addition all integers $\{i+1,\ldots, j-1\}$ under each arc $(i,j)$ are themselves on arcs.  It is well-known that when $X$ is nilpotent with two Jordan blocks of sizes $(n,N-n)$, then the Betti numbers of the Springer fiber $\mathcal{S}_X$ are bijective with standard noncrossing matchings with at most $N-n$ arc starts and $n$ arc ends \cite{Fun03, Kho04b, Var79, RusTym11} and indeed with many other combinatorial objects that are bijective with this set. Theorem~\ref{theorem: fM is a bijection onto Springer Schubert cell} shows that nesting within the standard noncrossing matchings provides a direct and explicit bijection to the cells in a paving by affines, and that the paving is the intersection with a particular set of Schubert cells.  For this reason, we call them \emph{Springer Schubert cells}.  Figure~\ref{figure: all closures n=4} shows all six Springer Schubert cells for the Springer fiber of type $(2,2)$ (and two affine subspaces of Springer Schubert cells); the free entries of each cell in the paving are just the labels on the arcs read from bottom to top.  

Furthermore, we define an iterated graph-theoretic unnesting that we call \emph{cutting arcs} in a matching.  Cutting an arc $(i,j)$ that is immediately below another arc $(i',j')$ creates two arcs $(i',i)$ and $(j,j')$ both of which are associated to the same free variable as $(i',j')$ in their Schubert cell.  Thus, cutting $k$ arcs in the matching $\match{M}$ identifies an $(|\match{M}|-k)$-dimensional subspace of the Springer Schubert cell $\cellX{\match{M}'}$ associated to another matching $\match{M}'$.  Figure~\ref{figure: all closures n=4} shows the closures of the two top-dimensional cells in the Springer fiber of Jordan type $(2,2)$.  Even though this is a small example, both the bottom cell in the left-hand closure and the right cell in the right-hand closure are strict subspaces of their corresponding Schubert cells.

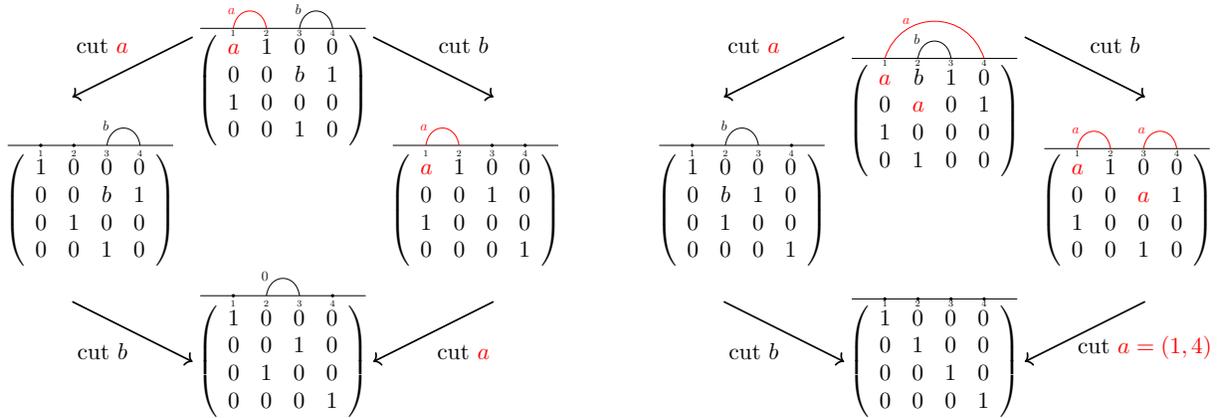
\begin{figure}[h]

\scalebox{0.8}{\begin{tikzpicture}
\draw(0,2.5) node[anchor=south] {\scalebox{0.55}{\begin{tikzpicture}
        \draw (0,0) -- (5,0);
         \draw[red] (1,0) .. controls (1.1,.7) and (1.9,.7) .. (2,0);
    \draw (3,0) .. controls (3.1,.7) and (3.9,.7) .. (4,0);
        \draw[red] (1.2,.3) node[anchor=south east] {$a$};
    \draw (3.2,.3) node[anchor=south east] {$b$};
        \draw (1,0.1) node[anchor=north] {\tiny $1$};
    \draw (2,0.1) node[anchor=north] {\tiny $2$};
    \draw (3,0.1) node[anchor=north] {\tiny $3$};
    \draw (4,0.1) node[anchor=north] {\tiny $4$};
\end{tikzpicture}}};
\draw(0,0.75) node[anchor=south] {\small $\left( \begin{array}{cccc} {\color{red} a} & 1 & 0 & 0 \\ 0 & 0 & b & 1 \\ 1 & 0 & 0 & 0 \\ 0 & 0 & 1 & 0 \end{array} \right) $};

\draw(-3.2,0.5) node[anchor=south] {\scalebox{0.55}{\begin{tikzpicture}
                \draw (0,0) -- (5,0);
                \filldraw[black] (1,0) circle (1pt);
    \filldraw[black] (2,0) circle (1pt);
    \draw (3,0) .. controls (3.1,.7) and (3.9,.7) .. (4,0);
        \draw (3.2,.3) node[anchor=south east] {$b$};
        \draw (1,0) node[anchor=north] {\tiny $1$};
    \draw (2,0) node[anchor=north] {\tiny $2$};
    \draw (3,0) node[anchor=north] {\tiny $3$};
    \draw (4,0) node[anchor=north] {\tiny $4$};
\end{tikzpicture}}};
\draw(-3.2,-1.25) node[anchor=south] {\small $\left( \begin{array}{cccc} 1 & 0 & 0 & 0 \\ 0 & 0 & b & 1 \\ 0 & 1 & 0 & 0 \\ 0 & 0 & 1 & 0 \end{array} \right) $};

\draw(3.2,0.5) node[anchor=south] {\scalebox{0.55}{\begin{tikzpicture}
    \draw (0,0) -- (5,0);
                \filldraw[black] (3,0) circle (1pt);
    \filldraw[black] (4,0) circle (1pt);
    \draw[red] (1,0) .. controls (1.1,.7) and (1.9,.7) .. (2,0);
        \draw[red] (1.2,.3) node[anchor=south east] {$a$};
        \draw (1,0) node[anchor=north] {\tiny $1$};
    \draw (2,0) node[anchor=north] {\tiny $2$};
    \draw (3,0) node[anchor=north] {\tiny $3$};
    \draw (4,0) node[anchor=north] {\tiny $4$};
\end{tikzpicture}}};
\draw(3.2,-1.25) node[anchor=south] {\small $\left( \begin{array}{cccc} {\color{red} a} & 1 & 0 & 0 \\ 0 & 0 & 1 & 0 \\ 1 & 0 & 0 & 0 \\ 0 & 0 & 0 & 1 \end{array} \right) $};

\draw(0,-2) node[anchor=south] {\scalebox{0.55}{\begin{tikzpicture}
      \draw (0,0) -- (5,0);
        \filldraw[black] (1,0) circle (1pt);
    \filldraw[black] (4,0) circle (1pt);
    \draw (2,0) .. controls (2.1,.7) and (2.9,.7) .. (3,0);
        \draw (2.2,.3) node[anchor=south east] {$0$};
    \draw (1,0) node[anchor=north] {\tiny $1$};
    \draw (2,0) node[anchor=north] {\tiny $2$};
    \draw (3,0) node[anchor=north] {\tiny $3$};
    \draw (4,0) node[anchor=north] {\tiny $4$};
\end{tikzpicture}}};
\draw(0,-3.75) node[anchor=south] {\small $\left( \begin{array}{cccc} 1 & 0 & 0 & 0 \\ 0 & 0 & 1 & 0 \\ 0 & 1 & 0 & 0 \\ 0 & 0 & 0 & 1 \end{array} \right) $};

\draw[thick,->] (-1.5,2.7) -- (-3.5,1.7);
\draw (-3,2.3) node[anchor=south] {\small cut {\color{red} $a$}};
\draw[thick,<-] (-1.5,-2.7) -- (-3.5,-1.7);
\draw (-3,-2.8) node[anchor=south] {\small cut $b$}; 
\draw[thick,->] (1.5,2.7) -- (3.5,1.7);
\draw (3,2.3) node[anchor=south] {\small cut $b$};
\draw[thick,<-] (1.5,-2.7) -- (3.5,-1.7);
\draw (3,-2.8) node[anchor=south] {\small cut {\color{red} $a$}};
\end{tikzpicture}} \hspace{0.3in}
\scalebox{0.8}{\begin{tikzpicture}
\draw(0,2) node[anchor=south] {\scalebox{0.55}{\begin{tikzpicture}
    \draw (0,0) -- (5,0);
      \draw[red] (1,0) .. controls (1.5,1.5) and (3.5,1.5) .. (4,0);
    \draw (2,0) .. controls (2.1,.7) and (2.9,.7) .. (3,0);
        \draw[red] (1.9,.9) node[anchor=south east] {$a$};
    \draw (2.2,.3) node[anchor=south east] {$b$};
    \draw (1,0.1) node[anchor=north] {\tiny $1$};
    \draw (2,0.1) node[anchor=north] {\tiny $2$};
    \draw (3,0.1) node[anchor=north] {\tiny $3$};
    \draw (4,0.1) node[anchor=north] {\tiny $4$};
\end{tikzpicture}}};
\draw(0,0.25) node[anchor=south] {\small $\left( \begin{array}{cccc} {\color{red} a} & b & 1 & 0 \\ 0 & {\color{red} a} & 0 & 1 \\ 1 & 0 & 0 & 0 \\ 0 & 1 & 0 & 0 \end{array} \right) $};

\draw(-3.2,0.5) node[anchor=south] {\scalebox{0.55}{\begin{tikzpicture}
    \draw (0,0) -- (5,0);
        \filldraw[black] (1,0) circle (1pt);
    \filldraw[black] (4,0) circle (1pt);
    \draw (2,0) .. controls (2.1,.7) and (2.9,.7) .. (3,0);
        \draw (2.2,.3) node[anchor=south east] {$b$};
    \draw (1,0) node[anchor=north] {\tiny $1$};
    \draw (2,0) node[anchor=north] {\tiny $2$};
    \draw (3,0) node[anchor=north] {\tiny $3$};
    \draw (4,0) node[anchor=north] {\tiny $4$};
\end{tikzpicture}}};
\draw(-3.2,-1.25) node[anchor=south] {\small $\left( \begin{array}{cccc} 1 & 0 & 0 & 0 \\ 0 & b & 1 & 0 \\ 0 & 1 & 0 & 0 \\ 0 & 0 & 0 & 1 \end{array} \right) $};

\draw(3.2,0.5) node[anchor=south] {\scalebox{0.55}{\begin{tikzpicture}
        \draw (0,0) -- (5,0);
         \draw[red] (1,0) .. controls (1.1,.7) and (1.9,.7) .. (2,0);
    \draw[red] (3,0) .. controls (3.1,.7) and (3.9,.7) .. (4,0);
        \draw[red] (1.2,.3) node[anchor=south east] {$a$};
    \draw[red] (3.2,.3) node[anchor=south east] {$a$};
        \draw (1,0.1) node[anchor=north] {\tiny $1$};
    \draw (2,0.1) node[anchor=north] {\tiny $2$};
    \draw (3,0.1) node[anchor=north] {\tiny $3$};
    \draw (4,0.1) node[anchor=north] {\tiny $4$};
\end{tikzpicture}}};
\draw(3.2,-1.25) node[anchor=south] {\small $\left( \begin{array}{cccc} {\color{red} a} & 1 & 0 & 0 \\ 0 & 0 & {\color{red} a} & 1 \\ 1 & 0 & 0 & 0 \\ 0 & 0 & 1 & 0 \end{array} \right) $};

\draw(0,-2) node[anchor=south] {\scalebox{0.55}{\begin{tikzpicture}
        \draw (0,0) -- (5,0);
            \filldraw[black] (1,0) circle (1pt);
      \filldraw[black] (2,0) circle (1pt);
      \filldraw[black] (3,0) circle (1pt);
      \filldraw[black] (4,0) circle (1pt);
        \draw (1,0.1) node[anchor=north] {\tiny $1$};
    \draw (2,0.1) node[anchor=north] {\tiny $2$};
    \draw (3,0.1) node[anchor=north] {\tiny $3$};
    \draw (4,0.1) node[anchor=north] {\tiny $4$};
\end{tikzpicture}}};
\draw(0,-3.75) node[anchor=south] {\small $\left( \begin{array}{cccc} 1 & 0 & 0 & 0 \\ 0 & 1 & 0 & 0 \\ 0 & 0 & 1 & 0 \\ 0 & 0 & 0 & 1 \end{array} \right) $};

\draw[thick,->] (-1.5,2.7) -- (-3.5,1.7);
\draw (-3,2.3) node[anchor=south] {\small cut {\color{red} $a$}};
\draw[thick,<-] (-1.5,-2.7) -- (-3.5,-1.7);
\draw (-3,-2.8) node[anchor=south] {\small cut $b$}; 
\draw[thick,->] (1.5,2.7) -- (3.5,1.7);
\draw (3,2.3) node[anchor=south] {\small cut $b$};
\draw[thick,<-] (1.5,-2.7) -- (3.5,-1.7);
\draw (3.5,-2.8) node[anchor=south] {\small cut {\color{red} $a = (1,4)$}};
\end{tikzpicture}}
\caption{The closures of the two top-dimensional cells for the Springer fiber of Jordan type $(2,2)$, with colors added for clarity} \label{figure: all closures n=4}
\end{figure}

Our main theorem, Theorem~\ref{theorem: closure main theorem}, says:

\medskip 

\noindent {\bf Theorem 6.9.}  \emph{Suppose that $\match{M}$ is a standard noncrossing matching with associated Springer Schubert cell $\cellX{\match{M}}$.  The closure $\overline{\cellX{\match{M}}}$ is the disjoint union
\[ \overline{\cellX{\match{M}}} = \bigcup_{\mathcal{A} \subseteq \match{M}} cut(\cellX{\match{M}}, \mathcal{A})\]
where $cut(\cellX{\match{M}},\mathcal{A})$ denotes the subspace of the Springer Schubert cell associated to the matching obtained by cutting the arcs $\mathcal{A}$ inside the matching $\match{M}$.}

\medskip 

Unnesting induces one of the standard partial orders on noncrossing matchings. This partial order is important in both the associated representation theory of the symmetric group (through Springer's representation) and of quantum $\mathfrak{sl}_2$ (through the web representation) \cite{Kup96, RusTym17}.  Recall that one key goal of \emph{modern Schubert calculus} is to describe geometry of flag varieties and their subvarieties in terms of natural combinatorial and representation-theoretic data., e.g. different properties of permutations determine closures and singularities within Schubert varieties. Thus Theorem~\ref{theorem: closure main theorem} can be considered a Springer example of modern Schubert calculus.
    
Moreover, data suggests our results can be generalized from noncrossing matchings to a larger family of graphs called \emph{webs}.  Spiders are a diagrammatic category encoding representations of $U_q(\mathfrak{sl}_n)$.  Webs index the functors in this category; for instance, noncrossing matchings are the (reduced) webs for $\mathfrak{sl}_2$.  Webs for $n \geq 3$ are significantly more complicated than noncrossing matchings, generally containing many interior trivalent vertices.  But webs for $\mathfrak{sl}_3$ also admit an intrinsic graph-theoretic notion of \emph{depth} that is intimately tied to the underlying representation theory \cite{KhoKup99, Tym12}.  This idea of depth associates a noncrossing matching to each $\mathfrak{sl}_3$-web for which the associated nesting/unnesting partial order has similar properties as its analogue for noncrossing matchings \cite{RusTym20}.  For this reason, we believe that our main theorem can be extended to larger $n$.

Finally, this theorem demonstrates the principle that different combinatorial representations of the same set can illuminate vastly different properties of the objects in that set.  For instance, it is more customary to index cells of the paving by affines with row-strict fillings of the Young diagram of shape $(N-n,n)$ \cite{RusTym11, Tym05}.  These tableaux illuminate linear algebraic properties of the matrix $X$, especially its Jordan blocks and the order in which basis vectors are chosen within the Jordan blocks.  But the geometry---which is depicted quite naturally by nesting/unnesting of arcs in the matching---is opaque using tableaux.

In the rest of this paper we do the following.  Section~\ref{section: background flags and springer} formally defines the geometric and linear-algebraic background, especially flags, Schubert cells, and Springer fibers.  Section~\ref{section: sncm and perms, def and first results} then gives combinatorial background, especially definitions and first results on noncrossing matchings.  Section~\ref{section: bijection matchings to cells} states and proves one of the two main results: an explicit bijection between standard noncrossing matchings and Springer Schubert cells for two-block Springer fibers.  Section~\ref{section: closures are indexed by arc cuts} describes some geometric properties that closures of Springer Schubert cells must satisfy, defines the process of cutting arcs in a matching, and proves that cutting arcs gives an upper bound on the closure of a Springer Schubert cell.  Finally, Section~\ref{section: closure section} uses induction to prove the main theorem, explicitly identifying the closure of a Springer Schubert cell in terms of cutting arcs in its matching.


\subsection{Acknowedgements} The second author was partially supported by the Smith College Center for Women in Mathematics. The last author was partially supported by NSF grants DMS-2054513 and DMS-1800773, and through an AWM MERP grant.  The authors gratefully acknowledge particularly useful conversations with Saraphina Forman, Rebecca Goldin, Heather Russell, Risa Vandergrift, Alia Tinsley, and Alex Woo, as well as many other students who have participated in the Springer fiber/web research group.

\section{Flags and Springer fibers} \label{section: background flags and springer}
We begin by describing background and notation for flags and the flag variety, its decomposition into Schubert cells, the subvarieties called Springer fibers, and the Springer Schubert cells.

The flag variety is the quotient $GL_{N}(\mathbb{C})/B$ of $N \times N$ invertible matrices by the Borel subgroup $B$ of upper-triangular invertible matrices.  Each flag can also be thought of as a sequence of nested subspaces $V_{\bullet}$ with
\[V_{\bullet} = V_1 \subseteq V_2 \subseteq \cdots V_{2n-1} \subseteq V_{N} = \mathbb{C}^{N}\]
where each subspace $V_i$ is $i$-dimensional.  

Going between these two descriptions is straightforward.  Given a flag $gB$, the first $i$ columns are linearly independent because $g$ is invertible, so span an $i$-dimensional subspace $V_i$. Varying $i$ gives a set $V_{\bullet}$ of nested subspaces. Conversely, given $V_{\bullet}$ we recover a matrix by successively choosing any vector in $V_i$ that's not in $V_{i-1}$ to be the $i^{th}$ column of $g$.  Together, the $N$ columns span $\mathbb{C}^{N}$ by construction so $g \in GL_{N}(\mathbb{C})$.  A small exercise proves all choices of $g$ differ by right-multiplication by a matrix in $B$ (see, e.g., \cite{Ful97} for more).

We can choose a unique coset representative for each $gB$ and use these coset representatives to partition the flag variety into Schubert cells.  For convenience, we give the following definition and facts; more details can be found in e.g. \cite{Ful97}.

\begin{definition} \label{definition: canonical representative and Schub cells}
Each flag $gB$ has a unique coset representative $g$ satisfying the following properties:
\begin{itemize}
    \item each column and row has a unique \emph{pivot} entry, which is $1$; 
    \item all entries below the pivot are zero; and
    \item all entries to the right of the pivot are zero.
\end{itemize}
We call this coset representative the \emph{canonical form} of $gB$.  

The pivot entries form a permutation matrix $w$.  The collection of all flags whose canonical form has permutation matrix $w$ is denoted $\mathcal{C}_w$ and is called a \emph{Schubert cell}.  The Schubert cells partition the flag variety, and in fact form a CW-decomposition of the flag variety.
\end{definition}

We use $\vec{e}_1, \vec{e}_2, \ldots, \vec{e}_{N}$ to denote the standard basis vectors in $\mathbb{C}^{N}$, namely the vectors that are zero in all but one entry, where they are $1$.

We now define Springer fibers and establish the notation we use in this paper.

\begin{definition} \label{definition: springer fiber and cell}
Given a $N \times N$ matrix $X$, the Springer fiber of $X$ is the subvariety of flags 
\[\begin{array}{ll} \mathcal{S}_X &= \{V_{\bullet}: XV_i \subseteq V_i \textup{ for all } i\} \\
&= \{gB: g^{-1}Xg \textup{ is upper-triangular}\} \end{array}\]
The Springer Schubert cell $\cellX{w}$ is the intersection
\[\mathcal{S}_X \cap \mathcal{C}_w.\]
In this paper, we assume that $X$ has two Jordan blocks and is in Jordan canonical form.  In other words, there exists $n$ with $2 \leq n \leq N-1$ so that the  $i^{th}$ column of $X$ is $\vec{e}_{i-1}$ for each $i$ except for $i=1, n+1$, which are both $\vec{0}$.  This means that the action of $X$ on the standard basis vectors satisfies
\begin{equation} \label{equation: action of N}
X \vec{e}_i = \vec{e}_{i-1}    
\end{equation}
for all $i$ except $i=1, n+1$, in which case $N\vec{e}_i=\vec{0}$. For convenience, we often refer to the first $n$ basis vectors as the \emph{top block}, denoted
\[\mathcal{T} = \{\vec{e}_1, \ldots, \vec{e}_n\},\]
and the last $N-n$ basis vectors as the \emph{bottom block}, denoted
\[\mathcal{B} = \{\vec{e}_{n+1}, \ldots, \vec{e}_N\}.\]
\end{definition}

\subsection{Springer Schubert cells for the Jordan canonical matrix} Note that we can determine the flags in the Springer Schubert cell $\cellX{w}$ by computing $Xg$ for each canonical form matrix $g \in \cellX{w}$ and then testing whether each column is in the span of the first $i$ columns of $g$. In this section, we give first results describing the matrices in canonical form that can appear in Springer fibers.

\begin{lemma} \label{lemma: combinatorial conditions on matrix in Springer}
    Let $X$ be in Jordan canonical form of Jordan type $(n, N-n)$. Suppose $g$ is a matrix representative of a flag in the Springer fiber of $X$, in the canonical form of Definition~\ref{definition: canonical representative and Schub cells}.  Let  $\vec{g}_k$ be its $k^{th}$ column vector and say the pivot is in row $piv(g_k)$, namely
    \[\vec{g}_k = \vec{e}_{piv(g_k)} + \vec{u}\]
    where $\vec{u}$ is zero in rows $piv(g_k), piv(g_k)+1, piv(g_k)+2,\ldots$. Then the following hold:
    \begin{enumerate}
        \item If $1 \leq piv(g_k) \leq n$ then $\vec{u} = \vec{0}$ and all of $\vec{e}_1, \vec{e}_2, \ldots, \vec{e}_{piv(g_k)-1}$ appear among columns $1, 2, \ldots, k-1$ of $g$.
        \item If $n+1 \leq piv(g_k) \leq N$ then columns $1, 2, \ldots, k-1$ of $g$ have pivots in rows $n+1, n+2,\ldots, piv(g_k)-1$.  
        \item If $n+2 \leq piv(g_k)$ and column $k'$ has pivot in row $piv(g_k)-1$ then 
        \[X(\vec{g}_k) - \vec{g}_{k'} \in sp \langle \vec{e}_1, \ldots, \vec{e}_n \rangle \cap sp \langle \vec{g}_1, \ldots \vec{g}_{k-1} \rangle\]
    \end{enumerate}
\end{lemma}

\begin{proof}
By definition of the Springer fiber, we know that $X(g_k) \in sp \langle \vec{g}_1, \vec{g}_2, \ldots, \vec{g}_{k-1} \rangle$ for all $k$ and hence $X^{\ell}(\vec{g}_k) \in sp \langle \vec{g}_1, \vec{g}_2, \ldots, \vec{g}_{k-1} \rangle $ for all $k$ as well.  

By construction of $X$ we know that $X\vec{e}_\ell = \vec{e}_{\ell-1}$ unless $\ell=n+1$ or $1$ and that for $\ell \in \{1,n+1\}$ we have $X\vec{e}_{\ell} = \vec{0}$.  It follows that if $X^\ell(\vec{u}) = \vec{0}$ for some vector $\vec{u}$ and integer $\ell$ then $\vec{u}$ can only be nonzero in rows $1, 2, \ldots, \ell$ and rows $n+1, n+2, \ldots, n+\ell$.

Putting these facts together, if we write $\vec{g}_k = \vec{e}_{piv(g_k)} + \vec{u}$ then all of 
\[X^{1}\vec{g}_k, X^{2}\vec{g}_k, \ldots, X^{piv(g_k)-1}\vec{g}_k\]
are in $sp \langle \vec{g}_1, \vec{g}_2, \ldots, \vec{g}_{k-1} \rangle$.  We treat the cases $piv(g_k) \leq n$ and $piv(g_k) > n$ separately.

If $piv(g_k) \leq n$ then the vector $X^{\ell}\vec{g}_k$ has pivot in row $piv(g_k)-\ell$ for $1 \leq \ell \leq n$.  In other words, the vectors $\{X^{\ell}\vec{g}_k\}$ form an upper-triangular set and
\[sp \langle X^{1}\vec{g}_k, X^{2}\vec{g}_k, \ldots, X^{piv(g_k)-1}\vec{g}_k \rangle = sp \langle \vec{e}_{piv(g_k)-1}, \vec{e}_{piv(g_k)-2}, \vec{e}_{piv(g_k)-3}, \ldots,  \vec{e}_1 \rangle\]
is contained in the span of the first $k-1$ columns of $g$.  The matrix $g$ is in the canonical form of Definition~\ref{definition: canonical representative and Schub cells}. Since $\vec{g}_k$ has pivot in row $piv(g_k)$ then in fact $\vec{g}_k$ must be $\vec{e}_{piv(g_k)}$.  By the same reasoning, all of $\vec{e}_1, \vec{e}_2, \ldots, \vec{e}_{piv(g_k)-1}$ appear within the first $k-1$ columns of $g$.

Now suppose that $piv(g_k)>n$.  As in the previous argument, columns with pivots in rows $n+1, n+2, \ldots, piv(g_k)-1$ all appear in the first $k-1$ columns of $g$.  This means the column vector $\vec{g}_k$ is zero in rows $n+1, n+2, \ldots, N$ except for row $piv(g_k)$.  This proves claim (2).  If $piv(g_k)>n+1$ then consider the column vector $\vec{g}_{k'}$ whose pivot is in row $piv(g_k)-1$. The difference
\[X(\vec{g}_k)-\vec{g}_{k'} \in sp \langle \vec{e}_1, \ldots, \vec{e}_n \rangle\]
because both $\vec{g}_k$ and $\vec{g}_{k'}$ have exactly one nonzero entry in rows $n+1, \ldots, 2n$.  By definition of the Springer fiber, the difference also satisfies 
\[X(\vec{g}_k)-\vec{g}_{k'} \in sp \langle \vec{g}_1, \ldots, \vec{g}_{k-1}\rangle \]
This proves the final claim.
\end{proof}

The next corollary applies the previous result to the permutation flags contained in the Springer fiber for $X$.  The first two claims are merely specializations but the third shows that if a flag $gB$ is in the Springer fiber then the permutation flag $wB$ consisting of just the pivots of $gB$ is in the Springer fiber.  This result is known for other pavings of Springer fibers as a corollary of how those pavings are constructed \cite{Pre13, Tym06a}.

\begin{corollary} \label{corollary: perm matrix pivots go in order}
    Suppose that $w$ is a permutation matrix and recall that we denote the $k^{th}$ column of $w$ by $\vec{w}_k = \vec{e}_{w(k)}$. 

    The permutation flag $wB$ is in the Springer fiber of the Jordan canonical matrix $X$ if and only if the following two conditions hold.
\begin{enumerate}
\item The pivots in the first $n$ rows increase left-to-right, namely $w^{-1}(1) < w^{-1}(2) < \cdots < w^{-1}(n)$.
\item The pivots in the last $N-n$ rows increase left-to-right, namely $w^{-1}(n+1) < w^{-1}(n+2) < \cdots < w^{-1}(N)$.
\end{enumerate}
Moreover, if $gB$ is a flag in the Springer fiber for $X$ and $w$ is the permutation matrix with the same pivots as $gB$ then the permutation flag $wB$ is also in the Springer fiber for $X$.
\end{corollary}

\begin{proof}
    Suppose $w$ is a permutation matrix satisfying properties (1) and (2).  Then $X(\vec{w}_k)$ is the standard basis vector $\vec{e}_{w(k)-1}$.  By our conventions, this basis vector is located in column $w^{-1}(w(k)-1)$ of the matrix $w$. Properties (1) and (2) guarantee that this column is to the left of the $k^{th}$ column, namely $X(\vec{w}_k) \in sp \langle \vec{w}_1, \ldots, \vec{w}_{k-1} \rangle$.  Thus $wB$ is in the Springer fiber as desired.
    
    Suppose $g$ is a matrix in the Springer fiber and in the canonical form of Definition~\ref{definition: canonical representative and Schub cells}.  Apply Lemma~\ref{lemma: combinatorial conditions on matrix in Springer} successively to each column of $g$.  If the pivots in rows $1, 2, \ldots, n$ occur in columns $k_1, k_2, \ldots, k_n$ respectively, then we conclude $k_1 < k_2 < \cdots < k_n$.  Similarly, if the pivots in rows $n+1, n+2, \ldots, N$ occur in columns $k_1, k_2, \ldots, k_n$ respectively, then we conclude $k_1 < k_2 < \cdots < k_n$. Specializing to the case when $g$ is a permutation matrix proves the first claim.  
    
    Now suppose that $g$ is in the canonical form of Definition~\ref{definition: canonical representative and Schub cells}, that $gB$ is in the Springer fiber, and that $w$ is the permutation matrix whose nonzero entries are in the positions of the pivots of $g$.  By the previous argument, if $\vec{w}_k$ is the $k^{th}$ column of $w$, then either $X\vec{w}_k$ is among the first $k-1$ columns of $w$ or $X\vec{w}_k = \vec{0}$. In particular, we conclude $X\vec{w}_k$ is in the span of the first $k-1$ columns of $w$ for all $k$ and so the flag $wB$ is also in the Springer fiber of $X$.  This completes the proof.
    \end{proof}

\section{Standard noncrossing matchings and permutations} \label{section: sncm and perms, def and first results}

In this section, we define standard noncrossing matchings as well as some of their key properties.  We then describe a map from standard noncrossing matchings to certain permutations and give a bijection between these different combinatorial objects.

 \begin{definition} An \emph{arc} $\alpha$ is a pair of integers $\init(\alpha) < \term(\alpha)$ referred to as the beginning (or start) of the arc and the end (or endpoint) of the arc.  We often draw an arc as a semicircle above the $x$-axis starting at $\init(\alpha)$ and ending at $\term(\alpha)$. 
 
 A \emph{matching}  $\match{M} = (\alpha_1, \alpha_2, \ldots, \alpha_k)$ on $\{1, 2, \ldots, 2n\}$ is a set of arcs  whose beginning and endpoints form a subset of $\{1, 2, \ldots, 2n\}$ with cardinality $2k$.  We usually assume the arcs are ordered so that $1 \leq \init(\alpha_1) < \init(\alpha_2) < \init(\alpha_k) \leq 2n$ and refer to $\alpha_i$ as the $i^{th}$ arc in the matching. 
 
 Two arcs $\alpha, \alpha'$ \emph{cross} if $\init(\alpha')<\init(\alpha) < \term(\alpha')<\term(\alpha)$.  Arcs that are drawn using our conventions are crossing if and only if the corresponding semicircles cross.
 \end{definition}

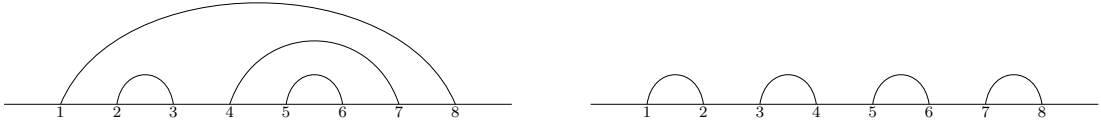
\begin{figure}[h]
\scalebox{0.75}{\begin{tikzpicture}
    \draw (0,0) -- (9,0);
    \draw (1,0) .. controls (2,2.4) and (7,2.4) .. (8,0);
    \draw (2,0) .. controls (2.1,.7) and (2.9,.7) .. (3,0);
    \draw (4,0) .. controls (4.5,1.5) and (6.5,1.5) .. (7,0);
    \draw (5,0) .. controls (5.1,.7) and (5.9,.7) .. (6,0);
    \draw (1,0.1) node[anchor=north] {\tiny $1$};
    \draw (2,0.1) node[anchor=north] {\tiny $2$};
    \draw (3,0.1) node[anchor=north] {\tiny $3$};
    \draw (4,0.1) node[anchor=north] {\tiny $4$};
    \draw (5,0.1) node[anchor=north] {\tiny $5$};
    \draw (6,0.1) node[anchor=north] {\tiny $6$};
    \draw (7,0.1) node[anchor=north] {\tiny $7$};
    \draw (8,0.1) node[anchor=north] {\tiny $8$};
\end{tikzpicture}}
\hspace{0.3in}
\scalebox{0.75}{\begin{tikzpicture}
    \draw (0,0) -- (9,0);
    \draw (1,0) .. controls (1.1,.7) and (1.9,.7) .. (2,0);
    \draw (3,0) .. controls (3.1,.7) and (3.9,.7) .. (4,0);
    \draw (5,0) .. controls (5.1,.7) and (5.9,.7) .. (6,0);
    \draw (7,0) .. controls (7.1,.7) and (7.9,.7) .. (8,0);

    \draw (1,0.1) node[anchor=north] {\tiny $1$};
    \draw (2,0.1) node[anchor=north] {\tiny $2$};
    \draw (3,0.1) node[anchor=north] {\tiny $3$};
    \draw (4,0.1) node[anchor=north] {\tiny $4$};
    \draw (5,0.1) node[anchor=north] {\tiny $5$};
    \draw (6,0.1) node[anchor=north] {\tiny $6$};
    \draw (7,0.1) node[anchor=north] {\tiny $7$};
    \draw (8,0.1) node[anchor=north] {\tiny $8$};
\end{tikzpicture}}
\caption{Examples of perfect noncrossing matchings} \label{figure: 2 perfect matchings}
\end{figure}

\begin{example} \label{example: matchings labeled by arcs} Figures~\ref{figure: 2 perfect matchings} and~\ref{figure: not perfect matching} show three examples of matchings.  In Figure~\ref{figure: 2 perfect matchings}, the matching on the left is $\match{M}_L = \{(1,8), (2,3), (4,7), (5,6)\}$ while the matching on the right is $\match{M}_R = \{(1,2), (3,4), (5,6), (7,8)\}$.  The matching in Figure~\ref{figure: not perfect matching} is $\match{M} = \{(1,4), (2,3), (7,8)\}$.
\end{example}

The next result is a lemma whose two cases simply restate the converse of the definition of crossing arcs.  

\begin{lemma} \label{lemma: nesting or adjacent}
Suppose that $\alpha$ and $\alpha'$ are two arcs that do not cross and suppose without loss of generality that $\init(\alpha')<\init(\alpha)$.  Then either:
\begin{itemize}
    \item $\term(\alpha)<\term(\alpha')$, in which case we say $\alpha$ and $\alpha'$ are \emph{nested} (and $\alpha'$ is nested above $\alpha$), or 
    \item $\term(\alpha')<\term(\alpha)$, in which case we say $\alpha$ and $\alpha'$ are \emph{unnested}.
\end{itemize}
\end{lemma}

 For example, the matching on the right in Figure~\ref{figure: 2 perfect matchings} has no nesting; its arcs are fully unnested.  In the matching on the left, arcs $(2,3)$, $(4, 7)$, and $(5,6)$ are all nested under arc $(1,8)$.  However, arcs $(2,3)$ and $(4,7)$ are unnested.
 
 \begin{definition}
A matching $\mathcal{M}$ on $\{1, 2, \ldots, 2n\}$ is called...
\begin{itemize}
    \item ... \emph{perfect} if it has exactly $n$ arcs.
    \item ... \emph{noncrossing} if no pair of arcs in $\mathcal{M}$ cross each other.  
    \item ... \emph{standard} if every $i \in \{1, 2, \ldots, 2n\}$ that lies between start and endpoints of an arc is itself the start or endpoint of an arc, namely if $i$ satisfies $\init(\alpha)<i<\term(\alpha)$ for some arc $\alpha$ then $i = \init(\alpha')$ or $i = \term(\alpha')$ for some other arc $\alpha'$.
\end{itemize}
 \end{definition}

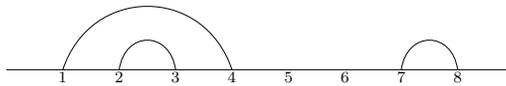
\begin{figure}[h]
    \scalebox{0.75}{\begin{tikzpicture}
    \draw (0,0) -- (9,0);
    \draw (1,0) .. controls (1.5,1.5) and (3.5,1.5) .. (4,0);
    \draw (2,0) .. controls (2.1,.7) and (2.9,.7) .. (3,0);
    \draw (7,0) .. controls (7.1,.7) and (7.9,.7) .. (8,0);

    \draw (1,0.1) node[anchor=north] {\tiny $1$};
    \draw (2,0.1) node[anchor=north] {\tiny $2$};
    \draw (3,0.1) node[anchor=north] {\tiny $3$};
    \draw (4,0.1) node[anchor=north] {\tiny $4$};
    \draw (5,0.1) node[anchor=north] {\tiny $5$};
    \draw (6,0.1) node[anchor=north] {\tiny $6$};
    \draw (7,0.1) node[anchor=north] {\tiny $7$};
    \draw (8,0.1) node[anchor=north] {\tiny $8$};
\end{tikzpicture}}
\caption{A standard noncrossing matching that is not perfect} \label{figure: not perfect matching}
\end{figure}

As examples, all of the matchings in Figures~\ref{figure: 2 perfect matchings} and~\ref{figure: not perfect matching} are standard and noncrossing.  The two matchings in Figure~\ref{figure: 2 perfect matchings} are perfect but the matching in Figure~\ref{figure: not perfect matching} is not perfect.

We now define vocabulary and a function to describe nesting more precisely.

\begin{definition} \label{definition: ancestors on init points}
Suppose that $\mathcal{M}$ is a noncrossing matching on $\{1, 2, \ldots, 2n\}$ and $\alpha$ is an arc in $\mathcal{M}$.  If $\alpha$ has at least one arc nested above it, the \emph{parent} of $\alpha$ is the arc $\alpha'$ nested immediately above $\alpha$, namely $\init(\alpha')<\init(\alpha)$ and $\term(\alpha) < \term(\alpha')$, and if $\alpha''$ is an arc with $\init(\alpha')<\init(\alpha'')<\init(\alpha)$ then also $\term(\alpha'')<\init(\alpha)$.  The \emph{ancestors} of $\alpha$ are all arcs nested at or above $\alpha$, namely $\alpha$, the parent of $\alpha$, the parent of the parent of $\alpha$, and so on (and typically listed in this order).

If $\beta$ is an ancestor of $\alpha$ then we write
\[\alpha \preceq \beta\]
and if $\beta \neq \alpha$ is an ancestor of $\alpha$ then we write $\alpha \prec \beta$.

Every noncrossing matching $\match{M} = \{\alpha_1, \ldots, \alpha_k\}$ defines a function 
\[anc: \{0, 1, 2, \ldots, 2n\} \rightarrow \{0, 1, 2, 3, \ldots, 2n\}\]
by the rule that $\anc{i}{}$ is 
\begin{itemize}
    \item $i'$ if there is an arc $(i',j')$ with $i' < i < j'$ and no arc $(i'',j'')$ has $i'<i''<i$, and
    \item $0$ if there is no such arc $(i',j')$.
\end{itemize}
We call this function the \emph{ancestor function}.  
\end{definition}

\begin{remark}
Note that if $i$ starts an arc $\alpha$ then $\anc{i}{}$ starts the parent of $\alpha$, which is the main point of the definition.  If you consider the subset of $\match{M}$ consisting of $\alpha$ and its ancestors, then this subset forms a shape like a rainbow.

Note also that $\anc{i}{}<i$ by definition and in fact the image of the ancestor function is contained in $\{0, \init(\alpha_1), \init(\alpha_2), \ldots, \init(\alpha_k)\}$.  
\end{remark}

\begin{lemma} \label{lemma: defining ancestors on arcs}
Suppose $\match{M}$ is a standard noncrossing matching and $i \in \{1, 2, \ldots, 2n\}$ satisfies $\anc{i}{}=i' \neq 0$.  Then $i$ is the start or end point of an arc $\alpha$ and 
\[\anc{\init(\alpha)}{} = \anc{\term(\alpha)}{} = i'.\]
We obtain a function on arcs that we also denote
\[anc{}{}: \{0\} \cup \match{M} \rightarrow \{0\} \cup \match{M}\]
and define by
\[\anc{\alpha}{} = \begin{array}{rl}
\beta & \textup{ if the parent of $\alpha$ is } \beta, \\
0 & \textup{ if $\alpha$ has no parent or }\alpha = 0. \end{array}\]
\end{lemma}

\begin{proof}
Suppose $\anc{i}{}=i' \neq 0$.  The matching $\match{M}$ is standard so if $i$ is not on an arc then there is no arc $(i',j')$ with $i'<i<j$.  Hence $\anc{i}{}=0$ which is a contradiction.  

Thus $i$ is on an arc, either at the start or end.  Since $\match{M}$ is noncrossing, any arc $(i',j')$ that satisfies $i'<i<j'$ also satisfies $i'<i<j<j'$.  This means $\anc{i}{}=\anc{j}{}$ for both endpoints of each arc in $\match{M}$.  In other words, the function $\anc{}{}$ is well-defined on arcs as well as on integers.
\end{proof}

\begin{example} \label{example: ancestors and function}
Consider arc $(5,6)$ in the matching on the left in Figure~\ref{figure: 2 perfect matchings}.  Its parent is $(4,7)$ and its ancestors are $(5,6)$, $(4,7)$, and $(1,8)$.  The matching has ancestor function
\[\begin{array}{c|c|c|c|c|c|c|c|c}
i & 1 & 2 & 3 & 4 & 5 & 6 & 7 & 8 \\
\cline{1-9} \anc{i}{} & 0 & (1,8) & (1,8) & (1,8) & (4,7) & (4,7) & (1,8) & 0
\end{array}\]
By contrast, no arc has a parent in the matching on the right in Figure~\ref{figure: 2 perfect matchings}, so the ancestor function in this case is identically zero.

The matching in Figure~\ref{figure: not perfect matching} has ancestor function
\[\begin{array}{c|c|c|c|c|c|c|c|c}
i & 1 & 2 & 3 & 4 & 5 & 6 & 7 & 8 \\
\cline{1-9} \anc{i}{} & 0 & (1,4) & (1,4) & 0 & 0 & 0 & 0 & 0
\end{array}\]

\end{example}

Note our convention that $\alpha$ is an ancestor of itself.  We will repeat this explicitly (and pedantically) to avoid confusion.

\begin{lemma}\label{lemma: number of ancestors}
Suppose that $\match{M}$ is a noncrossing matching and that $\alpha$ is an arc in $\match{M}$.   If $i$ arcs start and $j$ arcs end in $\{1, 2, \ldots, \init(\alpha)\}$ then $\alpha$ has $i-j$ ancestors (including itself). If $\match{M}$ is also a perfect matching, then $\alpha$ has $\init(\alpha)-j$ ancestors (including itself).
\end{lemma}

\begin{proof}
Every arc that starts before $\alpha$ either ends before $\alpha$ or is nested above $\alpha$ by Lemma~\ref{lemma: nesting or adjacent}. Thus the total number of arcs nested strictly above $\alpha$ is $i-j-1$.  This is precisely the set of ancestors of $\alpha$ other than $\alpha$. If $\match{M}$ is perfect then every integer is either a start or an endpoint, so $i+j = \init(\alpha)$.  This proves the claim.
\end{proof}

\begin{definition}
 We denote the $k$-fold composition of the ancestor function with itself by $anc^k$. For notational convenience, we use $anc^0$ to refer to the identity function.
\end{definition}

In the language of the ancestor function, the previous lemma says that if $i$ arcs start and $j$ arcs end before $\alpha$ we have
\begin{equation} \label{equation: bounds on ancestor function}
    \anc{\alpha}{k} \neq 0  \textup{ if and only if } k \leq i-j
\end{equation}
and $\anc{\alpha}{k} = 0$ if $k > i-j$.  

\begin{lemma} \label{lemma: shifting the ancestor function}
Suppose that $\match{M}$ is a standard noncrossing matching, that $\alpha_i$ is an arc in $\match{M}$, and that $\alpha_{i-1}$ is the previous arc (ordered by start points). Let 
\[r = \init(\alpha_i) - \init(\alpha_{i-1}) - 2\]

If $\alpha_i$ has a parent then for all $j \geq 1$ it satisfies
\[\anc{\alpha_i}{j} = \anc{\alpha_{i-1}}{j+r}.\]

If $\anc{\alpha_{i-1}}{r}=0$ then $\alpha_i$ has no ancestors other than itself.
\end{lemma}

\begin{proof}
First note that if the arcs are ordered by start point, then between $\init(\alpha_i)$ and $\init(\alpha_{i-1})$ lie only points that are either endpoints of arcs or (if the matching is not perfect) points not on any arc.  Thus every arc that ends in $\{\init(\alpha_{i-1})+1, \ldots, \init(\alpha_i)-1\}$ started before $\alpha_{i-1}$ and is an ancestor of $\alpha_{i-1}$ (including possibly $\alpha_{i-1}$ itself).  

In addition, any arc that is an ancestor of $\alpha_i$ must end after $\term(\alpha_i)$ and start at or before $\init(\alpha_{i-1})$ since it cannot start at any of $\{\init(\alpha_{i-1})+1,\ldots, \init(\alpha_i)-1\}$.  So all \emph{proper} ancestors of $\alpha_i$ itself are also ancestors of $\alpha_{i-1}$ (including possibly $\alpha_{i-1}$ itself).

Putting these two facts together, we conclude that the quantity $r$ defined as
\[r = \init(\alpha_i) - \init(\alpha_{i-1}) - 2 \]
counts two disjoint sets:
\begin{itemize}
    \item the number of arcs that are ancestors of $\alpha_{i-1}$ but not $\alpha_i$ 
    \item together with any points between $\init(\alpha_{i-1})$ and $\init(\alpha_i)$ that are not on any arc.
\end{itemize}
In a standard matching, no arc lies over an integer point not on an arc.  Thus if there is a point between $\init(\alpha_{i-1})$ and $\init(\alpha_i)$ that is not on any arc, there cannot be any common ancestor of $\alpha_i$ and $\alpha_{i-1}$.  In particular, arc $\alpha_i$ has no parent and $\alpha_{i-1}$ has fewer than $r$ ancestors.  This means $\anc{\alpha_{i-1}}{r}=0$ as desired.

Assume instead that $\alpha_i$ has an ancestor.  Then $\alpha_{i-1}$ has precisely $r$ ancestors that are not ancestors of $\alpha_i$.  The next ancestor of $\alpha_{i-1}$ must be the first common ancestor with $\alpha_i$ and hence the parent of $\alpha_i$.  In other words
\[\anc{\alpha_{i-1}}{r+1} = \anc{\alpha_i}{}.\]
Composing with the ancestor function on each side proves the rest of the claim.
\end{proof}

\subsection{Defining a map from standard noncrossing matchings to permutations}

Our goal is to describe Springer Schubert cells explicitly using standard noncrossing matchings, in which each arc is essentially a free variable and the nesting of the arcs indicates where in the matrix those variables are found.  To do this, we need to map arcs and matchings to matrices.  The following three functions formalize the ideas in Lemma~\ref{lemma: shifting the ancestor function}.

\begin{definition}
Suppose that $\match{M}$ is a standard noncrossing matching.  Then $\match{M}$ defines the following three functions from $\{1, 2, \ldots, N\}$ to $\{0, 1, 2, \ldots, N\}$:
    \[\begin{array}{ll}
    \jbeg(i) &= |\{j: 1 \leq j \leq i-1 \textup{ and } j \textup{ starts an arc}\}| \\
        \jend(i) &= |\{j: 1 \leq j \leq i-1 \textup{ and } j \textup{ ends an arc}\}| \\
    \jnot(i) &= |\{j: 1 \leq j \leq i-1 \textup{ and } j \textup{ is not on an arc}\}|
    \end{array}\]
\end{definition}

\begin{example} \label{example: jbeg jend jnot for example matchings}
We compute these functions for the three matchings in Figures~\ref{figure: 2 perfect matchings} and~\ref{figure: not perfect matching}.  For the matching on the left in Figure~\ref{figure: 2 perfect matchings}, we have:
\[\begin{array}{c|c|c|c|c|c|c|c|c}
i & 1 & 2 & 3 & 4 & 5 & 6 & 7 & 8 \\
\cline{1-9} \jbeg(i) & 0 & 1 & 2 & 2 & 3 & 4 & 4 & 4 \\
\cline{1-9} \jend(i) & 0 & 0 & 0 & 1 & 1 & 1 & 2 & 3 \\
\cline{1-9} \jnot(i) & 0 & 0 & 0 & 0 & 0 & 0 & 0 & 0 \\
\end{array}\]
while for the matching on the right, we have:
\[\begin{array}{c|c|c|c|c|c|c|c|c}
i & 1 & 2 & 3 & 4 & 5 & 6 & 7 & 8 \\
\cline{1-9} \jbeg(i) & 0 & 1 & 1 & 2 & 2 & 3 & 3 & 4 \\
\cline{1-9} \jend(i) & 0 & 0 & 1 & 1 & 2 & 2 & 3 & 3 \\
\cline{1-9} \jnot(i) & 0 & 0 & 0 & 0 & 0 & 0 & 0 & 0 \\
\end{array}\]
The matching in Figure~\ref{figure: not perfect matching} has
\[\begin{array}{c|c|c|c|c|c|c|c|c}
i & 1 & 2 & 3 & 4 & 5 & 6 & 7 & 8 \\
\cline{1-9} \jbeg(i) & 0 & 1 & 2 & 2 & 2 & 2 & 2 & 3 \\
\cline{1-9} \jend(i) & 0 & 0 & 0 & 1 & 2 & 2 & 2 & 2 \\
\cline{1-9} \jnot(i) & 0 & 0 & 0 & 0 & 0 & 1 & 2 & 2 \\
\end{array}\]
\end{example}

We use the previous functions to define a permutation matrix.  Lemma~\ref{lemma: pivots of matching perm increase in rows} characterizes the main properties and ideas of this permutation, though we need the technical descriptions here for our proofs.

\begin{definition} \label{definition: matching perm}
    If $\match{M}$ is a standard noncrossing matching with $k$ arcs then define a permutation $w_{\match{M}}$ of $\{1, 2, \ldots, N\}$ as follows:
    \[w_{\match{M}}(i) = \left\{ \begin{array}{ll}
    \jend(i) + \jnot(i)+1 & \textup{ if $i$ is on no arc and $j_{not} \leq n-k$,} \\
    n+\jbeg(i)+(\jnot(i)-(n-k))+1 & \textup{ if $i$ is on no arc and $j_{not} > n-k$,} \\
     n+  \jbeg(i) + \min\{0,\jnot(i)-(n-k)\} +1 & \textup{ if $i$ starts an arc, and} \\
     \jend(i)+\max\{\jnot(i),n-k\}+1 & \textup{ if $i$ ends an arc.} 
    \end{array} \right. \]
We identify this permutation with the $N \times N$ permutation matrix $w_{\match{M}}$ whose columns satisfy
\[w_{\match{M}} \vec{e}_i = \vec{e}_{w_{\match{M}}(i)}\]
\end{definition}

\begin{example} \label{example: perm matrix for three matchings}
This permutation is easier to define when the matching is perfect, in which case the columns simply indicate whether an arc starts or ends by placing nonzero entry in top half or bottom half, respectively. The idea is similar for all standard noncrossing matchings, but incorporates a shift from integers not on arcs. Below, we give from left to right the permutation matrix for the matching on the left in Figure~\ref{figure: 2 perfect matchings}, the matching on the right in Figure~\ref{figure: 2 perfect matchings}, and the matching in Figure~\ref{figure: not perfect matching}.
\[
\left(\begin{array}{cccccccc}  
0 & 0 & 1 & 0 & 0 & 0 & 0 & 0 \\
0 & 0 & 0 & 0 & 0 & 1 & 0 & 0 \\
0 & 0 & 0 & 0 & 0 & 0 & 1 & 0 \\
0 & 0 & 0 & 0 & 0 & 0 & 0 & 1 \\
\cdashline{1-8}
1 & 0 & 0 & 0 & 0 & 0 & 0 & 0 \\
0 & 1 & 0 & 0 & 0 & 0 & 0 & 0 \\
0 & 0 & 0 & 1 & 0 & 0 & 0 & 0 \\
0 & 0 & 0 & 0 & 1 & 0 & 0 & 0 \\
\end{array}\right) \hspace{1em}
\left(\begin{array}{cccccccc}  
0 & 1 & 0 & 0 & 0 & 0 & 0 & 0 \\
0 & 0 & 0 & 1 & 0 & 0 & 0 & 0 \\
0 & 0 & 0 & 0 & 0 & 1 & 0 & 0 \\
0 & 0 & 0 & 0 & 0 & 0 & 0 & 1 \\
\cdashline{1-8}
1 & 0 & 0 & 0 & 0 & 0 & 0 & 0 \\
0 & 0 & 1 & 0 & 0 & 0 & 0 & 0 \\
0 & 0 & 0 & 0 & 1 & 0 & 0 & 0 \\
0 & 0 & 0 & 0 & 0 & 0 & 1 & 0 \\
\end{array}\right)  \hspace{1em}
\left(\begin{array}{cccccccc}  
0 & 0 & 1 & 0 & 0 & 0 & 0 & 0 \\
0 & 0 & 0 & 1 & 0 & 0 & 0 & 0 \\
0 & 0 & 0 & 0 & 1 & 0 & 0 & 0 \\
0 & 0 & 0 & 0 & 0 & 0 & 0 & 1 \\
\cdashline{1-8}
1 & 0 & 0 & 0 & 0 & 0 & 0 & 0 \\
0 & 1 & 0 & 0 & 0 & 0 & 0 & 0 \\
0 & 0 & 0 & 0 & 0 & 1 & 0 & 0 \\
0 & 0 & 0 & 0 & 0 & 0 & 1 & 0 \\
\end{array}\right)
\]
\end{example}

We now confirm that this function $w_{\match{M}}$ is in fact a permutation and give a more intuitive description of it.

\begin{lemma} \label{lemma: pivots of matching perm increase in rows}
    The matrix $w_{\match{M}}$ is in fact a permutation matrix.  Moreover the pivots in the first $n$ rows occur in increasing order left to right, and similarly for the pivots in the last $N-n$ rows.  In other words:
    \[w_{\match{M}}^{-1}(1) <w_{\match{M}}^{-1}(2)< \cdots < w_{\match{M}}^{-1}(n) \hspace{0.3in} \textup{ and } \hspace{0.3in} w_{\match{M}}^{-1}(n+1) <w_{\match{M}}^{-1}(n+2)< \cdots < w_{\match{M}}^{-1}(N) \]
\end{lemma}

\begin{proof}
     Consider the set $\mathcal{S}$ of just the first $n-k$ integers not on any arc together with the integers that start arcs. If $i$ is in this set, then the function sends
     \[w_{\match{M}}(i) = \jbeg(i)+\jnot(i) + 1\]
     In other words, when $i \in \mathcal{S}$ then the image of $w_{\match{M}}$ counts how many elements of $\mathcal{S}$ have been encountered so far:
     \[w_{\match{M}}(i) = |\{j \in \mathcal{S}: j \leq i\}.\]
     There are $k$ arc endpoints and there are $n-k$ integers not on arcs in $\mathcal{S}$ so the image $w_{\match{M}}(\mathcal{S}) = \{1, 2, \ldots, n\}$. A similar statement holds for the next $N-n$ rows, except we consider arc ends and the last $N-n-k$ integers not on any arc, and the image is now $\{n+1, \ldots, N\}$.  This proves the result.
\end{proof}

In fact, there are classical bijections between the set of standard noncrossing matchings, these permutations, and strings of $N$-letter words in the alphabet $\{B, T\}$ with exactly $n$ occurrences of $T$.  We describe these bijections here.  

\begin{proposition} \label{proposition: bijection perm matching BT}
Consider the set $\mathcal{W}$ of all words $m_1 m_2 \cdots m_{N}$ consisting of $N$ letters in the alphabet $\{B, T\}$ with exactly $n$ occurrences of $T$.  There is a bijection between $\mathcal{W}$ and the set of standard noncrossing matchings on $\{1, 2, \ldots, N\}$ with at most $n$ arc starts and at most $N-n$ arc ends.  This bijection is defined by
\[\match{M} \mapsto m_1 m_2 \cdots m_{N}\]
with $m_i = B$ if $w_{\match{M}}(i) > n$ and $m_i = T$ if $w_{\match{M}}(i) \leq n$.  If $m_1m_2 \cdots m_{N}$ then the inverse map can be defined inductively as follows:
\begin{enumerate}
    \item If $n=N$ or $n=0$ the only matching is that with no arcs, which corresponds to $B^N$ or $T^N$ respectively.
    \item If $N=2$ and $n=1$ the word $TB$ is associated to the matching on $\{1, 2\}$ with no arcs and the word $BT$ is associated to the matching on $\{1, 2\}$ with arc $(1,2)$.  
    \item Suppose $1<n<N$.  
    \begin{enumerate}
        \item If the word $m_1m_2 \cdots m_{N}$ has no $i$ with $m_im_{i+1} = BT$ then the word is $T^n B^{N-n}$ and the associated matching has no arcs. 
        \item Otherwise, let $i$ have $m_im_{i+1} = BT$.  Then $\match{M} = \match{M}' \cup \{(i,i+1)\}$ where $\match{M}'$ is the matching defined inductively from the $2(n-1)$ integers $m_1m_2 \ldots m_{i-1} m_{i+2} \cdots m_n$ in $\{1, 2, \ldots, 2n\} \cap \{i, i+1\}^c$.
    \end{enumerate} 
\end{enumerate}  
\end{proposition}

\begin{proof}
Suppose that $m_1 m_2 \cdots m_N$ is a $\{B, T\}$-word and $\match{M}$ the output of this algorithm.  First note that letters in $m_1 \cdots m_{N}$ are only erased during Step (3b), and then only if they are placed on arcs in $\match{M}$.  Thus $(i,j) \in \match{M}$ is an arc only if in some iteration of the algorithm, the letters $m_im_j$ were adjacent in a substring. This means that all letters in the substring $m_{i+1} m_{i+2} \cdots m_{j-1}$ were erased in earlier iterations of the algorithm, so all integers $i+1, i+2, \ldots, j-1$ are also on arcs in $\match{M}$. In particular, the matching $\match{M}$ is standard.  

Now we show that $\match{M}$ is noncrossing.  Note that only adjacent letters can be erased.  Thus if $(i,j)$ is an arc then at some point in the inductive construction of the matching, letters $m_im_j$ were adjacent.  If $(i',j')$ is another arc with $i<i'<j$ then also $m_{i'}m_{j'}$ were adjacent.  Since $i<i'<j$ we know $m_{i'}$ must be erased before $i$ so that $m_im_j$ can be adjacent.  But if $i'<j<j'$ then $m_j$ must be erased before $i'$ so that $m_{i'}m_{j'}$ can be adjacent.  Both endpoints of an arc are erased from the word at the same time, so these cannot both be true.  Hence $i<i'<j$ implies that also $i'<j'<j$ and the matching produced by this process is standard noncrossing, as desired.  

 Finally, we confirm that this inductive process inverts the original map from matching to $\{B, T\}$-word.  Every arc start is associated to the letter $B$, and thus a pivot in the last $n$ rows in $w_{\match{M}}$, and similarly for arc ends and $T$.  The process ends when $k$ arcs have been created and the remaining word is $T^{n-k}B^{N-n-k}$, all of which correspond to integers not on any arc of $\match{M}$.  By construction, the associated columns of $w_{\match{M}}$ are assigned the lowest available pivot, which means the first $n-k$ columns associated to integers not on arcs have pivots in the $n-k$ remaining rows in the top block, and the final $N-n-k$ columns have pivots in the last $N-n$ rows.  This corresponds exactly to the desired $\{B, T\}$-string.
\end{proof}

The last paragraph proves the following corollary, which we state on its own for later convenience.

\begin{corollary} \label{corollary: standard means perfect plus T then B}
    Suppose $m_1 m_2 \cdots m_N$ is the $\{B, T\}$-word of a standard noncrossing matching $\match{M}$ on $\{1, 2, \ldots, N\}$ with $k$ arcs and with $t$ Ts and $b$ Bs. The subword obtained by erasing all $m_i$ for which $i$ is on an arc is $T^{t-k}B^{b-k}$.
\end{corollary}

\begin{example}
Consider the permutations in Example~\ref{example: perm matrix for three matchings}.  From left to right, their  $\{B, T\}$-words are BBTBBTTT, BTBTBTBT, and BBTTTBBT.  Using the algorithm on BBTBBTTT gives the following:
\begin{enumerate}
    \item Choosing the pair B{\bf BT}BBTTT we add arc $(2,3)$ to $\match{M}$ and then repeat with subword BBBTTT.
    \item Choosing the pair BB{\bf BT}TT we add arc $(5,6)$ with the original indexing of letters and then repeat with subword BBTT.
    \item Choosing the pair B{\bf BT}T we add arc $(4,7)$ with the original indexing and then repeat with subword BT.
    \item Finally we add arc $(1,8)$ so the final matching is $\match{M} = \{(1,8), (2,3), (4,7), (5,6)\}$.
\end{enumerate}
Similarly, the matching associated to BTBTBTBT is $\{(1,2),(3,4),(5,6),(7,8)\}$ and the matching associated to BBTTTBBT is $\{(1,4), (2,3), (7,8)\}$.
\end{example}

Proposition~\ref{proposition: bijection perm matching BT} gave a bijection between standard noncrossing matchings and $\{B, T\}$-words using the permutations $w_{\match{M}}$ as an intermediate step.  The next result fleshes out the bijection to permutations.

\begin{corollary} \label{corollary: direct map to perms}
 The set of $\{B, T\}$-words on $N$ letters with exactly $n$ instances of $B$ is bijective to the set of permutations $w$ with the property that 
 \begin{itemize} 
 \item $w^{-1}(1) < w^{-1}(2) < \cdots < w^{-1}(n)$ and
\item  $w^{-1}(n+1) < w^{-1}(n+2) < \cdots < w^{-1}(N)$. 
 \end{itemize}
 Explicitly, suppose $m_1 m_2 \cdots m_N$ is a $\{B, T\}$-word with exactly $n$ instances of $B$ and suppose $i_1, i_2, \ldots, i_n$ all satisfy 
    \[m_{i_1} = m_{i_2} = \cdots = m_{i_n}=B.\]  
    Associate to $m_1 \cdots m_N$ the permutation with $w(i_j)=j$ and with the complement of $\{i_1, \ldots, i_n\}$ associated successively to $n+1, n+2, \ldots, N$. This map is a bijection and inverts the map from Proposition~\ref{proposition: bijection perm matching BT}.
\end{corollary}

\begin{proof}
If $w_{\match{M}}$ is the permutation associated to a matching $\match{M}$ then Proposition~\ref{proposition: bijection perm matching BT} associates it to the word $m_1m_2\cdots m_{N}$ with $T$ in the letters $i_1, i_2, \ldots, i_n$ that have pivots in the first $n$ rows and $B$ in the other letters. The map of Corollary~\ref{corollary: direct map to perms}s inverts this by construction. The map from $\{B, T\}$-words to matchings confirms that every $w$ constructed in Corollary~\ref{corollary: direct map to perms} is in fact $w_{\match{M}}$ for some matching.
\end{proof}

\section{Defining a map from labeled matchings to Springer Schubert cells} \label{section: bijection matchings to cells}
We now construct a map from standard noncrossing matchings to Springer Schubert cells and prove that it is bijective.  To define the map, we use the same ideas as when constructing $w_{\match{M}}$ to reindex vectors according to arc starts and ends, and keep track of the arcs above each integer $i$. Intuitively, we think of each arc as being labeled by a free variable and insert the variable associated to an arc in different entries of the canonical-form matrix depending on how the arcs are nested.  We then analyze properties of arcs and Springer fibers to prove this map is a bijection.

To begin, we give a map from noncrossing matchings to vectors.  Recall that Lemma~\ref{lemma: number of ancestors} proved the number of ancestors of an arc $\alpha$ including $\alpha$ itself is $\jbeg(\init(\alpha)) - \jend(\init(\alpha))$ in our current notation.

\begin{definition} \label{definition: ancestor labels and matrix}
If $\match{M}$ is a noncrossing matching on $\{1, 2, \ldots, N\}$ with arc $\alpha$ then define a map $\iota_{\alpha, \match{M}}: \mathbb{C}^{|\match{M}|} \rightarrow \mathbb{C}^{N}$ according to the rule
    \[ \iota_{\alpha, \match{M}}(\vec{v}) = \sum_{j \geq 1} v_{\anc{\alpha}{j-1}} \vec{e}_{r_0+j}\]
    where we label the entries of each vector in $\mathbb{C}^{|\match{M}|}$ by the arcs in $\match{M}$, where $v_0=0$, and where
    \[r_0 = \jend(\init(\alpha))+\max\{\jnot(\init(\alpha)), n-k\}.\]
Let $\Im(\vec{v})$ be the $N \times N$ matrix that is zero except column $\init(\alpha)$ is $\iota_{\alpha, \match{M}}(\vec{v})$ for each $\alpha \in \match{M}$.
\end{definition}

The following observation is immediate from the definitions but useful later.

\begin{lemma} \label{lemma: number of nonzero entries in each column}
Suppose that $\match{M}$ is a standard noncrossing matching with $k$ arcs.  For each arc $\alpha \in \match{M}$ the vector $\iota_{\alpha, \match{M}}(\vec{v})$ is zero after entry
\[\jbeg(\init(\alpha))+\max\{\jnot(\init(\alpha)),n-k\}.\] 
\end{lemma}

\begin{proof}
Lemma~\ref{lemma: number of ancestors} showed the arc $\alpha$ has $\jbeg(\init(\alpha))-\jend(\init(\alpha))$ ancestors, giving the maximal $j$ with nonzero coefficient.  Summing $r_0$ in the definition of $\iota_{\alpha, \match{M}}$ gives the result.
\end{proof}

We are now in a position to define a map to each Springer Schubert cell.

\begin{definition}
    Given a standard noncrossing matching $\match{M}$ on $\{1,2,\ldots,N\}$ with $k$ arcs, define $f_{\match{M}}: \mathbb{C}^{|\match{M}|} \rightarrow M_{N}$ to be 
    \[f_{\match{M}}(\vec{v}) = w(\match{M})+ \Im(\vec{v})\] 
    where $w_{\match{M}}$ is as in Definition~\ref{definition: matching perm} and $\Im(\vec{v})$ is matrix of ancestor labels from Definition~\ref{definition: ancestor labels and matrix}.
\end{definition}

\begin{example}\label{example: image of function}
Continuing Example~\ref{example: perm matrix for three matchings}, we give the image of $f_{\match{M}}$ for the three matchings in Figures~\ref{figure: 2 perfect matchings} and~\ref{figure: not perfect matching}.  Our functions are notationally complex, but for perfect noncrossing matchings, the process they describe is simple: label the arcs with independent variables; if $i$ is an arc start then  then column $i$ has the labels of the ancestors of $i$ shifted down by as many zeroes as there are arcs that end before $i$. 
\[\begin{array}{ccc}
    \scalebox{0.55}{\begin{tikzpicture}
    \draw (0,0) -- (9,0);
    \draw[red] (1,0) .. controls (2,2.4) and (7,2.4) .. (8,0);
    \draw[blue] (2,0) .. controls (2.1,.7) and (2.9,.7) .. (3,0);
    \draw[green] (4,0) .. controls (4.5,1.5) and (6.5,1.5) .. (7,0);
    \draw (5,0) .. controls (5.1,.7) and (5.9,.7) .. (6,0);
    \draw[red] (2,1) node[anchor=south east] {$a$};
    \draw[blue] (2.3,.5) node[anchor=south] {$b$};
    \draw[green] (4.9,.9) node[anchor=south east] {$c$};
    \draw (5.2,.3) node[anchor=south east] {$d$};
\end{tikzpicture}} & \scalebox{0.55}{
\begin{tikzpicture}
    \draw (0,0) -- (9,0);
    \draw[red] (1,0) .. controls (1.1,.7) and (1.9,.7) .. (2,0);
    \draw[blue] (3,0) .. controls (3.1,.7) and (3.9,.7) .. (4,0);
    \draw[green] (5,0) .. controls (5.1,.7) and (5.9,.7) .. (6,0);
    \draw (7,0) .. controls (7.1,.7) and (7.9,.7) .. (8,0);
    \draw[red] (1.3,.5) node[anchor=south] {$a$};
    \draw[blue] (3.3,.5) node[anchor=south] {$b$};
    \draw[green] (5.3,.5) node[anchor=south] {$c$};
    \draw (7.3,.5) node[anchor=south] {$d$};
    \end{tikzpicture}} & \scalebox{0.55}{\begin{tikzpicture}
    \draw (0,0) -- (9,0);
    \draw[red] (1,0) .. controls (1.5,1.5) and (3.5,1.5) .. (4,0);
    \draw[blue] (2,0) .. controls (2.1,.7) and (2.9,.7) .. (3,0);
    \draw (7,0) .. controls (7.1,.7) and (7.9,.7) .. (8,0); 
     \draw[red] (1.9,.9) node[anchor=south east] {$a$};
    \draw[blue] (2.2,.3) node[anchor=south east] {$b$};
    \draw (7.3,.5) node[anchor=south] {$c$};
    \end{tikzpicture}} \\ \left(\begin{array}{cccccccc}  
{\color{red} a} & {\color{blue} b} & 1 & 0 & 0 & 0 & 0 & 0 \\
0 & {\color{red} a}& 0 & {\color{green} c} & d & 1 & 0 & 0 \\
0 & 0 & 0 & {\color{red} a} &  {\color{green} c} & 0  & 1 & 0 \\
0 & 0 & 0 & 0 & {\color{red} a} & 0 & 0 & 1 \\
\cdashline{1-8}
1 & 0 & 0 & 0 & 0 & 0 & 0 & 0 \\
0 & 1 & 0 & 0 & 0 & 0 & 0 & 0 \\
0 & 0 & 0 & 1 & 0 & 0 & 0 & 0 \\
0 & 0 & 0 & 0 & 1 & 0 & 0 & 0 \\
\end{array}\right) & \left(\begin{array}{cccccccc}  
{\color{red} a} & 1 & 0 & 0 & 0 & 0 & 0 & 0 \\
0 & 0 & {\color{blue} b} & 1 & 0 & 0 & 0 & 0 \\
0 & 0 & 0 & 0 & {\color{green} c} & 1 & 0 & 0 \\
0 & 0 & 0 & 0 & 0 & 0 & d & 1 \\
\cdashline{1-8}
1 & 0 & 0 & 0 & 0 & 0 & 0 & 0 \\
0 & 0 & 1 & 0 & 0 & 0 & 0 & 0 \\
0 & 0 & 0 & 0 & 1 & 0 & 0 & 0 \\
0 & 0 & 0 & 0 & 0 & 0 & 1 & 0 \\
\end{array}\right)&
\left(\begin{array}{cccccccc}  
{\color{red} a} & {\color{blue} b} & 1 & 0 & 0 & 0 & 0 & 0 \\
0 & {\color{red} a} & 0 & 1 & 0 & 0 & 0 & 0 \\
0 & 0 & 0 & 0 & 1 & 0 & 0 & 0 \\
0 & 0 & 0 & 0 & 0 & 0 & c & 1 \\
\cdashline{1-8}
1 & 0 & 0 & 0 & 0 & 0 & 0 & 0 \\
0 & 1 & 0 & 0 & 0 & 0 & 0 & 0 \\
0 & 0 & 0 & 0 & 0 & 1 & 0 & 0 \\
0 & 0 & 0 & 0 & 0 & 0 & 1 & 0 \\
\end{array}\right)
\end{array} \]
 The third example shows a phenomenon that can happen in matchings that are not perfect: a priori, based on the canonical form of Springer cells, we expect the sixth column to have another nonzero entry.  However, for the seventh column to satisfy the Springer condition, the sixth column must be zero except for its pivot entry. 
\end{example}

We begin with two lemmas relating the span of columns in the image of $f_{\match{M}}$ to the structure of the matching $\match{M}$ itself.

\begin{lemma} \label{lemma: image of columns under arbitrary arc}
Suppose $\match{M}$ is a standard noncrossing matching with arc $\alpha$ and that there are $k$ arcs in the interval $\{\init(\alpha), \ldots, \term(\alpha)\}$, which we denote $\alpha_1, \ldots, \alpha_{k-1}$ ordered by start point $\init(\alpha)<\init(\alpha_1)<\cdots \init(\alpha_{k-1})$.  Fix $\vec{v} \in \mathbb{C}^{|\match{M}|}$ and for notational convenience denote the columns of $f_{\match{M}}(\vec{v})$ by $\vec{c}_1, \vec{c}_2, \ldots, \vec{c}_{N}$.

The vectors $\vec{e}_{r_0+1}, \vec{e}_{r_0+2}, \ldots, \vec{e}_{r_0+k}$ are all contained in $\vec{c}_{\init(\alpha)}, \vec{c}_{\init(\alpha)+1}, \ldots, \vec{c}_{\term(\alpha)}$.

For each $j$ we may rewrite column $\init(\alpha_j)$ as 
\[\vec{c}_{\init(\alpha_j)} = \vec{e}_{w_{\match{M}}(\init(\alpha))+j} + \left( \sum_{1 \leq i \leq \ell} v_{\anc{\alpha_j}{i-1}} \vec{e}_{r_0+r_1+i}\right) + \left( \sum_{1 \leq i} v_{\anc{\alpha}{i-1}}\vec{e}_{r_0+j+i}\right)\]
where $\ell$ is the number of ancestors of $\alpha_j$ that are strictly below $\alpha$ (including $\alpha_j$ itself) and $r_1$ is the number of arcs nested under $\alpha$ that end before $\init(\alpha_j)$. Then $r_1+\ell = j$ and 
\[X^j \vec{c}_{\init(\alpha_j)} - \vec{c}_{\init(\alpha)} \in sp \langle \vec{e}_1, \ldots, \vec{e}_{r_0} \rangle \subseteq sp \langle \vec{c}_1, \ldots, \vec{c}_{\init(\alpha)-1} \rangle.\]
\end{lemma}

\begin{proof}
For notational convenience, write $I_{\alpha} = \{\init(\alpha), \ldots, \term(\alpha)\}$.  
First note that since $\match{M}$ is standard, every integer in $I_{\alpha}$ lies on an arc so $\jnot$ is constant when restricted to $I_{\alpha}$.  Thus the pivots of $w_{\match{M}}$ increase by exactly one at each successive arc start in $I_{\alpha}$ so
\[w_{\match{M}}(\init(\alpha_j)) = w_{\match{M}}(\init(\alpha))+j\]
for all $j$ as claimed. The analogous statement for arc ends tells us that 
\[\{\vec{c}_{\term(\alpha_1)}, \ldots, \vec{c}_{\term(\alpha_{k-1})}, \vec{c}_{\term(\alpha)}\} = \{\vec{e}_{r_0+1}, \vec{e}_{r_0+2}, \ldots, \vec{e}_{r_0+k}\}\]
proving the first part of the claim.

Since $\match{M}$ is noncrossing, every arc $\alpha_1, \ldots, \alpha_{j-1}$ satisfies one of two properties: it ends before $\init(\alpha_j)$ or it ends after $\term(\alpha_j)$.  In the first case, the arc is counted in $\jbeg(\init(\alpha_j))$ but not $\jbeg(\init(\alpha))$ so contributes one to $r_1$. In the second case, the arc is an ancestor of $\alpha_j$ but not $\alpha$ so appears as a term in the second summand.  Together, this means every arc that starts in the interval $\init(\alpha)+1, \ldots, \init(\alpha_j)$ contributes either one to $r_1$ or one to $\ell$ and so $r_1 + \ell=j$ as claimed.  

The column vector $\vec{c}_{\init(\alpha_j)}$ is defined as
\[\vec{c}_{\init(\alpha_j)} = \vec{e}_{w_{\match{M}}(\init(\alpha_j)} + \sum_{1 \leq i} v_{\anc{\alpha_j}{i-1}} \vec{e}_{r_0+r_1+i}.\]
Our analysis of the vector from $w_{\match{M}}$ allows us to rewrite the first term and Lemma~\ref{lemma: shifting the ancestor function} allows us to break the second sum into the part that is below $\alpha$ and the part that is nested above $\alpha$, proving the next formula in the claim.

Now apply $X^{d}$ using the formula for $X$ in Equation~\eqref{equation: action of N} and linearity:
\[\begin{array}{ll} X^{d} &\left(\vec{e}_{w_{\match{M}}(\init(\alpha))+j} + \left( \sum_{1 \leq i \leq \ell} v_{\anc{\alpha_j}{i-1}} \vec{e}_{r_0+r_1+i}\right) + \left( \sum_{1 \leq i} v_{\anc{\alpha}{i-1}}\vec{e}_{r_0+j+i}\right) \right) \\ &= \vec{e}_{w_{\match{M}}(\init(\alpha))+j-d} + \left( \sum_{1 \leq i \leq \ell} v_{\anc{\alpha_j}{i-1}} \vec{e}_{r_0+r_1+i-d}\right) + \left( \sum_{1 \leq i} v_{\anc{\alpha}{i-1}}\vec{e}_{r_0+j+i-d}\right)\end{array}\]
where our convention is that any vectors $\vec{e}_{d'}$ with $d' \leq 0$ are zero. Note that when $d=j$ the first and third terms in this sum add to $\vec{c}_{\init(\alpha)}$ while  if $d=j$ and $i \leq \ell$ then 
\[r_0+r_1+i-j \leq r_0+(r_1+\ell)-j = r_0.\]
In other words, we have
\[X^j(\vec{c}_{\init(\alpha_{j})}) - \vec{c}_{\init(\alpha)}) \in sp \langle \vec{e}_1, \ldots, \vec{e}_{r_0} \rangle\]
for each $j \leq k-1$. By definition of $r_0$, there are $r_0$ ends of arcs and integers not on arcs in $\{1, 2, \ldots, \init(\alpha)-1\}$ so by construction of $w_{\match{M}}$ and $f_{\match{M}}$, all of $\vec{e}_1, \ldots, \vec{e}_{r_0}$ occur within the set $\{\vec{c}_1, \ldots, \vec{c}_{\init(\alpha)-1}\}$.  This proves the claim.
\end{proof}

The next two results follow from the previous and will show that all matrices in the image of $f_{\match{M}}$ satisfy the Springer condition.

\begin{corollary} \label{corollary: Springer condition arc with parent}
Assume the hypotheses and notation of Lemma~\ref{lemma: image of columns under arbitrary arc} and suppose $\alpha' \in \match{M}$ has parent $\alpha \in \match{M}$. Then for all choices of $\vec{v} \in \mathbb{C}^{N}$ we have
\[X\vec{c}_{\init(\alpha')} \in sp \langle \vec{c}_1, \ldots, \vec{c}_{\init(\alpha')-1} \rangle.\]
\end{corollary}

\begin{proof}
Matching the notation of Lemma~\ref{lemma: image of columns under arbitrary arc} to these hypotheses, the arc $\alpha' = \alpha_1$ while $\init(\alpha)<\init(\alpha')$ and $r_0 < \init(\alpha')$.  Applying Lemma~\ref{lemma: image of columns under arbitrary arc} proves the claim.
\end{proof}

\begin{corollary} \label{corollary: span cols arc wo parent and Springer cond no parent}
Assume the hypotheses and notation of Lemma~\ref{lemma: image of columns under arbitrary arc} and suppose $\alpha$ has no parent.  Then 
\[X \vec{c}_{\init(\alpha)} \in sp \langle \vec{c}_1, \ldots, \vec{c}_{\init(\alpha)-1} \rangle\]
and also
\[\begin{array}{ll}sp & \langle \vec{c}_{\init(\alpha)}, \vec{c}_{\init(\alpha)+1}, \ldots, \vec{c}_{\term(\alpha)} \rangle \\ &= sp \langle \vec{e}_{r_0+1}, \ldots, \vec{e}_{r_0+k}, \vec{e}_{w_{\match{M}}(\init(\alpha))}, \vec{e}_{w_{\match{M}}(\init(\alpha))+1}, \ldots, \vec{e}_{w_{\match{M}}(\init(\alpha))+k-1} \rangle\end{array}\]
regardless of choice of $\vec{v} \in \mathbb{C}^{n}$.
\end{corollary}

\begin{proof}
If $\alpha$ has no parent then by Lemma~\ref{lemma: image of columns under arbitrary arc} we have
\[\vec{c}_{\init(\alpha)} = \vec{e}_{w_{\match{M}}(\init(\alpha))} + v_{\alpha}\vec{e}_{r_0+1} \]
and for each $j < k$ we can write
\[\vec{c}_{\init(\alpha_j)} = \vec{e}_{w_{\match{M}}(\init(\alpha))+j} + \left( \sum_{1 \leq i \leq \ell} v_{\anc{\alpha_j}{i-1}} \vec{e}_{r_0+r_1+i}\right) + v_{\alpha}\vec{e}_{r_0+j+1}\]
with the same notation as Lemma~\ref{lemma: image of columns under arbitrary arc} and with $j+1  \leq k$.  Also by Lemma~\ref{lemma: image of columns under arbitrary arc} we know that $\vec{e}_{r_0+1}, \ldots, \vec{e}_{r_0+k}$ are all amongst columns $\vec{c}_{\init(\alpha)}, \ldots, \vec{c}_{\term(\alpha)}$. Thus
\[\vec{c}_{\init(\alpha_j)}-\vec{e}_{w_{\match{M}}(\init(\alpha))+j} \in sp \langle \vec{c}_{\init(\alpha)}, \ldots, \vec{c}_{\term(\alpha)} \rangle\]
for each $j < k$ as well as $\alpha$ itself. We conclude the second part of the claim, that
\[sp \langle \vec{c}_{\init(\alpha)}, \ldots, \vec{c}_{\term(\alpha)} \rangle = sp \langle \vec{e}_{r_0+1}, \ldots, \vec{e}_{r_0+k}, \vec{e}_{w_{\match{M}}(\init(\alpha))}, \ldots, \vec{e}_{w_{\match{M}}(\init(\alpha))+k-1}\rangle. \]
Now consider $X\vec{c}_{\init(\alpha)}$.  If $w_{\match{M}}(\init(\alpha)) = n+1$ then $X\vec{c}_{\init(\alpha)} = v_{\alpha} \vec{e}_{r_0}$ which is by definition in $\vec{c}_1, \ldots, \vec{c}_{\init(\alpha)-1}$.  Otherwise $X\vec{c}_{\init(\alpha)}-\vec{e}_{w_{\match{M}}(\init(\alpha))-1}$ is a scalar multiple of $\vec{e}_{r_0}$.  

Thus we must show $\vec{e}_{w_{\match{M}}(\init(\alpha))-1}$ is within the span of the first $\init(\alpha)-1$ columns.  Suppose $i'$ indexes the column with pivot in row $w_{\match{M}}(\init(\alpha))-1$.  Note $i'<\init(\alpha)$ by Lemma~\ref{lemma: pivots of matching perm increase in rows}.  If $i'$ is not on an arc then $\vec{c}_{i'} = \vec{e}_{w_{\match{M}}(\init(\alpha))-1}$ so the claim holds.  If instead $i'$ starts an arc, then that arc has an ancestor $\alpha'$ without a parent.  According to the previous part of this proof, the span of columns $\vec{c}_{\init(\alpha')}, \ldots, \vec{c}_{\term(\alpha')}$ contains $\vec{e}_{w_{\match{M}}(\init(\alpha))-1}$ (as well as the basis vectors for all other pivots in those columns).  So in this case, too, the claim holds.  No part of this argument depended on the choice of $\vec{v} \in \mathbb{C}^{|\match{M}|}$ so this completes the proof.
\end{proof}

We now show that the matrices in the image of $f_{\match{M}}$ are in canonical form.  This will help us show that the map $f_{\match{M}}$ injects into the flag variety.

\begin{lemma}\label{lemma: correct form of matrix image}
Suppose that $\match{M}$ is a standard noncrossing matching. Each matrix in the image of $f_{\match{M}}$ is in the canonical form described in Definition~\ref{definition: canonical representative and Schub cells}.  
\end{lemma}

\begin{proof}
Suppose $M$ is a matrix in the image of $f_{\match{M}}$.  By construction, each of the last $N-n$ rows of $M$ has exactly one nonzero entry and that entry is one, so in particular, all entries to the right of the pivot are zero.  Now consider the $j^{th}$ row for $1 \leq j \leq n$. Denote the column of $w_{\match{M}}$ that is one in row $j$ by $i = w_{\match{M}}^{-1}(j)$.  By construction, if row $j$ is one in column $i$ of $w_{\match{M}}$ then at least $j$ of the integers $\{1, 2, \ldots, i\}$ are \emph{not} arc starts.  Thus at least $j$ of the integers $\{1, 2, \ldots, i'\}$ are not arc starts for all $i' > i$, namely for all columns to the right of column $i$.  If $\alpha$ is any arc that starts to the right of $i$, then the lowest-indexed basis vector in the sum $\iota_{\alpha, \match{M}}$ is $\vec{e}_{r_0}$ for some $r_0 > j$.  Thus the column $\init(\alpha)$ of $\Im(\vec{v})$ must also be zero in row $j$.  So the matrix $f_{\match{M}}(\vec{v}) = w_{\match{M}}+\Im(\vec{v})$ is in canonical form, as claimed.
\end{proof}

Finally, we show that the image of $f_{\match{M}}$ is in the Springer Schubert cell associated to $w_{\match{M}}$, in other words that each $f_{\match{M}}(\vec{v})$ satisfies the Springer condition $XV_i \subseteq V_i$ for all $i$.

\begin{lemma} \label{lemma: matrix image is in Springer fiber}
Fix $n < N$.  Suppose that $\match{M}$ is a standard noncrossing matching on $\{1, 2, \ldots, N\}$ with $k$ arcs such that $k \leq n$ and $k \leq N-n$. Each matrix $w_{\match{M}}+\Im(\vec{v})$ in the image of $f_{\match{M}}$ gives rise to a flag $(w_{\match{M}}+\Im(\vec{v}))B$ in the Springer fiber of the Jordan canonical form $X$ for Jordan type $(n, N-n)$.
\end{lemma}

\begin{proof}
Let $\vec{c}_1, \ldots, \vec{c}_N$ denote the columns of $w_{\match{M}}+\Im(\vec{v})$. We prove the Springer condition  
\[X \vec{c}_i \in sp \langle \vec{c}_1, \ldots, \vec{c}_{i-1} \rangle\]
holds for each column $\vec{c}_i$ using cases.

{\bf Case 1: $i$ starts an arc.} If $i$ starts an arc without a parent then Corollary~\ref{corollary: span cols arc wo parent and Springer cond no parent} proves the Springer condition holds for column $i$. If $i$ starts an arc with a parent then Corollary~\ref{corollary: Springer condition arc with parent} proves the Springer condition holds for column $i$.  

{\bf Case 2: $i$ does not start an arc.} In this case, column $i$ of the matrix $\Im(\vec{v})$ is zero so $\vec{c}_i = \vec{e}_{w_{\match{M}}(i)}$. Lemma~\ref{lemma: pivots of matching perm increase in rows} showed that $\vec{e}_{w_{\match{M}}(i) -1}$ is one of the first $i-1$ columns of $w_{\match{M}}$.  We need to show that it is also among the span of $\vec{c}_1, \ldots, \vec{c}_{i-1}$.  

If the nonzero entry of $\vec{c}_i$ is within the first $n$ rows then $i$ must be an arc end or an integer not on an arc.  In this case $\Im(\vec{v})$ is zero in column $i$ and so $\vec{c}_i=\vec{e}_{w_{\match{M}}(i)}$ proving the claim.

Otherwise, $i$ is not on an arc. If no arc starts to the left of $i$ then the first $i-1$ columns of $w_{\match{M}}$ agree with $\vec{c}_1, \vec{c}_2, \ldots, \vec{c}_{i-1}$ so the claim holds for $\vec{c}_i$ because it holds for all columns of $w_{\match{M}}$ by Lemma~\ref{lemma: pivots of matching perm increase in rows}.

Otherwise, no arc starts to the left of $i$ and ends to the right of $i$ because $\match{M}$ is standard.  Let $\alpha$ be the last arc to the left of $i$, namely with the property that no arcs start or end between $\term(\alpha)+1$ and $i$.  By definition, the arc $\alpha$ has no ancestor. Lemma~\ref{lemma: if no ancestor then span of basis vectors} thus proves that the span of the first $\term(\alpha)$ columns of $f_{\match{M}}(\vec{v})$ is the same as the span of the first $\term(\alpha)$ columns of $w_{\match{M}}$.  The next $i-\term(\alpha)$ columns of $f_{\match{M}}(\vec{v})$ agree with those of $w_{\match{M}}$ since none correspond to arc-starts.  Thus the span of the first $i-1$ columns of $w_{\match{M}}$ is precisely the span of $\vec{c}_1, \ldots, \vec{c}_{i-1}$.  Since the Springer condition holds for column $i$ of $w_{\match{M}}$ it also holds for $\vec{c}_i$. This completes the proof.
\end{proof}

We are now in a position to prove the main result of this section.

\begin{theorem} \label{theorem: fM is a bijection onto Springer Schubert cell}
Suppose that $X$ is in Jordan canonical form of Jordan type $(n,N)$ and consider the Springer fiber $\mathcal{S}_X$.  Then the Springer Schubert cells $\cellX{w}$ that are nonempty are exactly those for permutations $w=w_{\match{M}}$ corresponding to the standard noncrossing matchings $\match{M}$ on $\{1, 2, \ldots, N\}$ with at most $n$ arc starts and at most $N-n$ arc ends.  Moreover, for each nonempty Springer Schubert cell, the map
\[f_{\match{M}}: \mathbb{C}^{|\match{M}|} \rightarrow \cellX{w_{\match{M}}}\]
is a bijection.
\end{theorem}

\begin{proof}
Recall that the Springer Schubert cell $\cellX{w}$ is the intersection $\mathcal{S}_X \cap \mathcal{C}_w$ of Springer fiber with Schubert cell.  Corollary~\ref{corollary: perm matrix pivots go in order} says that the Springer Schubert cell $\cellX{w}$ is nonempty if and only if the permutation flag $wB \in \mathcal{S}_X$.  Moreover Corollary~\ref{corollary: perm matrix pivots go in order} also shows that the permutation flag $wB \in \mathcal{S}_X$ if and only if the pivots of the permutation matrix $w$ increase in both top and bottom blocks, namely $w^{-1}(1) < w^{-1}(2) < \cdots < w^{-1}(n)$ and similarly $w^{-1}(n+1) < \cdots < w^{-1}(N)$.  Corollary~\ref{corollary: direct map to perms} showed that these are exactly the set of permutations $\{w_{\match{M}}\}$ for standard noncrossing matchings $\match{M}$ on $\{1, \ldots, N\}$ with at most $n$ arc starts and at most $N-n$ arc ends.   
Lemma~\ref{lemma: correct form of matrix image} shows that $f_{\match{M}}: \mathbb{C}^{|\match{M}|} \rightarrow \cellX{w_{\match{M}}}$ in fact induces a map into the Schubert cell for $w_{\match{M}}$ and that this map is injective, since each $\vec{v} \in \mathbb{C}^{|\match{M}|}$ is sent to a distinct canonical matrix representative $f_{\match{M}}(\vec{v})$.  Lemma~\ref{lemma: matrix image is in Springer fiber} confirms that the image of $f_{\match{M}}$ is contained in the Springer Schubert cell.

All that remains is to prove that the map $f_{\match{M}}$ is surjective.  Suppose that $g \in \cellX{w_{\match{M}}}$ is a matrix in canonical form that is not in the image of $f_{\match{M}}$. The first column of each matrix in $\cellX{w_{\match{M}}}$ is in the kernel of $X$ since the Springer condition says $XV_1 \subseteq \{0\}$ in this case.  Depending on the pivots of $w_{\match{M}}$, the first column of $g$ either is $\vec{e}_1$ or  has the form $\vec{e}_{n+1}+a\vec{e}_1$ for some $a \in \mathbb{C}$. In the first case, all matrices in the image of $f_{\match{M}}$ also have first column $\vec{e}_1$ since the pivot is in the first row. In the second case, if $\alpha_1$ denotes the first arc in $\match{M}$ then setting $v_{\alpha_1} =a$ creates at least one matrix in the image of $f_{\match{M}}$ that agrees with the first column of $g$.  

So assume there is at least one matrix in the image of $f_{\match{M}}$ that agrees with the first $i$ columns of $g$ and consider the $(i+1)^{th}$ column.   All of the columns of $g$ with pivots in the first $n$ rows in fact are standard basis vectors $\vec{e}_j$ by Lemma~\ref{lemma: combinatorial conditions on matrix in Springer}.  By construction, all matrices in the image of $f_{\match{M}}$ also have $\vec{e}_j$ in these columns.  

So assume the column has pivot in the bottom $N-n$ rows.  Let $\vec{g}_1, \vec{g}_2, \ldots$ denote the columns of $g$.  Suppose that $f_{\match{M}}(\vec{v})$ is a matrix whose first $i$ columns agree with the first $i$ columns of $g$ and denote the columns of $f_{\match{M}}(\vec{v})$ by $\vec{c}_1, \vec{c}_2, \ldots$.  Lemma~\ref{lemma: combinatorial conditions on matrix in Springer} says that for some column $k \leq i$ both
\[X\vec{g}_{i+1} - \vec{c}_{k},  X\vec{c}_{i+1} - \vec{c}_{k} \in \langle \vec{e}_1, \ldots, \vec{e}_n \rangle \cap \langle \vec{c}_1, \ldots, \vec{c}_i \rangle.\]
In other words, we have
\[X\left(\vec{g}_{i+1} - \vec{c}_{i+1} \right) \in \langle \vec{e}_1, \ldots, \vec{e}_n \rangle  \cap \langle \vec{c}_1, \ldots, \vec{c}_i \rangle.\]
Since both $g$ and $f_{\match{M}}(\vec{v})$ are in canonical form, we know their columns are zero in every row to the right of a pivot.  In this case, we know $\{\vec{c}_1, \ldots, \vec{c}_i\}$ is a subset of $\{vec{e}_1, \ldots, \vec{e}_n\}$ of the form $\vec{e}_1, \vec{e}_2, \ldots, \vec{e}_r$ for some $r$. So both $\vec{g}_{i+1}$ and $\vec{c}_{i+1}$ are zero in their first $r$ rows.  Since $X\vec{e}_j = \vec{e}_{j-1}$ for all vectors $j \neq 1, n+1$ we conclude that $\vec{g}_{i+1}-\vec{c}_{i+1} \in \langle \vec{e}_{r+1} \rangle$.  An arc $\alpha \in \match{M}$ must start at $i+1$ because $\vec{c}_{i+1}$ has pivot in the last $N-n$ rows.  By definition of $f_{\match{M}}$ we know that entry in row  $r+1$ of $\vec{c}_{i+1}$ is $v_{\alpha}$ and that $v_{\alpha}$ does not appear in any earlier columns.  Since $v_{\alpha} \in \mathbb{C}$ is free, we may choose it to agree with row $r+1$ of $\vec{g}_{i+1}$.  Thus there is a matrix in the image of $f_{\match{M}}$ that agrees with the first $i+1$ columns of $g$.  By induction, we conclude $g$ is in the image of $f_{\match{M}}$ and thus $f_{\match{M}}$ is surjective.  This proves the claim.
\end{proof}

\section{The geometry and combinatorics of closures of Springer Schubert cells} \label{section: closures are indexed by arc cuts}

In this section, we describe some necessary conditions for a flag to be in the closure of a given Springer Schubert cell.  We begin with some geometric conditions on the flags and then interpret the conditions combinatorially.  In particular, we will define a combinatorial operation on matchings called \emph{cutting arcs}, which is a particular kind of unnesting operation.

\subsection{Geometric properties of flags in the boundary of a Springer Schubert cell}
We now prove two characteristics of Springer Schubert cells that are inherited by flags in the boundary: 1) containing the span of a particular set of basis vectors and 2) containing the image under a linear transformation.  In Section~\ref{section: cutting arcs}, we connect these to conditions on matchings.

We first recall a fact about flag varieties, see e.g. \cite{Ful97} for more.

\begin{proposition} \label{proposition: projection facts}
Suppose that $GL_N(\mathbb{C})/B$ is the flag variety and $P \supseteq B$ is a parabolic subgroup of $GL_N(\mathbb{C})$.  Then the natural projection
\[\pi_P: GL_N(\mathbb{C})/B \rightarrow GL_N(\mathbb{C})/P\]
is a closed continuous map. In particular, this applies to the following parabolics:
\begin{enumerate}
    \item The subgroup $P_i$ consisting of invertible matrices that are zero below the diagonal in the first $i$ columns. Geometrically, the flags in $GL_N(\mathbb{C})/P_i$ consist of partial flags
    \[V_1 \subseteq V_2 \subseteq \cdots \subseteq V_i\]
    obtained by forgetting the last $N-i$ subspaces in the original flag.
    \item The subgroup $P_i^c$ consisting of invertible matrices that are zero to the left of the diagonal in the last $N-i$ rows.  Geometrically, the flags in $GL_N(\mathbb{C})/P_i^c$ consist of partial flags
    \[V_{i+1} \subseteq V_{i+2} \subseteq \cdots \subseteq V_N\]
    obtained by forgetting the first $i$ subspaces in the original flag.
    \item The subgroup $P_i^{c,r}$ consisting of invertible matrices that are zero in the bottom-left $(N-i) \times i$ block. Geometrically, the flags in $GL_N(\mathbb{C})/P_i^{c,r}$ consist of the $i$-dimensional subspaces $V_i \subseteq \mathbb{C}^N$ obtained by forgetting almost all of the original flag. In other words, $GL_N(\mathbb{C})/P_i^{r,c}$ is the Grassmannian of $i$-planes in $\mathbb{C}^N$.
\end{enumerate}
\end{proposition}

We now prove the first characteristic, which is actually a property of limits in the flag variety.

\begin{lemma}\label{lemma: if cell contains vectors then closure does too}
Suppose that $V_{\bullet}(t)$ is a continuous one-parameter family of flags with limit $V_{\bullet}(t) \rightarrow V_{\bullet}'$ and that for every $t \in \mathbb{C}$ the subspace $V_i(t)$ contains a specific set of basis vectors $\{e_{i_1}, e_{i_2}, \ldots, e_{i_j}\}$.  Then in the limit, the subspace $V_i'$ also contains $\{e_{i_1}, e_{i_2}, \ldots, e_{i_j}\}$.

In particular if every flag $V_{\bullet} \in \cellX{w}$ in a Springer Schubert cell satisfies $V_i \supseteq \{e_{i_1}, e_{i_2}, \ldots, e_{i_j}\}$ then every flag $V_{\bullet}'$ in the closure $\overline{\cellX{w}}$ does too.
\end{lemma}

\begin{proof}
The flag variety is a complex manifold so if $V_{\bullet}'$ is in the boundary of a Springer Schubert cell then there is a one-parameter family of flags $V_{\bullet}(t) \in \cellX{w}$ with limit $V_{\bullet}(t) \rightarrow V_{\bullet}'$. Thus the second claim follows immediately from the first. 

For each $i$ the subspaces also have limit $V_i(t) \rightarrow V_i'$ since the projection map $\pi_i$ in Proposition~\ref{proposition: projection facts} is continuous, meaning 
\[\begin{array}{ll} \lim_t V_i(t) & = \lim_t \pi_i\left(V_{\bullet}(t)\right) \\
&= \pi_i \left(\lim_t V_{\bullet}(t) \right) \\ 
&= \pi_i(V_{\bullet}') = V_i'. \end{array}\]

Now let $V^j$ be the subspace spanned by the vectors $\{e_{i_1}, e_{i_2},\ldots,e_{i_j}\}$ and consider the quotient map $q: \mathbb{C}^N \rightarrow \mathbb{C}^N/V^j$.  The map $q$ is continuous and induces a map on $i$-dimensional subspaces of $\mathbb{C}^N$ that we also denote $q$. If $U$ is any $i$-dimensional subspace of $\mathbb{C}^N$ then the image $q(U)$ is at least $(i-j)$-dimensional, with equality if and only if $U$ contains $V^j$.

Since $V_i(t) \supseteq V^j$ for each $t$ it follows that $q(V_i(t))$ is an $(i-j)$-dimensional subspace of $\mathbb{C}^n/V^j$ for each $t$.   The rank of a matrix is a lower semicontinuous function so the limit of a sequence of $(i-j)$-dimensional subspaces has dimension at most $i-j$. In particular, the image $q(V_i')$ has dimension at most $i-j$.  At the same time, the subspace $V_i'$ is $i$-dimensional by hypothesis; thus, the image $q(V_i')$ is at least $(i-j)$-dimensional by above.  So $q(V_i')$ has dimension exactly $i-j$ which means by above $V_i' \supseteq V^j$. We conclude $V_i' \supseteq V^j$ as desired.
\end{proof}

We now show a property about how limits of flags interact with a linear operator $X$. This property is again inherited by flags in the boundary of Springer Schubert cells.

\begin{lemma} \label{lemma: limit point satisfies image condition} 
Suppose that $V_{\bullet}(t) \rightarrow V_{\bullet}'$ is a continuous one-parameter family and that for some $k,i,j$ and all $t$ we have 
\[X^kV_i(t) \subseteq V_j(t).\]
Then also $X^kV_i' \subseteq V_j'$.

In particular, if $X^k V_i \subseteq V_j$ for all $V_{\bullet} \in \cellX{w}$ in a Springer Schubert cell then also $X^k V_i \subseteq V_j$ for all flags $V_{\bullet} \in \overline{\cellX{w}}$ in the closure of the Springer Schubert cell.
\end{lemma}

\begin{proof}
Again, the forgetful map $\pi_i: V_{\bullet} \rightarrow V_i$ is continuous because it is the projection from the flag variety to a Grassmannian, as in Proposition~\ref{proposition: projection facts}.  Thus we have $V_i(t) \rightarrow V_i'$ and by the same reasoning $V_j(t) \rightarrow V_j'$.  

The linear map $X^k$ is continuous so 
\[\lim X^k V_i(t) = X^k \lim V_i(t) = X^k V_i'\]  
At the same time, we also have 
\[\lim \left(X^k V_i(t)\right) \subseteq \lim V_j(t) = V_j'\]  
Thus we conclude $X^k V_i' \subseteq V_j'$ as desired. As before, if $V_{\bullet}'$ is in the boundary of a Springer Schubert cell then there is a one-parameter family $V_{\bullet}(t)$ contained in the Springer Schubert cell with $V_{\bullet}(t) \rightarrow V_{\bullet}'$ so this result applies to the boundary of Springer Schubert cells.
\end{proof}

Now we give combinatorial conditions on matchings that imply the previous two lemmas.  

\begin{lemma} \label{lemma: if no ancestor then span of basis vectors}
Let $X$ be a nilpotent matrix in Jordan canonical form of Jordan type $(n,N-n)$.  Suppose $\match{M}$ is a standard noncrossing matching with $k$ arcs and that $i \in \{1,  \ldots, N\}$ either ends an arc that has no parent or is not on an arc. The span of the first $i$ columns of each matrix in the image of $f_{\match{M}}$ is the same subspace $V_i$.  Moreover
\[V_i = sp \langle e_1, e_2, \ldots, e_t, e_{n+1}, e_{n+2}, \ldots, e_{n+b} \rangle \]
where $t = \jend(i)+\max\{n-k,\jnot(i)\}$ and $b = \jbeg(i)+\min\{0, \jnot(i)-(n-k)\}$.  
\end{lemma}

\begin{proof}
Suppose that $i_1 < i_2 < \cdots < i_{\ell}$ are all integers that either end an arc without a parent or are not on any arc.   We prove the claim by inducting on the subscript $j$ of $i_j$.  

Our base case is $i_1$.  If $i_1$ ends an arc without a parent then there are $i_1/2$ arcs in $\{1, 2, \ldots, i_1\}$.  By Corollary~\ref{corollary: span cols arc wo parent and Springer cond no parent} we know 
\[V_{i_1} = sp \langle \vec{e}_1, \ldots, \vec{e}_{i_1/2}, \vec{e}_{n+1}, \ldots, \vec{e}_{n+(i_1/2)} \rangle.\]
Since $\jnot(i_1)=0$ by hypothesis, this proves the claim.  If instead $i_1$ is not on an arc then 
\[V_{i_1} = V_1 = sp \langle \vec{e}_1 \rangle\] 
which also satisfies the claim.

Now suppose that the claim holds for $i_{j-1}$ and consider $i_j$.  If $i_j$ ends an arc $\alpha$ without a parent then $\init(\alpha)-1$ either is not on an arc or ends a parentless arc to the left of $\alpha$. So the inductive hypothesis applies to $V_{\init(\alpha)-1}$.   Corollary~\ref{corollary: span cols arc wo parent and Springer cond no parent} applies to the span $U$ of columns $\init(\alpha), \ldots, \term(\alpha)$ of $f_{\match{M}}(\vec{v})$.  Since we can decompose each subspace as
\[V_{i_j} = V_{\init(\alpha)-1} \oplus U\]
we obtain 
\[V_{i_j} = sp \langle \vec{e}_1, \ldots, \vec{e}_{t}, \vec{e}_{n+1}, \ldots, \vec{e}_{b} \rangle\]
as desired.  Similarly, if $i_j$ is not on an arc then since $\match{M}$ is standard, all arcs that start before $i_j$ end before $i_j$.  Hence $i_j-1$ either is not on an arc or is the end of an arc with no parent.  In either case, by the inductive hypothesis together with the analogous decomposition 
\[V_{i_j} = V_{i_j-1} \oplus sp \langle \vec{e}_{w_{\match{M}}(i_j)} \rangle\]
we have the desired result. By induction, the claim is proven.
\end{proof}

\begin{example}
For example, consider the third matrix and matching in Example~\ref{example: image of function}. Any choice of $i \in \{4, 5, 6\}$ satisfies the hypotheses of Lemma~\ref{lemma: if no ancestor then span of basis vectors} on the matching.  By inspection of the matrix, we see that the span of the first four, five, and six columns are all the span of basis vectors. If $i=4$ then $t=b=2$.  If $i=5$ then $t=3$ and $b=2$.  If $i=6$ then $t=b=3$.  

This does not hold for all $i$.  For instance, if we pick $i=2$ or $i=3$ note that the span of the first two (respectively three) column vectors is not the span of two basis vectors.
\end{example}

\begin{lemma} \label{lemma: closure condition under arcs}
Suppose that $\match{M}$ is a standard noncrossing matching on $\{1, 2, \ldots, N\}$ with at most $n$ arc starts and at most $N-n$ arc ends, that  $\alpha$ is an arc in $\match{M}$, and that there are $k$ arcs with start or endpoints in $\{\init(\alpha), \ldots, \term(\alpha)\}$.  Then 
\[X^k V_{\term(\alpha)} \subseteq V_{\init(\alpha)-1}\]
for each flag $V_{\bullet}$ in the image of $f_{\match{M}}$.
\end{lemma}

\begin{proof}
With the notation of Lemma~\ref{lemma: image of columns under arbitrary arc},  the basis vectors $\vec{e}_1, \ldots, \vec{e}_{r_0}$ are all in $V_{\init(\alpha)-1}$.  Among the next $\term(\alpha)-\init(\alpha)+1$ columns, Lemma~\ref{lemma: image of columns under arbitrary arc} showed that the basis vectors $\vec{e}_{r_0+1}, \ldots, \vec{e}_{r_0+k}$ are the only ones whose pivots are in the first $n$ rows.  For every $j \in \{1, \ldots, k\}$ the image $X^k(\vec{e}_{r_0+j})$ is either $\vec{0}$ or $\vec{e}_{r_0+j-k}$ where $r_0+j-k \leq r_0$.  In both cases, we conclude $X^k(\vec{e}_{r_0+j})$ is in $V_{\init(\alpha)-1}$.

Any other column indexed by $\init(\alpha), \ldots, \term(\alpha)$ starts an arc since $\match{M}$ is standard. Label the arcs under $\alpha$ by $\alpha_1, \ldots, \alpha_{k-1}$ in order of increasing start point.  Lemma~\ref{lemma: image of columns under arbitrary arc} says
\[X^j \vec{c}_{\init(\alpha_j)} - \vec{c}_{\init(\alpha)} \in V_{\init(\alpha)-1}.\]
Lemma~\ref{lemma: matrix image is in Springer fiber} showed that the image of $f_{\match{M}}$ is in the Springer fiber $\mathcal{S}_X$ so in particular
\[X\vec{c}_{\init(\alpha)} \in V_{\init(\alpha)-1}.\]
Putting these two facts together, we conclude that for all $j<k$ we have
\[X^{j+1} \vec{c}_{\init(\alpha_j)} \in V_{\init(\alpha)-1}.\]
The Springer condition says $XV_i \subseteq V_{i-1}$ for all $i$ while the definition of flags says $V_{i-1} \subseteq V_i$.  Thus the subspace $V_{\init(\alpha)-1}$ is invariant under the action of $X$.  Consequently
\[X^k \vec{c}_{\init(\alpha_j)} \in V_{\init(\alpha)-1}\]
which completes the proof.
\end{proof}

The previous result implies another condition relating the end points of arcs that are parents and children.

\begin{corollary} \label{corollary: parent of arc to end of child}
Suppose $\match{M}$ is a standard noncrossing matching on $\{1, 2, \ldots, N\}$ with at most $n$ arc starts and at most $N-n$ arc ends, and suppose $\match{M}$ contains an arc $\beta_1$ whose parent is $\beta_2$.  Then $\term(\beta_2)-\term(\beta_1)=2k$ and 
\[X^{k+1} V_{\term(\beta_2)} \subseteq V_{\term(\beta_1)}\]
\end{corollary}

\begin{proof}
Since $\match{M}$ is standard, all integers in the interval $[\term(\beta_2)+1, \term(\beta_1)-1]$ are on arcs.   Since $\match{M}$ is noncrossing, all arcs that start in $[\term(\beta_2)+1, \term(\beta_1)-1]$ end in $[\term(\beta_2)+1, \term(\beta_1)-1]$, else they would cross $\beta_2$.  Since $\beta_2$ is the parent of $\beta_1$ we know that all arcs that end in $[\term(\beta_2)+1, \term(\beta_1)-1]$ also start in the same interval, else they would be nested between $\beta_1$ and $\beta_2$. So $\term(\beta_2)-\term(\beta_1)$ is even.

Suppose that $\alpha_1, \alpha_2, \ldots, \alpha_j$ are all arcs in $[\term(\beta_2)+1, \term(\beta_1)-1]$ whose parent is $\beta_2$ listed in increasing order of start point.  Thus we have $\term(\alpha_{\ell-1})=\init(\alpha_{\ell})$ for each $2 \leq \ell \leq j$ since none of the $\alpha_\ell$ can be nested over each other.  For each $1 \leq \ell \leq j$ let $k_\ell$ denote the number of arcs that start and end in the interval $[\init(\alpha_{\ell}), \term(\alpha_{\ell})]$.  By construction $k=k_1+k_2+\cdots +k_j$.  By Lemma~\ref{lemma: closure condition under arcs} we know 
\[X^{k_{\ell}}V_{\term(\alpha_\ell)} \subseteq V_{\init(\alpha_\ell)-1}\]
which by definition of the $\alpha_\ell$ satisfies
\[V_{\init(\alpha_{\ell})-1} = V_{\term(\alpha_{\ell-1})}\]
using the notational convention that $\alpha_0$ is $\beta_1$.  Combining these, we have
\[\begin{array}{rl}
X^kV_{\term(\alpha_j)} &\subseteq X^{k-k_j} V_{\init(\alpha_j)-1} =  \\ X^{k-k_j} V_{\term(\alpha_{j-1})} &\subseteq X^{k-k_j-k_{j-1}}V_{\init(\alpha_{j-1})-1} = \\ X^{k-k_j-k_{j-2}}V_{\term(\alpha_{j-2})} &\subseteq  \cdots  \\ X^{k-k_j-k_{j-2}-\cdots -k_2}V_{\term(\alpha_{1})} = X^{k_1}V_{\term(\alpha_1)} & \subseteq V_{\init(\alpha_1)-1} = V_{\term(\beta_1)}.\end{array}\]
The Springer condition says $XV_{\term(\beta_2)} \subseteq V_{\term(\alpha_j)}$.  Hence as desired
\[X^{k+1} V_{\term(\beta_2)} \subseteq X^k V_{\term(\alpha_j)} \subseteq V_{\term(\beta_1)}.\]
\end{proof}

We now obtain an upper bound on the number of cells that intersect the closure of $\cellX{\match{M}}$.

\begin{lemma} \label{lemma: closure can only swap B and T}
    Suppose that $\match{M}$ and $\match{M}'$ are standard noncrossing matchings on $\{1, 2, \ldots, N\}$ with at most $n$ arc starts and at most $N-n$ arc ends, and denote their $\{B,T\}$-words $m_1m_2\cdots m_{2n}$ and $m_1'm_2'\cdots m_{2n}'$ respectively.  Suppose that $\cellX{\match{M}'}$ contains at least one flag in the closure $\overline{\cellX{\match{M}}}$.  If $i$ is not on an arc in $\match{M}$ then $m_i = m_i'$ and if $\alpha \in \match{M}$ then 
    \[\{m_{\init(\alpha)}', m_{\term(\alpha)}'\} = \{B, T\}.\]
\end{lemma}

\begin{proof}
If $i$ is not on an arc in $\match{M}$ then Lemma~\ref{lemma: if no ancestor then span of basis vectors} says that all flags $V_{\bullet} \in \overline{\cellX{\match{M}}}$ satisfy $V_i = sp \langle \vec{e}_1, \ldots, \vec{e}_t, \vec{e}_{n+1}, \ldots, \vec{e}_{n+b}\rangle$ for some $t, b$.  If $i=1$ then $V_1$ is either the span of $\vec{e}_1$ or $\vec{e}_{n+1}$ for all flags in the closure, and $m_1'=m_1$ is T or B respectively.  Otherwise, if $i$ is not on an arc then $i-1$ is either on no arc or the end of an arc without a parent.  Again Lemma~\ref{lemma: if no ancestor then span of basis vectors} applies and says that $V_{i-1}$ is constant among all flags $V_{\bullet} \in \overline{\cellX{\match{M}}}$ as well.  We can write $V_i = V_{i-1} \oplus sp \langle \vec{e}_\ell \rangle$ and since both $V_i$ and $V_{i-1}$ are constant, we have $m_i'=m_i$ is either B or T, depending on whether the $i^{th}$ column of $f_{\match{M}}$ is $\vec{e}_b$ or $\vec{e}_t$.  

Now suppose $\alpha \in \match{M}$ is an arc. As in Corollary~\ref{corollary: parent of arc to end of child}, 
we have $\term(\alpha)-\init(\alpha)+1 =2k$ where $k$ is the number of arcs in the interval $\init(\alpha)+1, \ldots, \term(\alpha)-1$.  First we show the substring $m_{\init(\alpha)}'m_{\init(\alpha)+1}' \cdots m_{\term(\alpha)-1}'m_{\term(\alpha)}'$ has exactly $k$ instances of $B$ and exactly $k$ instances of $T$.  Lemma~\ref{lemma: closure condition under arcs} says that
\[X^{-k} V_{\init(\alpha)-1} \subseteq V_{\term(\alpha)}\]
for all flags $V_{\bullet}$ in the closure of $\cellX{\match{M}}$. Lemma~\ref{lemma: image of columns under arbitrary arc} states that there are $t, b \leq n$ so that $V_{\init(\alpha)-1}$ is generated by vectors of the form 
\[\vec{e}_1, \ldots, \vec{e}_{t}, \vec{e}_{n+1}+\vec{u}_{1}, \vec{e}_{n+2}+\vec{u}_2, \ldots, \vec{e}_{n+b}+\vec{u}_{b}\]
where each $\vec{u}_{\ell} \in sp \langle \vec{e}_1, \ldots, \vec{e}_n\rangle$.  If the sequence $m_{\init(\alpha)}' m_{\init(\alpha)+1}' \cdots m_{\term(\alpha)}'$ has at least $k+1$ instances of $T$ then $V_{\term(\alpha)}$ contains the vector $\vec{e}_{t+k+1}$.  By construction of $X$ the image
\[X^k(\vec{e}_{t+k+1}) = \vec{e}_{t+1}\]
which cannot be in $V_{\init(\alpha)-1}$.  If instead the sequence $m_{\init(\alpha)}' m_{\init(\alpha)+1}' \cdots m_{\term(\alpha)}'$ has at least $k+1$ instances of $B$ then $V_{\term(\alpha)}$ contains the vector $\vec{e}_{n+b+k+1}+\vec{u}_{k+1}$ where $\vec{u}_{k+1}$ is zero in the last $N-n$ entries. In this case, we have 
\[X^k(\vec{e}_{n+b+k+1}+\vec{u}_{k+1}) = \vec{e}_{n+b+1}+\vec{u}'\]
where $\vec{u}'$ is also in $sp \langle \vec{e}_1, \ldots, \vec{e}_n\rangle$.  From our description of $V_{\init(\alpha)-1}$ we know that all vectors in $V_{\init(\alpha)-1}$ have at least $N-n-b$ terminal zeros.  So this case, too, is impossible. We conclude that the substring $m_{\init(\alpha)}' m_{\init(\alpha)+1}' \cdots m_{\term(\alpha)}'$ has exactly $k$ $B$s and exactly $k$ $T$s.

Finally, we induct on $k$ to complete the proof.  If a standard matching $\match{M}$ contains an arc then it contains at least one arc with $k=0$.  Indeed, if $\alpha$ has $k>0$ then there is at least one arc $\beta$ in $\init(\alpha)+1, \ldots, \term(\alpha)-1$ and by construction 
\[\term(\beta)-\init(\beta)+1 \leq (\term(\alpha) - 1) - (\init(\alpha)+1) + 1.\]
The rightmost expression simplifies to $\term(\alpha) - \init(\alpha) - 1 = 2k-2$ so by continuing this process, we eventually find an arc as desired.  If an arc $\alpha$ has $k=0$ then $\term(\alpha)=\init(\alpha)+1$ so by the previous paragraph, the matching $\match{M}'$ has
\[\{m_{\init(\alpha)}', m_{\init(\alpha)+1}'\} = \{B,T\}\]
as claimed.  As inductive hypothesis, suppose $\alpha$ has $k>0$ and for every arc $\beta$ nested below in $\match{M}$, the matching $\match{M}'$ satisfies $\{m_{\init(\beta)}', m_{\term(\beta)}'\} = \{B, T\}$.  By above, we also know that $m_{\init(\alpha)}' m_{\init(\alpha)+1}' \cdots m_{\term(\alpha)}'$ has exactly $k$ $B$s and exactly $k$ $T$s.  The $k-1$ arcs properly nested below $\alpha$ contribute $k-1$ $B$s and $k-1$ $T$s by the inductive hypothesis.  Thus there is exactly one $B$ and one $T$ remaining for $\{m_{\init(\alpha)}', m_{\term(\alpha)}'\}$.  By induction, the result holds.
\end{proof}

Thus we obtain an upper bound on the number of Springer Schubert cells that could intersect the closure $\overline{\cellX{\match{M}}}$ --- as well as a \emph{maximal set} of Springer Schubert cells $\cellX{\match{M}'}$ that could intersect $\overline{\cellX{\match{M}}}$.  In Section~\ref{section: closure section}, we'll show that in fact all these cells do intersect $\overline{\cellX{\match{M}}}$ and identify explicitly the boundary of ${\cellX{\match{M}}}$.

\begin{corollary} \label{corollary: max number of cells}
Suppose $\match{M}$ is a standard noncrossing matching on $\{1, 2, \ldots, N\}$ with $k$ arcs, where $k \leq n$ and $k \leq N-n$. Then there are at most $2^k$ standard noncrossing matchings $\match{M}'$ for which $\cellX{\match{M}'} \cap \overline{\cellX{\match{M}}}$ is nonempty.
\end{corollary}

\begin{proof}
Lemma~\ref{lemma: closure can only swap B and T} says that if $\cellX{\match{M}'}$ intersects the closure of $\cellX{\match{M}}$ then the word for $m_1'm_2'\cdots m_{2n}'$ is obtained from that of $m_1m_2 \cdots m_{2n}$ by swapping $m_{\init(\alpha)}', m_{\term(\alpha)}'$ for some number of $\alpha \in \match{M}$. Since the start and endpoints of the arcs in $\match{M}$ are disjoint, the transpositions that exchange $\init(\alpha)$ and $\term(\alpha)$ for different $\alpha \in \match{M}$ all commute.  This gives at most $2^k$ possible strings $m_1' m_2\ \cdots m_{2n}'$ that can be obtained in this way, as desired.
\end{proof}

\subsection{Cutting arcs in a matching} \label{section: cutting arcs}
We end by defining a combinatorial operation on matchings that comes from transposing the endpoints in the $\{B, T\}$-word of an arc in a matching.

\begin{definition} \label{definition: cutting arc}
Let $\match{M}$ be a standard noncrossing matching with $\{B, T\}$-word $m_1 m_2 \cdots m_N$ and $\alpha \subseteq \match{M}$. \emph{Cutting the arc $\alpha$ in $\match{M}$} gives the unique matching $cut(\match{M}, \alpha)$ associated to the $\{B, T\}$-word obtained from $m_1 m_2 \cdots m_N$ by switching letters $m_{\init(\alpha)}$ and $m_{\term(\alpha)}$.

If $\mathcal{A} \subseteq \match{M}$ then $cut(\match{M}, \mathcal{A})$ is the unique matching associated to the $\{B, T\}$-word obtained from $m_1 \cdots m_N$ by switching all pairs $(m_{\init(\alpha)}, m_{\term(\alpha)})$ for $\alpha \in \mathcal{A}$.
\end{definition}

\begin{remark} The pairs $\init(\alpha), \term(\alpha)$ are disjoint for distinct $\alpha$ since $\match{M}$ is a matching.  Thus cutting distinct arcs produces the same $\{B, T\}$-string and hence the same matching regardless of the order we cut, in the following sense.  Suppose $\mathcal{A} \subseteq \match{M}$ and $\sigma$ is any ordering $\alpha_1, \alpha_2, \ldots, \alpha_k$ of the arcs in $\mathcal{A}$.  Define $M_{\sigma} = M_{\sigma, k}$ by the recursive process
\[\match{M}_{\sigma, 1}= cut(\match{M}, \alpha_1) \hspace{0.25in} \textup{ and } \hspace{0.25in} \match{M}_{\sigma, i} = cut(\match{M}_{\sigma, i-1}, \alpha_i).\]
If $\sigma, \sigma'$ are two different orderings of the arcs in $\mathcal{A}$ we have $\match{M}_{\sigma} = \match{M}_{\sigma'}$.  So we may define
\[cut(\match{M}, \mathcal{A}) = \match{M}_{\sigma}\]
for any ordering $\sigma$ of the arcs in $\mathcal{A}$. 

As a consequence, if $\mathcal{A} = \mathcal{A}_1 \cup \mathcal{A}_2$ we say we \emph{first cut the arcs in $\mathcal{A}_1$ then the arcs in $\mathcal{A}_2$} to refer to the arc $\match{M}_{\sigma}$ obtained from any ordering of $\mathcal{A}$ in which all arcs in $\mathcal{A}_1$ are cut before any arc in $\mathcal{A}_2$.  This is well-defined even though $cut(\match{M}, \mathcal{A}_1)$ may not have all the arcs of $\mathcal{A}_2$ so $cut(cut(\match{M}, \mathcal{A}_1), \mathcal{A}_2)$ may not be well-defined.
\end{remark}

\begin{example} {\label: example of cutting arcs}
Consider the matchings in Example~\ref{example: perm matrix for three matchings}.  Let $\match{M} = \{(1,8), (2,3), (4,7), (5,6)\}$ be the matching on the far left in this example.  Its $\{B, T\}$-word is $BBTBBTTT$ so the $\{B, T\}$-word of $cut(\match{M},(4,7))$ is $BBTTBTBT$.  In other words, we have
\[cut(\match{M}, (4,7)) = \{(1,4), (2,3), (5,6), (7,8)\}\]
which we can compute via the recursive algorithm in Proposition~\ref{proposition: bijection perm matching BT}.  If we cut arc $(5,6)$ from this matching, we obtain $\{B, T\}$-word $BBTTTBBT$ and conclude
\[cut(\match{M}, \{(4,7), (5,6)\}) = \{(1,4), (2,3), (7,8)\}. \]
However if we instead compute $cut(\match{M},(5,6))$ we get the $\{B, T\}$-string $BBTBTBTT$, which corresponds to
\[cut(\match{M}, (5,6)) = \{(1,8), (2,3), (4,5), (6,7)\}.\]
This does not have an arc $(4,7)$ to cut, though the operation of exchanging the fourth and seventh letters of the $\{B, T\}$-word still makes sense.  Indeed, it produces the word $BBTTTBBT$, which is what we found before.
\end{example}

In the previous example, whenever the arc $\alpha$ exists in $\match{M}$ and has a parent, cutting $\alpha$ (considered as an operation on $\{B, T\}$-strings) produced the same matching as the operation of unnesting $\alpha$ and its parent.  This is true in general, as the following lemma establishes.

\begin{lemma} \label{lemma: cutting an arc is unnesting}
Suppose $\match{M}$ is a standard noncrossing matching.  If $\alpha$ has a parent $\beta$ then $cut(\match{M}, \alpha)$ is the matching in which $\alpha$ and $\beta$ are \emph{unnested}, namely
\[cut(\match{M}, \alpha) = \{(\init(\beta), \init(\alpha)), (\term(\alpha), \term(\beta))\} \cup \{\gamma \in \match{M}: \gamma \neq \alpha, \beta \}.\]
\end{lemma}

\begin{proof}
If $\alpha$ has parent $\beta$ then all integers in $[\init(\beta), \term(\beta)]$ are on arcs (because $\match{M}$ is standard) and all integers in $[\init(\beta)+1, \ldots, \init(\alpha)-1]$ are matched with themselves (because $\match{M}$ is noncrossing and $\beta$ is the parent of $\alpha$), and similarly for $[\term(\alpha)+1,\init(\beta)-1]$.  

Denote the $\{B, T\}$-words for $\match{M}$ and $cut(\match{M}, \alpha)$ by $m_1 \cdots m_N$ and $m_1' \cdots m_N'$ respectively.  Using the bijection of Proposition~\ref{proposition: bijection perm matching BT}, we have  
\[m_{\init(\alpha)}=m_{\init(\beta)}=B \hspace{0.2in} \textup{ and } \hspace{0.2in} m_{\term(\alpha)}=m_{\term(\beta)}=T.\]  
By definition of cutting arcs, we know $m_{\init(\alpha)}'=T$ and $m_{\term(\alpha)}'=B$ while $m_i'=m_i$ for all other letters in the $\{B, T\}$-word.  We construct the matching $cut(\match{M},\alpha)$ from the word $m_1' \cdots m_N$ using the recursive algorithm of Proposition~\ref{proposition: bijection perm matching BT}.  We just showed that all letters $m_{\init(\beta)+1}', \ldots, m_{\init(\alpha)-1}'$ are paired by arcs amongst themselves and then erased, and similarly for $m_{\term(\alpha)+1}',\ldots, m_{\term(\beta)-1}'$. So $m_{\init(\beta)}'m_{\init(\alpha)}'$ and $m_{\term(\alpha)}'m_{\term(\beta)}'$ are eventually adjacent BT subwords, meaning that arcs $(\init(\beta), \init(\alpha))$ and $(\term(\alpha), \term(\beta))$ are in the matching $cut(\match{M}, \alpha)$.  The $\{B, T\}$-strings of $\match{M}$ and $cut(\match{M}, \alpha)$ agree everywhere else so the algorithm outputs the same matching except for those two arcs. This proves the claim.
\end{proof}

\begin{corollary} \label{corollary: top to bottom gives unnesting}
Suppose $\match{M}$ is a standard noncrossing matching on $\{1, \ldots, N\}$ with fixed subset $\mathcal{A} \subseteq \match{M}$ and an ordering $\sigma$ of the arcs in $\mathcal{A}$ that is contravariant to the partial order of ancestors, in the sense that if $\alpha_i$ is nested over $\alpha_j$ then $i < j$. Let $\match{M}_{\sigma} = cut(\match{M}, \mathcal{A})$ be the matching constructed by cutting the arcs in order $\sigma$.  Let $\match{M}'$ be the matching defined successively for each $j=1, 2, \ldots, |\mathcal{A}|$ by
\begin{enumerate}
    \item unnesting arc $\alpha_j$ if it has a parent, or
    \item erasing arc $\alpha_j$ if all ancestors of $\alpha_j$ are in the set $\{\alpha_1, \ldots, \alpha_{j-1}\}$.
\end{enumerate}
Then $\match{M}'$ and $\match{M}_\sigma$ can only differ at $\init(\alpha_j), \term(\alpha_j)$ for arcs $\alpha_j$ of type (2) above. Moreover, if $\match{M}'$ differs from $\match{M}_{\sigma}$ at $\init(\alpha_j)$ then  $\match{M}_{\sigma}$ has either
\begin{itemize}
    \item an arc $(\term(\alpha_k), \init(\alpha_j))$ for an arc $\alpha_k$ to the left of $\alpha_j$ with $k<j$ whose ancestors are all in the set $\{\alpha_1, \ldots, \alpha_{k-1}\}$ 
    \item or an arc $(k, \init(\alpha_j))$ where $k$ is an integer not on any arc in $\match{M}$
\end{itemize} 
and similarly for $\term(\alpha_j)$.
\end{corollary}

\begin{proof}
We prove this by induction on the number of arcs in $\mathcal{A}$.  Lemma~\ref{lemma: cutting an arc is unnesting} shows that if $\alpha_1$ has a parent in $\match{M}$ then $cut(\match{M}, \alpha_1)$ is the unnesting of $\alpha_1$ and its parent.  

If $\alpha_1$ has no parent then cutting $\alpha_1$ exchanges the letters $m_{\init(\alpha_1)}=B$ and $m_{\term(\alpha_1)}=T$.  For $m_{\init(\alpha_1)}'=T$ to be on an arc in the matching $cut(\match{M}, \alpha_1)$ it must be matched with an instance of $B$ to the left, say in position $k$.  Suppose that this $B$ begins an arc in $\match{M}$.  Then the arc must be $(k,k')$ for $k' > \init(\alpha_j)$ else the recursive algorithm constructing matchings from $\{B, T\}$-words would not pair $(k,\init(\alpha_j))$.  But then also $k' > \term(\alpha_j)$ because $\match{M}$ is noncrossing.  So $(k,k')$ is an ancestor of $\alpha_j$ which contradicts our hypothesis.  We conclude in this case that $k$ is not on an arc in $\match{M}$.

The inductive step is the same except that $B$ could be in position $\term(\alpha_k)$ for some arc $\alpha_k \in \mathcal{A}$ with $k<j$.  Since $k<j$ our hypothesis on the ordering $\sigma$ guarantees that either $\alpha_k \succ \alpha_j$ in $\match{M}$ or $\alpha_k$ and $\alpha_j$ are unnested in $\match{M}$.  In the former case, the endpoint $\term(\alpha_k)$ must be to the right of $\alpha_j$ so there cannot be an arc $(\term(\alpha_k),\init(\alpha_j))$.  In the latter case, $\term(\alpha_k)$ must be to the left of $\init(\alpha_j)$ by hypothesis and all ancestors of $\alpha_k$ must be in $\{\alpha_1, \ldots, \alpha_{k-1}\}$ else Lemma~\ref{lemma: cutting an arc is unnesting} applies.  By induction, the claim holds.
\end{proof}

The previous result means we can index the arcs of $cut(\match{M}, \mathcal{A})$ by arcs in $\match{M}$.  If arcs are cut from top to bottom, as described in Corollary~\ref{corollary: top to bottom gives unnesting}, we can think of this as a \emph{relabeling} process all but the newly-created arcs preserve their labels, and newly-created arcs are labeled zero unless they start or end at the parent of the cut arc.

\begin{corollary} \label{corollary: function labeling via unnesting}
Let $\match{M}$ be a standard noncrossing matching on $\{1, 2, \ldots, N\}$ with $\alpha \in \match{M}$.  Define a map $\lambda_{\alpha, \match{M}}: cut(\match{M}, \alpha) \rightarrow (\match{M}-\{\alpha\}) \cup \{0\}$ by the rule that 
\[\lambda_{\alpha, \match{M}}(\beta) = \left\{ \begin{array}{ll} \beta & \textup{ if } \beta \in \match{M} - \mathcal{A}, \\
0 & \textup{ if } \beta \not \in \match{M}, \textup{ and } \\
\anc{\alpha}{} & \textup{ if }\init(\beta)=\init(\anc{\alpha}{}) \textup{ or }\term(\beta) = \term(\anc{\alpha}{})\end{array} \right.\]
Then $\lambda_{\alpha, \match{M}}$ is a well-defined function.  

Now suppose $j>1$ and $\alpha_1, \alpha_2, \ldots, \alpha_j$ is an ordering of arcs in $\match{M}$ that is contravariant with respect to nesting in the sense of Corollary~\ref{corollary: top to bottom gives unnesting}.  Define 
\[\lambda_{\alpha_1, \ldots, \alpha_j, \match{M}}: cut(\match{M}, \{\alpha_1, \ldots, \alpha_j\}) \rightarrow \match{M} \cup \{0\}\]
recursively by the rule that 
\[\lambda_{\alpha_1, \ldots, \alpha_j, \match{M}}(\beta) = \left\{ \begin{array}{ll}\lambda_{\alpha_1, \alpha_2, \ldots, \alpha_{j-1}, \match{M}}(\beta) & \textup{ if } \beta \in cut(\match{M}, \{\alpha_1, \ldots, \alpha_{j-1}\}) - \{\alpha_j\}\\
\lambda_{\alpha_1, \alpha_2, \ldots, \alpha_{j-1}, \match{M}}(\anc{\alpha_j}{}) & \textup{ if } \init(\beta)=\init(\alpha_j) \textup{ or } \term(\beta) = \term(\alpha_j) \textup{ and }\\
0 & \textup{ otherwise.}\end{array} \right.\]
Then $\lambda_{\alpha_1, \ldots, \alpha_j, \match{M}}$ is a function with image in $(\match{M} - \mathcal{A}) \cup \{0\}$ and containing $\match{M}-\mathcal{A}$.
\end{corollary}

\begin{proof}
Corollary~\ref{corollary: top to bottom gives unnesting} shows that the function $\lambda_{\alpha, \match{M}}$ is well-defined with image in $\match{M}-\{\alpha\} \cup \{0\}$. The map $\lambda_{\alpha, \match{M}}$ surjects onto $\match{M}-\{\alpha\}$ by construction. We include the following in our inductive hypothesis:  if the fiber over $\beta$ has more than one element then it must be the parent of a cut arc.  

Inducting on $j$ we conclude that $\lambda_{\alpha_1, \ldots, \alpha_j, \match{M}}$ is well-defined and that nonzero elements in its image are contained in the the image of $\lambda_{\alpha_1, \ldots, \alpha_{j-1}, \match{M}}$.  No arc nested under $\alpha_j$ could be among the $\alpha_1, \ldots, \alpha_{j-1}$ by hypothesis.  By the additional inductive hypothesis, we know $\alpha_j$ appears only once in the image of $\lambda_{\alpha_1, \ldots, \alpha_{j-1}, \match{M}}$.  Thus $\alpha_j$ does not appear in the image of $\lambda_{\alpha_1,\ldots, \alpha_j, \match{M}}$. Every other arc does, by the inductive hypothesis.  Additionally, the parent of $\alpha_j$ is the only additional arc that can appear repeatedly, so the claim holds by induction.
\end{proof}

We use the function $\lambda_{\alpha_1, \ldots, \alpha_j,\ldots, \match{M}}$ to define a particular subset of the cell $\cellX{cut(\match{M}, \mathcal{A})}$.

\begin{definition} \label{definition: cell for cut matching}
Given a matching $\match{M}$ and $\mathcal{A} \subseteq \match{M}$ define a subset $cut(\cellX{\match{M}}, \mathcal{A})$ of the cell $\cellX{cut(\match{M}, \mathcal{A})}$ as follows. Fix any ordering $\alpha_1, \ldots, \alpha_j$ of the arcs in $\mathcal{A}$ contravariant with respect to nesting as in Corollary~\ref{corollary: top to bottom gives unnesting}. Then $cut(\cellX{\match{M}}, \mathcal{A}) = f_{cut(\match{M}, \mathcal{A})}(\vec{v})$ where
\[v_{\beta} = \left\{ \begin{array}{ll} 
0 & \textup{ if } \lambda_{\alpha_1, \ldots, \alpha_j, \match{M}}(\beta) = 0 \textup{ and } \\
v_{\lambda_{\alpha_1, \ldots, \alpha_j, \match{M}}(\beta)} & \textup{ otherwise.}
\end{array} \right.\]
\end{definition}

\section{The closure of the Springer Schubert cell for $\match{M}$ is the union of cells obtained by cutting the arcs in $\match{M}$} \label{section: closure section}

We now show that for every one of the matchings $\match{M}'$ obtained by swapping $B$s and $T$s at endpoints of arcs in $\match{M}$(as in Lemma~\ref{lemma: closure can only swap B and T}), the Springer Schubert cell for $\match{M}'$ intersects the closure of the Springer Schubert cell for $\match{M}$.  Moreover, we will identify the intersection explicitly by cutting arcs as in Section~\ref{section: cutting arcs}. Our proof will proceed by induction, for which we need some technical language. 

\subsection{Restricting nilpotents to sub- and quotient spaces.} \label{section: subspace and quotient maps}
We establish some notation to describe certain maps induced by $X$.  Let $\mathcal{T} = \{\vec{e}_1, \ldots, \vec{e}_n\}$ and $\mathcal{B} = \{\vec{e}_{n+1}, \ldots, \vec{e}_N\}$.  Suppose that $\mathcal{T} \cup \mathcal{B}$ denotes an (ordered) basis with respect to which $X$ is in Jordan form of Jordan type $(n, N-n)$, namely with two Jordan blocks satisfying
\begin{equation} \label{equation: nilpotent formula}
X\vec{e}_1 = X \vec{e}_{n+1} = \vec{0} \textup{ while } X\vec{e}_j = \vec{e}_{j-1} \textup{ for all other } j.
\end{equation}
We refer to $\mathcal{T}$ as the \emph{top block} and $\mathcal{B}$ as the \emph{bottom block}, respectively.  

Now choose $0 \leq t \leq n$ and $0 \leq b \leq N-n$ and let $\mathcal{T}_L = \{\vec{e}_1, \ldots, \vec{e}_t\}$ and $\mathcal{B}_L = \{\vec{e}_{n+1}, \ldots, \vec{e}_{n+b}\}$.  Define $U_{b,t}$ to be the subspace spanned by $\mathcal{T}_L \cup \mathcal{B}_L$.  

Observe that $U_{b,t}$ is $X$-invariant by construction so $X$ induces a nilpotent operator on $U_{b,t}$ that we also denote $X$.  Moreover, with respect to the basis $\mathcal{T}_L \cup \mathcal{B}_L$ of $U_{b,t}$, the induced matrix $X$ is in Jordan form of Jordan type $(b,t) = (|\mathcal{T}_L|, |\mathcal{B}_L|)$ and Equation~\eqref{equation: nilpotent formula} still holds.

Similarly, the linear operator $X$ induces a well-defined nilpotent linear operator on the quotient $\mathbb{C}^N/U_{b,t}$ and we also denote this linear operator $X$.  The basis vectors 
\[\mathcal{T}_R = \{\vec{e}_{t+1}, \ldots, \vec{e}_n\} = \mathcal{T} - \mathcal{T}_L \hspace{0.2in} \textup{ and } \hspace{0.2in} \mathcal{B}_R = \{\vec{e}_{n+b+1}, \ldots, \vec{e}_N\} = \mathcal{B} - \mathcal{B}_R\]
induce a basis for the quotient space $\mathbb{C}^N/U_{b,t}$ with respect to which $X$ is in Jordan form of Jordan type $(n-t,N-b-n)=(|\mathcal{T}_R|, |\mathcal{B}_R|)$ and for which Equation~\eqref{equation: nilpotent formula} still holds.

\subsection{Restricting matchings to left and right sides} \label{section: restricting matchings to left and right}
Suppose that $\match{M}$ is a standard noncrossing matching on $\{1, 2, \ldots, N\}$ and $i$ satisfies $1 \leq i \leq N$.  Write $\match{M}_L$ for the subset $\{\alpha \in \match{M}: \term(\alpha) \leq i\}$ consisting of arcs that end at or before $i$.  Similarly, write $\match{M}_R$ for the subset $\{\alpha \in \match{M}: \init(\alpha) > i\}$ consisting of arcs that start after $i$. Note that $\match{M}$ is the disjoint union
\[\match{M} = \match{M}_L \cup \match{M}_R \cup \match{M}_i\]
where $\match{M}_i$ consists of arcs in $\match{M}$ with $\init(\alpha) \leq i < \term(\alpha)$.  

In particular, the matchings $\match{M}_L$ and $\match{M}_R$ are complementary if and only if $i$ is either the end of an arc without parent or not on an arc at all, which itself is true if and only if $|\match{M}_L| + |\match{M}_R| = |\match{M}|$.  We denote $|\match{M}| =k$ as well as $|\match{M}_L|=k_L$ and $|\match{M}_R|=k_R$. 

Define the two sets 
\[\mathcal{T}_L =  \{w_{\match{M}}(\vec{e}_{i'}) : 1 \leq i' \leq i\} \cap \mathcal{T} \hspace{0.25in} \textup{ and } \hspace{0.25in} \mathcal{B}_L =  \{w_{\match{M}}(\vec{e}_{i'}) : 1 \leq i' \leq i\} \cap \mathcal{B} \]
and let $\mathcal{T}_R = \mathcal{T} - \mathcal{T}_L$ and $\mathcal{B}_R = \mathcal{B} - \mathcal{B}_L$ to be the complements.

\begin{example} \label{example: restricting blocks isn't same as submatching blocks}
The same matching can represent two different matrices under the map $f_{\match{M}}$ of Definition~\ref{definition: springer fiber and cell} depending on the Jordan type of $X$, or equivalently the partition $\mathcal{T} \cup \mathcal{B}$.  For instance, consider the cell corresponding to the matching $\{(3,4)\}$ with respect to Jordan type $(2,2)$ and $(3,1)$ respectively.  The partition $\mathcal{T} \cup \mathcal{B}$ on the left is $\{\vec{e}_1, \vec{e}_2\} \cup \{\vec{e}_3, \vec{e}_4\}$ while on the right it is $\{\vec{e}_1, \vec{e}_2, \vec{e}_3\} \cup \{\vec{e}_4\}$. 
\[\begin{array}{ccc}
\scalebox{0.55}{
\begin{tikzpicture}
    \draw (0,0) -- (5,0);
    \filldraw[black] (1,0) circle (2pt);
    \filldraw[black] (2,0) circle (2pt);
    \draw[red] (3,0) .. controls (3.1,.7) and (3.9,.7) .. (4,0);
    \draw[red] (3.3,.5) node[anchor=south] {$a$};
    \end{tikzpicture}} & \hspace{0.5in} & \scalebox{0.55}{
\begin{tikzpicture}
    \draw (0,0) -- (5,0);
    \filldraw[black] (1,0) circle (2pt);
    \filldraw[black] (2,0) circle (2pt);
    \draw[red] (3,0) .. controls (3.1,.7) and (3.9,.7) .. (4,0);
    \draw[red] (3.3,.5) node[anchor=south] {$a$};
    \end{tikzpicture}} \\
    \left(\begin{array}{cccc}  
1 & 0 & 0 & 0 \\
0 & 0 & {\color{red} a} & 1 \\
\cdashline{1-4}
0 & 1 & 0 & 0  \\
0 & 0 & 1 & 0 \\
\end{array}\right) & &
\left(\begin{array}{cccc}  
1 & 0 & 0 & 0 \\
0 & 1 & 0 & 0 \\
0 & 0 & {\color{red} a} & 1  \\
\cdashline{1-4}
0 & 0 & 1 & 0 \\
\end{array}\right)
\end{array}\]
Note that this affects the bijection of Proposition~\ref{proposition: bijection perm matching BT}: the $\{B, T\}$-word on the left is $TBBT$ while on the right it is $TTBT$.

The set $\mathcal{T}_L$ is not the same as the set of basis vectors indexed by all $i'\leq i$ that either start arcs or are not on arcs in $\match{M}$.  For instance, suppose $\match{M}= \{(5,6), (7,8)\}$ and $i=6$.  Then $\match{M}_L = \{(5,6)\}$ and we have $\mathcal{T}_L = \{\vec{e}_1, \vec{e}_2, \vec{e}_3\}$ and $\mathcal{B}_L = \{\vec{e}_5, \vec{e}_6, \vec{e}_7\}$.  However, there are five integers starting arcs or not on arcs in $\match{M}$.  Nonetheless, the $\{B, T\}$-word corresponding to $\match{M}_L$ is the first $i$ letters of the $\{B, T\}$-word corresponding to $\match{M}$, namely $TTBBBT$ and $TTBBBBTBT$ respectively.
\end{example}

\subsection{Decomposing Springer Schubert cells when $\match{M}$ does not have the arc $(1,N)$}
In this section and in the next, our main goal is to prove for certain $i$ that the projection $pi_i$ of Proposition~\ref{proposition: projection facts} is in fact a fiber bundle when restricted to the closure of a Springer Schubert cell. This will allow us to set up an induction, which we do in the final section of the paper.  

The key step of this proof is to analyze explicit maps on matrices to confirm that the projection $\pi_i$ is a fiber bundle.

\begin{lemma} \label{lemma: block matrix maps work}
For each $j$ let $M_j$ denote the collection of $j \times j$ square matrices.  Let $i=t+b$ and define the map $\chi: M_{i} \times M_{N-i} \rightarrow M_N$ that creates a block $N \times N$ matrix $g$ out of a pair of submatrices $(g_L, g_R)$ according to the rule:
\[g = \left( \begin{array}{c|c} \textup{ first $t$ rows of }g_L & 0 \\
\cdashline{1-2} 0 & \textup{ first $n-t$ rows of }g_R \\ 
\cdashline{1-2} \textup{ last $b$ rows of }g_L & 0 \\ 
\cdashline{1-2} 0 & \textup{ last $N-n-b$ rows of }g_R\end{array} \right)\]
Then the rank of $\chi(g_L, g_R)$ is the sum of the ranks $rk(g_L) + rk(g_R)$.  Define $U$ to be
\[U = sp \langle \vec{e}_1, \vec{e}_2, \ldots, \vec{e}_t, \vec{e}_{n+1}, \vec{e}_{n+2}, \ldots, \vec{e}_{n+b} \rangle\]
and let $\mathcal{Y}$ be the subvariety of flags given by
\[\mathcal{Y} = \{gB=V_{\bullet} \in GL_N(\mathbb{C})/B : V_i = U\}.\]
Then $\chi$ induces a homeomorphism
\[\chi: GL_i(\mathbb{C})/B \times GL_{N-i}(\mathbb{C})/B \rightarrow \mathcal{Y}\]
in which the first factor $GL_i(\mathbb{C})/B$ is identified with the full flags in $U$ and $GL_{N-i}(\mathbb{C})/B$ is identified with the full flags in $\mathbb{C}^N/U$ with respect to the quotient basis induced from $\{\vec{e}_{t+1}, \ldots, \vec{e}_n, \vec{e}_{n+b+1}, \ldots, \vec{e}_N\}$.  Moreover the map $\chi$ on flags restricts to Schubert cells with:
\[\chi(\mathcal{C}_{w_L}, \mathcal{C}_{w_R}) \subseteq \mathcal{C}_{\chi(w_L, w_R)}.\]
\end{lemma}

\begin{proof}
Considered as a map on all matrices, the linear transformation $\chi$ is an isomorphism onto its image, which is precisely the collection of $N \times N$ block matrices with zeroes in the indicated blocks.  The rank of $\chi$ is additive because these are block matrices.  Thus the image of $\chi$ when restricted to pairs of invertible matrices $GL_i(\mathbb{C}) \times GL_{N-i}(\mathbb{C})$ is entirely contained within $GL_N(\mathbb{C})$.  Moreover, the map $\chi$ is compatible with multiplication by elements of $B$ in the sense that if $b_L \in GL_i(\mathbb{C})$ and $b_R \in GL_{N-i}(\mathbb{C})$ are both upper-triangular then for all $(g_L, g_R)$ we have
\[\chi(g_Lb_L, g_Rb_R) = \chi(g_L, g_R) \left( \begin{array}{c|c} b_L & 0 \\ \cline{1-2} 0 & b_R \end{array} \right)\]
where the matrix on the right is upper-triangular in each diagonal block. Thus $\chi$ descends to a map on flags.  

Suppose $g_L, g_R$ are both in the canonical form of Definition~\ref{definition: canonical representative and Schub cells} within $GL_i(\mathbb{C}), GL_{N-i}(\mathbb{C})$ respectively.  By definition of canonical form, the image $\chi(g_L, g_R)$ is also in canonical form in $GL_N(\mathbb{C})$.  We conclude that the map $\chi$ is injective on the product $\mathcal{C}_{w_L} \times \mathcal{C}_{w_R}$ of Schubert cells in $GL_i(\mathbb{C})/B \times GL_{N-i}(\mathbb{C})/B$ and moreover the image is a subset of the Schubert cell for $\chi(w_L, w_R)$ in $GL_N(\mathbb{C})/B$ --- possibly but not necessarily a proper subset. In particular, the map $\chi$ is injective on $GL_i(\mathbb{C})/B \times GL_{N-i}(\mathbb{C})/B$.

We now show that the image of $\chi$ is precisely the collection of flags with $V_i=U$. First choose an arbitrary element $(g_L, g_R) \in GL_i(\mathbb{C})/B \times GL_{N-i}(\mathbb{C})/B$.  We will show the image $\chi(g_L, g_R)$ is in $\mathcal{Y}$.  By construction, we know that the first $i=t+b$ columns of $\chi(g_L, g_R)$ lie in the span of $U$.  Also by construction, the subspace $U$ is $i$-dimensional so in fact the span of the first $i$ columns of $\chi(g_L, g_R)$ is precisely $U$ as desired.  Conversely, if $V_{\bullet} =gB \in GL_N(\mathbb{C})/B$ is a flag with $V_i=U$ then the first $i$ columns of every matrix representative of the coset $gB$ are zero in rows $t+1, \ldots, n$ and $n+b+1, \ldots, N$.  Suppose $g$ is any such matrix.  Then since $g$ is invertible, there is an upper-triangular invertible matrix $b$ that eliminates all entries in the last $N-i$ columns in rows whose pivots are in the first $i$ columns.  In other words, the matrix $gb$ is in the image of $\chi$.  This shows that $\chi$ surjects onto $\mathcal{Y}$.

Finally we show that $\chi$ is a homeomorphism onto $\mathcal{Y}$.  Note that the space of full flags in $U$ is homeomorphic to $GL_i(\mathbb{C})/B$.  Both of these are homeomorphic to the image $\pi_i(\mathcal{Y})$ under the projection $\pi_i: GL_N(\mathbb{C})/B \rightarrow GL_N(\mathbb{C})/P_i$ where $P_i$ consists of the $N \times N$ invertible matrices that are zero below the diagonal in the first $i$ columns.  Similarly, the space of full flags in $sp \langle \vec{e}_{t+1}, \ldots, \vec{e}_n, \vec{e}_{n+b+1}, \ldots, \vec{e}_N\rangle$ is homeomorphic to both the space of full flags in $\mathbb{C}^N/U$ and to $GL_{N-i}(\mathbb{C})/B$.  All of these are homeomorphic to the image of $\mathcal{Y}$ under the projection $\pi_i^c: GL_N(\mathbb{C})/B \rightarrow GL_N(\mathbb{C})/P_i^c$ where $P_i^c$ consists of $N \times N$ invertible matrices that are zero to the left of the diagonal in the last $N-i$ rows.  Both $\pi_i$ and $\pi_i^c$ are closed continuous maps on $GL_N(\mathbb{C})/B$ by Proposition~\ref{proposition: projection facts} and hence are also closed and continuous on the closed subvariety $\mathcal{Y}$.  The map $\chi$ is inverse to $\pi_i \times \pi_i^c$ restricted to $\mathcal{Y}$.  A bijective closed continuous map is a homeomorphism, so $\chi$ is a homeomorphism.  This proves the claim.
\end{proof}

\begin{example}
For instance, suppose $\match{M} = \{(3,4)\}$ considered as a matching on $\{1, 2, \ldots, 8\}$ with $|\mathcal{T}|=|\mathcal{B}|=4$.  When $i=4$ we get the following matchings $\match{M}_L$ and $\match{M}_R$ which we show together with the associated Springer Schubert cells $\cellX{\match{M}_L}, \cellX{\match{M}_R},$ and $\cellX{\match{M}}$.  
\[\begin{array}{cccc}
\scalebox{0.55}{
\begin{tikzpicture}
    \draw (0,0) -- (5,0);
    \filldraw[black] (1,0) circle (2pt);
    \filldraw[black] (2,0) circle (2pt);
    \draw[red] (3,0) .. controls (3.1,.7) and (3.9,.7) .. (4,0);
    \draw[red] (3.3,.5) node[anchor=south] {$a$};
    \end{tikzpicture}} & 
     \scalebox{0.55}{
\begin{tikzpicture}
    \draw (0,0) -- (5,0);
    \filldraw[black] (1,0) circle (2pt);
    \filldraw[black] (2,0) circle (2pt);
    \filldraw[black] (3,0) circle (2pt);
    \filldraw[black] (4,0) circle (2pt);
    \end{tikzpicture}} & \hspace{0.5in} &
 \scalebox{0.55}{
\begin{tikzpicture}
    \draw (0,0) -- (9,0);
    \filldraw[black] (1,0) circle (2pt);
    \filldraw[black] (2,0) circle (2pt);
    \filldraw[black] (5,0) circle (2pt);
    \filldraw[black] (6,0) circle (2pt);
    \filldraw[black] (7,0) circle (2pt);
    \filldraw[black] (8,0) circle (2pt);
    \draw[red] (3,0) .. controls (3.1,.7) and (3.9,.7) .. (4,0);
    \draw[red] (3.3,.5) node[anchor=south] {$a$};
    \end{tikzpicture}} \\
    \left(\begin{array}{cccc}  
1 & 0 & 0 & 0 \\
0 & 1 & 0 & 0 \\
0 & 0 & {\color{red} a} & 1 \\
\cdashline{1-4}
0 & 0 & 1 & 0 \\
\end{array}\right)
    &  \left(\begin{array}{cccc}  
1 & 0 & 0 & 0 \\
\cdashline{1-4}
0 & 1 & 0 & 0 \\
0 & 0 & 1 & 0  \\
0 & 0 & 0 & 1 \\
\end{array}\right)
     &  &  \left(\begin{array}{cccccccc}  
1 & 0 & 0 & 0 & 0 & 0 & 0 & 0 \\
0 & 1 & 0 & 0 & 0 & 0 & 0 & 0 \\
0 & 0 & {\color{red} a} & 1 & 0 & 0 & 0 & 0 \\
0 & 0 & 0 & 0 & 1 & 0 & 0 & 0 \\
\cdashline{1-8}
0 & 0 & 1 & 0 & 0 & 0 & 0 & 0 \\
0 & 0 & 0 & 0 & 0 & 1 & 0 & 0 \\
0 & 0 & 0 & 0 & 0 & 0 & 1 & 0 \\
0 & 0 & 0 & 0 & 0 & 0 & 0 & 1 \\
\end{array}\right)
    \end{array}
    \]
    In this case, the map $\chi$ inserts the data from $\cellX{\match{M}_L}$ into the first three and fifth rows of $\cellX{\match{M}}$ and the data from $\cellX{\match{M}_R}$ into the fourth and last three rows.
\end{example}

Next, we analyze the map $\chi$ on permutation matrices and identify the Springer Schubert cell in the image of $\chi$.

\begin{corollary} \label{corollary: restricting to left partition gives same BT sequence}
Fix $i=t+b$ for $t \leq n$ and $b \leq N-n$.  Let $X$ be a nilpotent matrix in Jordan canonical form of Jordan type $(n, N-n)$.  

Suppose $w_{\match{M}}B \in \mathcal{S}_X$ and $\match{M}_L, \match{M}_R$ are constructed from $\match{M}$ as in Section~\ref{section: restricting matchings to left and right}.  Then 
\[\chi(w_{\match{M}_L},w_{\match{M}_R}) = w_{\match{M}}.\]
In particular, if the $\{B, T\}$-word for $\match{M}$ is $m_1 m_2 \cdots m_n$ then the $\{B, T\}$-word for $\match{M}_L$ is $m_1 m_2 \cdots m_i$ and the $\{B, T\}$-word for $\match{M}_R$ is $m_{i+1} m_{i+2} \cdots m_N$. 

Conversely, if $w_{\match{M}_L}B \in \mathcal{S}_X^i, w_{\match{M}_R}B \in \mathcal{S}_X^{N-i}$ are permutations in the Springer fibers of Jordan type $(t,b)$ and $(n-t, N-n-b)$ respectively, with $\{B, T\}$-words denoted $m_1m_2 \cdots m_i$ and $m_{i+1}m_{i+2}\cdots m_N$ respectively, then the $\{B, T\}$-word of $\chi(w_{\match{M}_L},w_{\match{M}_R})$ is $m_1 \cdots m_im_{i+1} \cdots m_N$.  The matching $\match{M}$ contains all arcs in $\match{M}_L \cup \match{M}_R$ plus possibly additional arcs over $i$ created by matching some of the last integers not on arcs in $\match{M}_L$ with some of the first integers not on arcs in $\match{M}_R$.
\end{corollary}

\begin{proof}
First note that Lemma~\ref{lemma: block matrix maps work} showed that the permutation $\chi(w_{\match{M}_L},w_{\match{M}_R})$ has pivots in the first $n$ rows in exactly those columns where $w_{\match{M}_L}$ has pivot in the first $t$ rows or $w_{\match{M}_R}$ has pivot in the first $n-t$ rows. All other pivots are in the last $N-n$ rows.  Thus the $\{B, T\}$-word of the image is the concatenation of the $\{B, T\}$-words of $\match{M}_L$ and $\match{M}_R$.

Now we confirm the converse is true.  By construction, all arc-ends in $\match{M}_L$ are also arc-ends in the first $i$ integers of $\match{M}$ so the corresponding columns of $w_{\match{M}_L}$ and $w_{\match{M}}$ all have pivots in the first $t$ rows (respectively arc-starts and bottom block).  If an integer $j \leq i$ is not on an arc in $\match{M}$ then also $j$ is not on an arc in $\match{M}_L$.  The substring that corresponds to the integers not on arcs in the $\{B, T\}$-word for a matching is $T^{\ell_1}B^{\ell_2}$ by Corollary~\ref{corollary: standard means perfect plus T then B}.  The only possibility within the first $i$ integers in $\match{M}$ are  arcs that start in the first $i$ integers and end in the last $N-i$ integers. In $\match{M}_L$ none of these integers are on arcs.  Since $\match{M}$ is standard, all integers not on arcs in $\match{M}$ occur before the first arc over $i$ in $\match{M}$.  Thus the substring of the $\{B, T\}$-word for $\match{M}$ corresponding to integers not on arcs in $\match{M}_L$ is $T^{\ell_1}B^{\ell_2+\ell_3}$ where $\ell_1+\ell_2$ is the number of integers not on arcs in the first $i$ integers of $\match{M}$ and $\ell_3$ is the number of arcs starting before $i+1$ and ending after $i$. Also $\ell_1 = t - |\match{M}_L|$ is the same in both $\match{M}$ and $\match{M}_L$ so both have $T$s in the same positions in the first $i$ letters of their $\{B, T\}$-words.  Since all the rest of the letters must be $B$, we conclude that the $\{B, T\}$-word for $\match{M}_L$ is the same as the first $i$ letters of that for $\match{M}$.  A similar argument holds when comparing arc-starts and arc-ends in $\match{M}_R$ and the last $N-i$ integers in $\match{M}$. 

In summary, we showed that the pivots in the first $n$ rows of $\chi(w_{\match{M}_L}, w_{\match{M}_R})$ and $w_{\match{M}}$ are in the same columns, and similarly for the last $N-n$ rows. Lemma~\ref{lemma: pivots of matching perm increase in rows} shows that the pivots increase in each block for the permutation flags in Springer fibers, so as desired we have
\[\chi(w_{\match{M}_L}, w_{\match{M}_R}) = w_{\match{M}}.\]
\end{proof}

Finally we identify the image of $\chi$ within the Springer Schubert cells of $\mathcal{S}_X$.

\begin{corollary} \label{corollary: restricting to left partition gives same f}
Fix $i=t+b$ and let $U$ be as in Lemma~\ref{lemma: block matrix maps work}. Let $X$ be a nilpotent matrix in Jordan canonical form of Jordan type $(n, N-n)$ and also denote by $X$ the induced maps on $U$ and on the quotient $\mathbb{C}^N/U$ as in Section~\ref{section: subspace and quotient maps}, with $\mathcal{S}_X^i \subseteq GL_i(\mathbb{C})/B$ and $\mathcal{S}_X^{N-i} \subseteq GL_{N-i}(\mathbb{C})/B$ denoting the respective Springer fibers.  Then $\chi$ induces a homeomorphism
\[\chi: \mathcal{S}_X^i \times \mathcal{S}_X^{N-i} \rightarrow \mathcal{S}_X \cap \mathcal{Y}.\]
Suppose $\match{M}_L, \match{M}_R$ are standard noncrossing matchings on $i, N-i$ that index $w_{\match{M}_L}B \in \mathcal{S}_X^i$ and $w_{\match{M}_R}B \in \mathcal{S}_X^{N_i}$ respectively.  Then the following map commutes:
\[\begin{tikzcd}
    \mathbb{C}^{k_L} \times \mathbb{C}^{k_R} \arrow[r, hook] 
    \arrow[d, "f_{\match{M}_{L}} \times f_{\match{M}_{R}}" ']  & \mathbb{C}^k  \arrow[d, "f_\match{M}"]  \\ 
    \cellX{\match{M}_L} \times \cellX{\match{M}_R} \arrow[r, hook, "\chi"] & \cellX{\match{M}}
\end{tikzcd}\]
where $k_L$ is the number of arcs in $\match{M}_L$ (respectively $k_R, \match{M}_R$ and $k, \match{M}$).

Moreover the image $\chi(\cellX{\match{M}_L}, \cellX{\match{M}_R})$ is the subspace of $\cellX{\match{M}}$ defined by setting $v_{\init(\alpha)} = 0$ whenever $\alpha$ is an arc in $\match{M}$ that starts at or before $i$ and ends to the right of $i$.
\end{corollary}

\begin{proof}
Both $\mathcal{S}_X$ and $\mathcal{Y}$ are closed in $GL_N(\mathbb{C})/B$ so the homeomorphism of Lemma~\ref{lemma: block matrix maps work} restricts to a homeomorphism onto $\mathcal{S}_X \cap \mathcal{Y}$ as long as $\chi$ is well-defined and surjective.  

Corollary~\ref{corollary: restricting to left partition gives same BT sequence} showed that the map $\chi$ surjects onto the permutation flags $w_{\match{M}}B \in \mathcal{S}_X \cap \mathcal{Y}$.   

Now we identify the image $\chi(\cellX{\match{M}_L} \times \cellX{\match{M}_R})$ inside $\cellX{\match{M}}$.  In fact, since we showed $\chi$ sends pivots to pivots within these cells, we only need to compare the ordered sequence of ancestors of each arc in each matching.  We already observed that $\match{M}$ is the disjoint union
\[\match{M} = \match{M}_L \cup \match{M}_R \cup \{\alpha \in \match{M}: \init(\alpha) \leq i \textup{ and } \term(\alpha) > i\}.\]
If $\init(\alpha) \leq i$ and $\term(\alpha)>i$ and $\beta$ is an ancestor of $\alpha$ then by definition of ancestors, we also know $\init(\beta) \leq i$ and $\term(\beta)>i$.  Suppose the ordered sequence of ancestors of an arc $\alpha \in \match{M}$ is
\[\init(\alpha), \anc{\init(\alpha)}{}, \anc{\init(\alpha)}{2}, \ldots, \anc{\init(\alpha)}{k-1}\]
and suppose $\alpha \in \match{M}_L$ as well, where it has $j$ ancestors.  Then $j \leq k-1$ and the last $k-1-j$ ancestors of $\alpha$ in $\match{M}$ all end to the right of $i$. In other words, corresponding nonzero columns of $f_{\match{M}_L}(\vec{v}) - {w}_{\match{M}_L}$ and $f_{\match{M}} - w_{\match{M}}$ are
\[ \sum_{\ell = 1}^{j} v_{\anc{\alpha}{\ell-1}} \vec{e}_{r_0+j} \hspace{0.2in} \textup{ and } \hspace{0.2in} \left( \sum_{\ell = 1}^{j} v_{\anc{\alpha}{\ell-1}} \vec{e}_{r_0+j}\right) + \sum_{\ell > j} v_{\anc{\alpha}{\ell-1}} \vec{e}_{r_0+j}.\]
In particular if $\beta$ is one of the last $k-1-j$ ancestors of $\alpha$ in $\match{M}$ then the variable $v_{\beta}=0$ everywhere in the image $\chi(\cellX{\match{M}_L}, \cellX{\match{M}_R})$. A similar argument holds for $\match{M}_R$.  

These are the only relations imposed on the image because if $\alpha \in \match{M}_L$ has $j$ ancestors in $\match{M}_L$ then also there are at least $j$ arc-ends in the interval $[\init(\alpha)+1, \ldots, i]$.  Thus all entries that are not identically zero in column $\init(\alpha)$ of the matrix $f_{\match{M}_L}-w_{\match{M}_L}$ are within the first $t$ rows, and similarly for $\match{M}_R$.  This completes the proof.
\end{proof}

\subsection{Decomposing Springer Schubert cells when $\match{M}$ contains $(1,N)$}
We can extend this argument slightly when $V_i$ is not (necessarily) the span of basis vectors, more specifically for the case $i=1$.  Our argument follows a similar outline: identify an explicit matrix map and use it to confirm that the projection $\pi_1$ restricts to a fiber bundle on the closure of the Springer Schubert cells.  

\begin{lemma} \label{lemma: restricting fiber bundle to project to line}
Suppose that $\match{M}$ is a standard noncrossing matching on $\{1, 2, \ldots, N\}$ with arc $(1, N)$ and $X$ is nilpotent of Jordan type $(n,N-n)$.  Then 
\begin{itemize}
    \item the matching $\match{M}$ is a perfect noncrossing matching,
    \item the projection $\pi_1: \overline{\cellX{\match{M}}} \rightarrow V_1$ surjects onto the collection of lines in the kernel of $X$,
    \item the preimage over every $V_1 \subseteq \ker X$ is contained in the collection of full flags in the quotient space
    \[X^{-(N-2)/2}(V_1)/V_1\]
    that are in the Springer fiber induced by $X$ as in Section~\ref{section: subspace and quotient maps}.
\end{itemize} 
\end{lemma}

\begin{proof}
Since $\match{M}$ is standard, every integer under the arc $(1,N)$ is itself on an arc.  Thus the matching has $N/2$ arcs and is perfect as well as noncrossing.  This proves the first claim.

Next note that $V_1 \subseteq \ker X$ by definition of Springer fibers.  Since $X$ has Jordan type $(N/2, N/2)$ the kernel is the $2$-dimensional subspace $sp \langle \vec{e}_1, \vec{e}_{N/2+1} \rangle$.  The image $\pi_1(\cellX{\match{M}})$ is the open dense set consisting of all lines except the line spanned by $\vec{e}_1$.  The projection $\pi_1$ is a closed map by Proposition~\ref{proposition: projection facts}.  Thus the image $\pi_1(\overline{\cellX{\match{M}}})$ must contain all lines in $\ker N$.  This is homeomorphic to $\mathbb{P}^1$ as claimed.

Now we analyze the fibers of the map $\pi_1$. Since $\match{M}$ is a perfect matching with arc $(1, N)$ there must be $(N/2)-1$ arcs under $(1,N)$.  By construction of $X$ this further implies that $X$ has Jordan type $(N/2, N/2)$. Lemma~\ref{lemma: closure condition under arcs} implies that every flag $V_{\bullet} \in \overline{\cellX{\match{M}}}$ satisfies
\[X^{(N/2)-1}V_{N-1} \subseteq V_1\]
Hence $V_{N-1} \subseteq X^{-(N-2)/2}(V_1)$.  Since $X$ has two Jordan blocks of the same size, we know that $\dim X^{-(N-2)/2}(V_1)$ is $(N-2+1)$-dimensional.  Comparing dimensions, we have
\[V_{N-1} = X^{-(N-2)/2}(V_1)\]
for every flag $V_{\bullet}$ in the closure of the Springer Schubert cell $\cellX{\match{M}}$.  Any full flag in the Springer fiber of the linear operator $X$ induced on the $N-2$-dimensional vector space 
\[X^{-(N-2)/2}(V_1)/V_1\]
as in Section~\ref{section: subspace and quotient maps} gives rise to a flag in $\cellX{\match{\match{M}}}$ via the invertible map that sends each subspace $V_i'$ in the quotient to the $(i+1)$-dimensional subspace $V_i'+V_1$ in the full flag variety, and inserts the fixed subspaces $V_1$ and $V_{N-1}=X^{-(N-2)/2}(V_1)$ at start and end.  So the preimage of each flag in the image of $\pi_1(\overline{\cellX{\match{M}}})$ is contained in the Springer fiber of Jordan type $(N/2-1, N/2-1)$ in the full flag variety $GL_{N-2}(\mathbb{C})/B$. This completes the proof.
\end{proof}

Our next goal is to identify explicitly the fiber over each $V_1 \subseteq \ker X$ in the total space $\overline{\cellX{\match{M}}}$ with respect to the map $\pi_1$.  We do this by writing and analyzing a map of matrices.

\begin{lemma} \label{lemma: explicit matrices for proj to v1}
For each $a \in \mathbb{C}$ let $T_a$ be the $N \times N$-matrix that sends
\begin{itemize}
    \item $\vec{e}_j \mapsto \vec{e}_j$ for all $j \leq N/2$
    \item $\vec{e}_{j+N/2} \mapsto \vec{e}_{j+N/2} + a\vec{e}_{j}$ for all $1 \leq j \leq N/2$
\end{itemize}
namely $T_a$ multiplies on the left by the $N \times N$ block matrix $\left( \begin{array}{c|c} I_{N/2} & a I_{N/2} \\ \cline{1-2} 0 & I_{N/2} \end{array} \right)$.  Identify $\ker X$ with $\mathbb{P}^1 = \mathbb{C} \cup \infty$ by the rule that the line spanned by $\vec{e}_{N/2+1}+a\vec{e}_1 \mapsto a$ and the line spanned by $\vec{e}_1 \mapsto \infty$. Define a map on matrices $\varphi: \mathbb{P}^1 \times M_{N-2} \rightarrow M_N$ by the rule that for $a \in \mathbb{C}$
\[\varphi(a, g) = T_a \left( \begin{array}{c|c|c} 0 & \textup{first $N/2-1$ } & 0 \\
 0 & \textup{rows of } g & 0 \\
\cline{1-3} 0 & 0 & 1 \\
\cdashline{1-3} 1 & 0 & 0 \\
\cline{1-3} 0 & \textup{last }N/2-1 & 0 \\
0 & \textup{ rows of }g & 0
\end{array} \right) \hspace{0.5in} \textup{ and } \hspace{0.5in} \varphi(\infty, g) = \left( \begin{array}{c|c|c}  1 & 0 & 0 \\
\cline{1-3} 0 & g & 0 \\
\cline{1-3} 0 & 0 & 1\end{array} \right)\]
is the block diagonal matrix with blocks of size $1, N-2, 1$ along the diagonal.

Let $\mathcal{Y}' \subseteq GL_N(\mathbb{C})/B$ denote the collection of flags satisfying
\[\mathcal{Y}'= \{ \widetilde{g}B: sp \langle \vec{g}_1, X^{N/2-1}\vec{g}_N \rangle = \ker X\}\]
where $\vec{g}_i$ denotes the $i^{th}$ column of the matrix $\widetilde{g}$. Moreover, suppose $\mathcal{S}_X^{N-2}$ denotes the Springer fiber of type $(N/2-1, N/2-1)$  in $GL_{N-2}(\mathbb{C})/B$ respectively $\mathcal{S}_X^N$ and $(N/2,N/2)$ and $GL_N(\mathbb{C})/B$.  Then 
\[\varphi: \mathbb{C}^1 \times \mathcal{S}_X^{N-2} \rightarrow \mathcal{S}_X^N \cap \mathcal{Y}'\]
is a local homeomorphism that commutes with the fiber bundle $\pi_1: GL_N(\mathbb{C})/B \rightarrow \{\textup{lines in } \mathbb{C}^N\}$.
\end{lemma}

\begin{proof}
We show the following things, in order: 1) the map $\varphi$ sends invertible matrices to invertible matrices; 2) the map $\varphi$ sends all matrices in the same $B$-coset to a single $B$-coset in the image; 3) when restricted to the open sets in the domain defined by $a \neq 0$ and $a \neq \infty$ the map $\varphi$ is injective and almost surjective; and 4) the map $\varphi$ is a homeomorphism from Springer fiber to $\mathcal{S}_X^N \cap \mathcal{Y}'$.

Suppose first that $a \in \mathbb{C}$.  Define the set of vectors
\[\mathcal{B}_a = \{\vec{e}_{N/2+1} + a\vec{e}_1, \vec{e}_{N/2+2}+a\vec{e}_2, \ldots, \vec{e}_N + a\vec{e}_N\}.\]
Note that $\mathcal{T} \cup \mathcal{B}_a$ form another basis with respect to which $X$ is in Jordan canonical form.  The operator $T_a$ changes basis between $\mathcal{T} \cup \mathcal{B}_a$ and $\mathcal{T} \cup \mathcal{B}$.    So by construction, if $g$ is invertible then $\varphi(a,g)$ is invertible for all $a \in \mathbb{P}^1$.  Thus $\varphi$ induces a map $\mathbb{P}^1 \times GL_{N-2}(\mathbb{C}) \rightarrow GL_N(\mathbb{C})$. 

Now suppose $b \in GL_{N-2}(\mathbb{C})$ is upper-triangular and let $\widetilde{b}$ be the block diagonal matrix in $GL_N(\mathbb{C})$ with diagonal blocks $1, b, 1$.  Since $T_a$ acts by left-multiplication, we have
\[\varphi(a,gb) = \varphi(a,g)\widetilde{b}\]
when $a \in \mathbb{C}$. The same equation holds for $a = \infty$ by direct computation.  Thus $\varphi$ induces a map on flag varieties.

If $\varphi(a,g)B$ is a flag with $a \neq 0$ then it can be described as
\[\varphi(a,g)B = T'_{a'}  \left( \begin{array}{c|c|c}  1 & 0 & 0 \\
\cline{1-3} 0 & g & 0 \\
\cline{1-3} 0 & 0 & 1\end{array} \right) \hspace{0.25in} \textup{ where } T'_{a'} = \left( \begin{array}{c|c} I_{N/2} & 0\\ \cline{1-2} a' I_{N/2}  & I_{N/2} \end{array} \right) \hspace{0.2in} \textup{ and } a'=\frac{1}{a}\]
Indeed, when $aa'=1$ the maps $T'_{a'}$ and $T_a$ agree and their image consists of flags with $V_1$ spanned by neither $\vec{e}_1, \vec{e}_{1+N/2}$. Multiplication by an invertible linear operator is injective so $\varphi$ is injective both on the open set $a \neq 0$ and on the open set $a \neq \infty$. 

We now show that the $\varphi(\mathbb{P}^1 \times \mathcal{S}_X^{N-1})$ is precisely $\mathcal{S}_X^N \cap \mathcal{Y}'$. First note that if $\widetilde{g}B=V_{\bullet}$ is a flag in $\mathcal{S}_X^N$ then $V_1$ is in $\ker X$ and so is either spanned by $\vec{e}_{1+N/2}$ or by $\vec{e}_1$.  If $\widetilde{g}$ is in canonical form then all entries in columns $2, \ldots, N$ of row $1+N/2$ respectively $1$ are zero.  If $\widetilde{g}B$ is also in $\mathcal{Y}'$ then the pivot in column $N$ of $\widetilde{g}$ is in the row $N/2$ or $N$ respectively. So the first and last columns of all matrices $\widetilde{g}$ agree with the image of $\varphi$. We must confirm that the entries in pivot rows also agree with $\varphi$ and so we analyze the Springer condition. 

Let $\vec{v}_1, \ldots, \vec{v}_N$ be linearly independent vectors in $\mathbb{C}^N$. Both $T_a$ and $X$ are linear so
\[X(T_a\left(\sum_{j=1}^{i-1} c_j \vec{v}_j\right) = \sum_{j=1}^{i-1} c_j XT_a \vec{v}_j.\]
So $X\vec{v}_i$ is a linear combination of the vectors $\vec{v}_1, \ldots, \vec{v}_{i-1}$ if and only if $XT\vec{v}_i$ is a linear combination of the vectors $T_a(\vec{v}_1), \ldots, T_a(\vec{v}_{i-1})$.  So if $a \in \mathbb{C}$ and $gB$ is in $\mathcal{S}_X^{N-2}$ then $\varphi(a,g)B$ is in $\mathcal{S}_X^N$ and -- by construction of $\varphi$ -- also in $\mathcal{Y}'$. Conversely if $\widetilde{g}$ is in canonical form with $\widetilde{g}B \in \mathcal{S}_X^N \cap \mathcal{Y}'$ and if the first column of $\widetilde{g}$ is not $\vec{e}_1$ then acting $\widetilde{g}$ by the inverse of $T_a$ changes basis to obtain a matrix with first column $\vec{e}_{N/2}$ and middle $N-2$ columns satisfying $X^{(N-2)/2}V_{N-1} \subseteq V_1$ by Lemma~\ref{lemma: closure condition under arcs}.  In particular, the middle $N-2$ columns are zero in row $N/2$.  In other words, the middle $N-2$ columns come from a flag $gB \in \mathcal{S}_X^{N-2}$ as desired.

This proves that for $a \in \mathbb{C}$ the flag $\varphi(a,g)B$ is in the Springer fiber $\mathcal{S}_X^N$ if and only if $gB$ is in $\mathcal{S}_X^{N-2}$.  Replacing $T_a$ by the identity matrix gives the analogous result for $a = \infty$. 

Finally, our description of $T_a$ and $T'_{a'}$ implies that $\varphi$ induces surjective invertible maps from the open sets $(a \neq 0) \times \mathcal{S}_X^{N-2}$ and $(a \neq \infty) \times \mathcal{S}_X^{N-2}$ onto their images in $\mathcal{S}_X^N \cap \mathcal{Y}'$.  Thus $\varphi$ is a bijection that commutes with the fiber bundle $\pi_1: GL_N(\mathbb{C})/B \rightarrow GL_1(\mathbb{C})/P_1$ by identifying the latter with $\mathbb{P}^1 \cong \ker X$ as above.  Since the fiber bundle is a closed continuous map, we conclude that in fact $\varphi$ is continuous and thus a homeomorphism.  This completes our proof. 
\end{proof}

\begin{corollary}\label{corollary: springer cell homeomorphism for (1,N)}
Denote the Springer fiber in $GL_{N-2}(\mathbb{C})/B$ for the Jordan canonical matrix $X$ of Jordan type $(N/2-1, N/2-1)$ by $\mathcal{S}_X^{N-2}$ and similarly for $\mathcal{S}_X^N$.  Identify $\mathbb{P}^1$ with $\ker X$ as in Lemma~\ref{lemma: explicit matrices for proj to v1}. Let $\match{M}'$ index a Springer Schubert cell $\cellX{\match{M}'} \subseteq \mathcal{S}_X^{N-2}$ with $\{B, T\}$-word $m_2'm_3'\cdots m_{N-1}'$.  Then for each such $\match{M}'$ the map $\varphi$ is a homeomorphism onto the cell
\[\varphi: \mathbb{P}^1 \times \cellX{\match{M}'} \rightarrow \cellX{\match{M}}\]
where if $a \in \mathbb{C}$ then the $\{B, T\}$-word for $\match{M}$ is $B m_2' \cdots m_{N-1}' T$ while if $a = \infty$ then it is $T m_2' \cdots m_{N-1}' B$. As matchings:
\begin{itemize}
    \item If $a = \infty$ then $\match{M} = \match{M}'$. 
    \item If $a \neq \infty$ and $\match{M}'$ is perfect then $\match{M} = \match{M}' \cup \{(1,N)\}$.
    \item If $a \neq \infty$ and the first and last integer not on arcs in $\match{M}'$ are $j_1, j_2$ respectively, then $\match{M} = \match{M}' \cup \{(1,j_1), (j_2,N)\}$.
\end{itemize}   
\end{corollary}

\begin{proof}
Restricting the homeomorphism $\varphi$ gives a homeomorphism on each Springer Schubert cell $\cellX{\match{M}'}$ by Lemma~\ref{lemma: explicit matrices for proj to v1}. Whether $a = \infty$ or $a \neq \infty$ determines the pivot in the first column and thus the first letter of the $\{B, T\}$-word for $\match{M}$.  We now determine the other $N-1$ letters. The map $T_a$ only changes entries in the top $N/2$ rows in columns whose pivots are in the bottom $N/2$ rows.  Otherwise $T_a$ shifts pivots from rows $N/2, \ldots, N-2$ to rows $N/2+2, \ldots, N$ in the lower block. So the pivot entry in the middle $N-2$ columns of the image $\varphi(\cdot, g)$ is in the top block if and only if the same is true for the corresponding column of $g$, and similarly for the bottom block.  In other words, the $\{B, T\}$-word for $\match{M}'$ consists of the middle $N-2$ letters of that for $\match{M}$ with first and last pivots coming from whether $a \in \mathbb{C}$ or $a = \infty$ as stated. 

Now we identify the matching $\match{M}$.  Recall that if $\match{M}'$ is not a perfect matching then the subword of letters associated to non-arcs is $T^jB^j$ for some $j$ by Corollary~\ref{corollary: standard means perfect plus T then B} together with our hypothesis that $N$ is even.  Apply Proposition~\ref{proposition: bijection perm matching BT} to construct the matching for $\match{M}$ from its $\{B, T\}$-word. The first place this matching can differ from that of $\match{M}'$ is in the first $T$ in $m_2'm_3' \cdots m_{N-1}'$ that the matching algorithm does not assign to an arc.  If $j=0$ then there is no such $T$ so the matching algorithm assigns to $\match{M}$ the arc $(1,N)$ from the first and last letters in its $\{B, T\}$-word.  Else suppose this $T$ is in position $j_1$.  Then $(1,j_1) \in \match{M}$.  By an analogous argument, there is a $B$ that is not assigned to an arc in $\match{M}'$ and the rightmost such, say in position $j_2$, is assigned the arc $(j_2,N) \in \match{M}$.  The matching algorithm then ends the same on both $\match{M}$ and $\match{M}'$.  This proves the claim.
\end{proof}

\subsection{Completing an inductive proof of the main theorem}

We now complete the proof of our main theorem.  Figure~\ref{figure: all closures n=4} gives an example with cell closures for $n=4$.

\begin{theorem} \label{theorem: closure main theorem}
Let $X$ be a nilpotent matrix in Jordan form of Jordan type $(n,N-n)$.  Suppose that $\match{M}$ is a standard noncrossing matching on $\{1, 2, \ldots, N\}$ with associated Springer Schubert cell $\cellX{\match{M}}$.  The closure $\overline{\cellX{\match{M}}}$ is the disjoint union
\[ \overline{\cellX{\match{M}}} = \bigcup_{\mathcal{A} \subseteq \match{M}} cut(\cellX{\match{M}}, \mathcal{A})\]
where $cut(\cellX{\match{M}},\mathcal{A})$ denotes the subspace of the Springer Schubert cell associated to the matching obtained by cutting the arcs $\mathcal{A}$ inside the matching $\match{M}$.
\end{theorem}

\begin{proof}
Our proof proceeds by induction on $N$.  The base cases are when $N=1$ and $N=2$.  In the first case, the matching $\match{M} = \emptyset$ and the Springer fiber consists of a single flag so the claim is vacuously true.  In the second case, the matching is either $\emptyset$ or $\{(1,2)\}$.  As before, the closure of a Springer Schubert cell consisting of a single flag is itself so the claim holds for the Springer Schubert cell corresponding to $\match{M} = \emptyset$.  The other Springer Schubert cell is the open dense set consisting of all lines spanned by $\vec{e}_2 + a\vec{e}_1$ for $a \in \mathbb{C}$ inside the Springer fiber (which in this case is  homeomorphic to $\mathbb{P}^1$). Thus the base cases hold.

Now consider the Springer Schubert cell $\cellX{\match{M}}$ corresponding to the nilpotent matrix $X$ in Jordan form of Jordan type $(n,N-n)$.  For the inductive step, we have two cases.  First suppose that $\match{M}$ does not contain the arc $(1,N)$.  Then there exists an $i$ that is either the end of an arc without parent or not on an arc at all in $\match{M}$. Corollary~\ref{corollary: restricting to left partition gives same f} proved that the closure $\overline{\cellX{\match{M}}}$ is homeomorphic to the closure
\[\overline{\cellX{\match{M}_L} \times \cellX{\match{M}_R}} = \overline{\cellX{\match{M}_L}} \times \overline{\cellX{\match{M}_R}}\]
using the product topology.  The inductive hypothesis tells us that $\overline{\cellX{\match{M}_L}} = \bigcup_{\mathcal{A}_L} cut(\cellX{\match{M}_L}, \mathcal{A}_L)$ where the union is taken over all subsets $\mathcal{A}_L \subseteq \match{M}_L$ and similarly for $\match{M}_R$.  By the definition in Lemma~\ref{lemma: block matrix maps work}, the map $\chi$ sends every flag in
\[cut(\cellX{\match{M}_L}, \mathcal{A}_L) \times cut(\cellX{\match{M}_R}, \mathcal{A}_R) \]
to the flag in $\cellX{\chi(\match{M}_L, \match{M}_R)}$ with entries corresponding to each arc $\alpha$ over $i$ labeled zero and with all other entries labeled as indicated by $cut(\match{M}_L, \match{A}_L)$ and $cut(\match{M}_R, \match{A}_R)$.  Since $\match{M}$ has no arcs $\alpha$ with $\init(\alpha) \leq i$ and $\term(\alpha)>i$ we conclude that the image under $\chi$ for each $\mathcal{A}_L, \mathcal{A}_R$ satisfies
\[\chi(cut(\cellX{\match{M}_L}, \mathcal{A}_L) \times cut(\cellX{\match{M}_R}, \mathcal{A}_R)) = cut(\cellX{\match{M}}, \mathcal{A}_L \cup \mathcal{A}_R).\]
Every $\mathcal{A} \subseteq \match{M}$ decomposes as the disjoint union $(\mathcal{A} \cap \match{M}_L) \cup (\mathcal{A} \cap \match{M}_R)$ since $\match{M}$ is a disjoint union of arcs that end at or before $i$ and arcs that start after $i$.  Thus we have
\[\overline{\cellX{\match{M}}} = \bigcup_{(\mathcal{A}_L, \mathcal{A}_R)} \chi(cut(\cellX{\match{M}_L}, \mathcal{A}_L) \times cut(\cellX{\match{M}_R}, \mathcal{A}_R)) = \bigcup_{\mathcal{A} \subseteq \match{M}} cut(\cellX{\match{M}}, \mathcal{A})\]
which is the inductive claim in this case.

Now suppose that $\match{M}$ contains $(1,N)$.  Consider the projection $\pi_1: \overline{\cellX{\match{M}}} \rightarrow GL_N(\mathbb{C})/P_1$.  The image of the map consists of all lines in a flag in the closure of $\cellX{\match{M}}$ and thus contains all lines in the Springer Schubert cell itself, namely all lines spanned by vectors of the form $\vec{e}_{N/2+1}+a\vec{e}_1$ for $a \in \mathbb{C}$.  This is all of $\ker X$ except for the line spanned by $\vec{e}_1$ so in particular $\pi_1(\cellX{\match{M}})$ is an open dense set in $\ker X$.  The image $\pi_1(\mathcal{S}_X)$ of the entire Springer fiber is contained in the set $\ker X$.  The map $\pi_1$ is closed so the image $\pi_1(\overline{\cellX{\match{M}}})$ is closed.   We conclude that $\pi_1(\overline{\cellX{\match{M}}})=\ker X$.  

The map $\varphi$ is a homeomorphism that commutes with $\pi_1$ by Lemma~\ref{lemma: explicit matrices for proj to v1} so we have 
\begin{equation} \label{equation: cells in image of varphi} \overline{\varphi(\mathbb{C} \times \cellX{\match{M}'})} = \varphi(\overline{\mathbb{C} \times \cellX{\match{M}'}}) = \varphi(\bigcup_{(d,\mathcal{A}')} \mathbb{C}^d \times cut(\cellX{\match{M}'}, \mathcal{A}')) = \bigcup_{(d,\mathcal{A}')} \varphi(\mathbb{C}^d \times cut(\cellX{\match{M}'}, \mathcal{A}'))\end{equation}
where the union is taken over all pairs of a subset $\mathcal{A}' \subseteq \match{M}'$ and $d \in \{0, 1\}$ with $d=0$ corresponding to the case $a=\infty$.
Suppose that $cut(\match{M}', \mathcal{A}')$ has $\{B, T\}$-word $m_2''m_3''\cdots m_{N-1}''$. By Corollary~\ref{corollary: springer cell homeomorphism for (1,N)}, the image $\varphi(\mathbb{C} \times cut(\cellX{\match{M}'}, \mathcal{A}'))$ is in the Springer Schubert cell corresponding to $Bm_2''\cdots m_{N-1}''T$ while the image $\varphi(\{\infty\} \times cut(\cellX{\match{M}'}, \mathcal{A}'))$ is in the Springer Schubert cell corresponding to $T m_2''  \cdots m_{N-1}'' B$.  By definition of cutting arcs, the latter is the $\{B, T\}$-word corresponding to $cut(\match{M}, \mathcal{A}')$  while the former is the $\{B, T\}$-word corresponding to $cut(\match{M}, \mathcal{A}' \cup \{(1,N)\})$. Thus, each expression in the union on the far right of Equation~\eqref{equation: cells in image of varphi} is contained in a different Springer Schubert cell, explicitly
\[\bigcup_{(d,\mathcal{A}')} \varphi(\mathbb{C}^d \times cut(\cellX{\match{M}'}, \mathcal{A}')) \subseteq \bigcup_{\mathcal{A} \subseteq \match{M}} \cellX{cut(\match{M}, \mathcal{A})}.\]
We now confirm that for each $(d,\mathcal{A}')$ the image $\varphi(\mathbb{C}^d \times cut(\cellX{\match{M}'},\mathcal{A}'))$ is precisely $cut(\cellX{\match{M}}, \mathcal{A})$. If $(1,N)$ is cut then every ordering of the arcs in $\mathcal{A}$ starts with $(1,N)$.  Note that $\lambda_{(1,N), \match{M}}$ is the identity on all arcs $\beta \neq (1,N)$ and $0$ on $(1,N)$.  Since $cut(\match{M}, \mathcal{A})$ starts with a $T$ and ends with a $B$, we have for all $\beta \in cut(\match{M}', \mathcal{A}')$ the equality
\[\lambda_{(1,N), \alpha_2, \ldots, \alpha_j, \match{M}}(\beta) = \lambda_{\alpha_2, \ldots, \alpha_j, \match{M}'}(\beta)\]
which proves that $\varphi(\infty \times cut(\cellX{\match{M}'}, \mathcal{A}') = cut(\cellX{\match{M}}, \mathcal{A})$ as desired.

Now suppose instead that $(1,N) \not \in \mathcal{A}$.  The function $\lambda$ is defined recursively so that $\lambda(\beta)=\alpha$ only if $\alpha=\beta$ or if $\alpha$ is the parent of an arc that is cut.  Without loss of generality, assume that $\mathcal{A}$ is ordered so that $\alpha_1, \alpha_2, \ldots, \alpha_s$ are all of the arcs in $\mathcal{A}$ with parent $(1,N)$, or equivalently all the arcs in $\mathcal{A}'$ with no parent. By construction we have 
\[cut(\match{M}, \mathcal{A}) = cut(\match{M}', \mathcal{A}') \cup \{(1, \init(\alpha_1)), (\term(\alpha_s), N)\}.\]
There may be arcs in $\{1, 2, \ldots, \init(\alpha_1)\}$ that are cut, but they must all be nested \emph{under} the arc with parent $(1,N)$ so the arc $(1,\init(\alpha_1))$ survives to $cut(\match{M}, \mathcal{A})$ and similarly for all of the $(\term(\alpha_i), \init(\alpha_{i+1}))$ as well as $(\term(\alpha_j), N)$. A parallel argument also gives 
\[\lambda_{\alpha_1, \ldots, \alpha_j, \match{M}}(\beta) = \lambda_{\alpha_1, \ldots, \alpha_j, \match{M}'}(\beta)\]
for all $\beta \in cut(\match{M}',\mathcal{A}')$ except the arcs $(\term(\alpha_1), \init(\alpha_2)), \ldots, (\term(\alpha_{s-1}), \init(\alpha_s))$.  Those arcs have a parent in $\match{M}$ but not $\match{M}'$ so for those $\beta$ we have
\[\lambda_{\alpha_1, \ldots, \alpha_j, \match{M}}(\beta) = (1,N) \hspace{0.25in} \textup{ but } \hspace{0.25in} \lambda_{\alpha_1, \ldots, \alpha_j, \match{M}'}(\beta) = 0\]
just as we have $\lambda_{\alpha_1, \ldots, \alpha_j, \match{M}}(\beta)=(1,N)$ for $\beta = (1, \init(\alpha_1)), (\term(\alpha_s), N)$.  

By the previous paragraph, if $\beta$ is an arc in $cut(\match{M}',\mathcal{A}')$ other than $(\term(\alpha_i), \init(\alpha_{i+1}))$ for some $i=1, \ldots, s-1$ then $\lambda_{\alpha_1, \ldots, \alpha_s, \match{M}}(\beta)=\lambda_{\alpha_1, \ldots, \alpha_s, \match{M}'}(\beta)$.  The additional arcs in $cut(\match{M}, \mathcal{A})$ all have corresponding entry in $cut(\cellX{\match{M}}, \mathcal{A})$ given by $v_{(1,N)}$. These occur as the (new) lowest nonzero entry in columns with pivot in the last $N/2$ rows in the intervals 
\[[1,\init(\alpha_1)], [\term(\alpha_1),\init(\alpha_2)],\ldots,[\term(\alpha_{s-1}),\init(\alpha_s)], [\term(\alpha_s),N]\]
By construction, each of these intervals (as well as the intervals between them) has the same number of $B$s as $T$s.  Within these intervals, if a column has pivot in the bottom $N/2$ rows then its lowest nonzero entry in the top $N/2$ rows is $v_{(1,N)}$. So if $g$ is in canonical form within $cut(\match{M}',\mathcal{A}')$ and we set $a = v_{(1,N)}$ then $\varphi(a,g)$ agrees with the corresponding matrix in $cut(\match{M},\mathcal{A})$ in these columns.  

Outside of these intervals, no column has entry $v_{(1,N)}$. Suppose $i$ in one of the intervals $[\init(\alpha_j)+1,\term(\alpha_j)-1]$ corresponds to a $B$ in the $\{B, T\}$-word for $cut(\match{M}', \mathcal{A}')$ but not for $\match{M}'$. Then $i$ ends an arc $\alpha$ that is in $\match{M}'$ and in fact in $\mathcal{A}'$.  The $\{B,T\}$-word for $cut(\match{M}',\mathcal{A}')$ has the same number of $B$s as $T$s before $\init(\alpha_j)$ and in the intervals $[\init(\alpha_j)+1,\init(\alpha)-1]$ and $[\init(\alpha)+1,i-1]$ because it differs from the word for $\match{M}'$ only by flipping endpoints of arcs in $\mathcal{A}'$. Both $\init(\alpha_j)$ and $\init(\alpha)$ are $T$ in these words.  So the $i^{th}$ column for $cut(\match{M}',\mathcal{A}')$ is preceded by $\ell+2$ columns with pivot in the top half and has pivot $\vec{e}_{N/2+\ell}$.  The map $\varphi$ adds $a\vec{e}_{\ell+2}$ to this column vector; writing the corresponding flag in canonical form gives column vector $\vec{e}_{N/2+\ell+2}$ as desired, since in $cut(\match{M}, \mathcal{A})$ the integer $i$ must either not be on an arc or start an arc labeled $0$ by the map $\lambda$.  

Otherwise $i$ corresponds to a $B$ in both $cut(\match{M}',\mathcal{A}')$ and $\match{M}'$ which means $i$ starts an arc $\alpha \in \match{M}'$ that is not in $\mathcal{A}'$.  The definition of cutting arcs implies that $i$ starts an arc $\beta \in cut(\match{M}',\mathcal{A}')$ though the arc $\beta$ may end before $\term(\alpha)$.  Moreover, by construction of the map $\lambda$ we know that $\lambda_{\alpha_1,\ldots, \alpha_s,\match{M}'}(\beta) = \alpha$. If $\gamma \in cut(\match{M}',\mathcal{A}')$ is nested above $\beta$ then by definition of cutting arcs in standard noncrossing matchings, either $\init(\gamma)$ starts an arc in $\match{M}'$ that is nested above $\init(\alpha)$ or $\init(\gamma)$ ends an arc in $\mathcal{A}' \subseteq \match{M}'$ to the left of $\alpha$.  Furthermore we know $\term(\gamma) = \term(\gamma')$ for some arc $\gamma' \in \match{M}'$ that ends at or after $\term(\gamma)$.  If $\gamma'$ starts before $\beta$ then  $\lambda_{\alpha_1, \ldots, \alpha_s,\match{M}'}(\gamma) = \gamma'$ and $\gamma'$ is an ancestor of $i$ in $\match{M}'$ and hence of $\alpha$.  Otherwise  $\gamma'$ starts after $\beta$ and so $\lambda_{\alpha_1, \ldots, \alpha_s,\match{M}'}(\gamma) = 0$. So the ancestors of $\beta \in cut(\match{M}',\mathcal{A}')$ are labeled by a  subset of the ancestors of $\alpha \in \match{M}'$ without repetition.  This is in fact a proper subset because $\alpha_j \in \mathcal{A}'$ and no cut arc can be one of these labels by Corollary~\ref{corollary: top to bottom gives unnesting}. Any ancestor of $\alpha$ that is cut contributes a $T$ to the interval $[1,i-1]$ in $cut(\match{M}',\mathcal{A}')$ that is $B$ in $\match{M}'$.  Every arc that ends before $\alpha$ in $\match{M}'$ contributes a $B$ and a $T$ to the interval $[1,i-1]$ in both $cut(\match{M}',\mathcal{A}')$ and $\match{M}'$ regardless of whether or not it's cut.  Also $\match{M}'$ is a perfect matching so every arc either ends before $i$ or is an ancestor of $i$.  So the $i^{th}$ column of each matrix in $cut(\match{M}',\mathcal{A}')$ has the form $\vec{e}_{N/2+j_1}+\sum_{\ell'} v_{\gamma_{\ell'}} \vec{e}_{j_2+\ell'}$ where $j_1$ is the number of $B$s before $i$ in $cut(\match{M}',\mathcal{A}')$, $j_2$ is the number of $T$s before $i$ in $cut(\match{M}',\mathcal{A}')$, and the sum is taken over ancestors $\beta_{\ell'}$ of $\alpha$ that are not in $\mathcal{A}$.  Since every ancestor of $\alpha$ that is in $\mathcal{A}$ contributes an extra $T$ before $i$ in $cut(\match{M}',\mathcal{A}')$ we conclude that the lowest nonzero nonpivot term is  $v_{\beta''} \vec{e}_{\ell''}$ where $\ell''$ is the number of $B$s before $i$ in $\match{M}'$ and $v_{\beta''}$ is not identically zero.  Thus applying $\varphi$ adds $a$ to either one of the rows for which a pivot appears earlier (in which case we can put the matrix in canonical form as we argued previously) or to one of these $v_{\beta''}$ (in which case the expression $v_{\beta''}+a$ is independent of all other $v_{\gamma}, a$ as all $v_{\beta''}, v_{\gamma}, a$ are free in $\mathbb{C}$).  

In all cases, we have shown that the image of $\varphi$ on the cut arcs $cut(\match{M}', \mathcal{A}')$ is precisely $cut(\match{M},\mathcal{A})$.  By induction, the claim holds.
\end{proof}


\def\cprime{$'$}

\end{document}